\documentclass{article}
\usepackage{CJK,CJKnumb,CJKulem,times,dsfont,ifthen,mathrsfs,latexsym,amsfonts,color}
\usepackage{amsmath,amsthm,makeidx,fontenc,amssymb,bm,graphicx,psfrag,listings,curves,extarrows}
\usepackage[colorinlistoftodos]{todonotes}
\usepackage[latin1]{inputenc}
\makeindex
\newtheorem{theorem}{Theorem}[section]
\newtheorem{definition}[theorem]{Definition}
\newtheorem{lemma}[theorem]{Lemma}
\newtheorem{corollary}[theorem]{Corollary}
\newtheorem{proposition}[theorem]{Proposition}
\newtheorem{remark}[theorem]{Remark}

\newcommand{\R}{\mathbb R}

\newcommand{\bt}{\begin{theorem}}
\newcommand{\et}{\end{theorem}}
\newcommand{\bl}{\begin{lemma}}
\newcommand{\el}{\end{lemma}}
\newcommand{\bd}{\begin{definition}}
\newcommand{\ed}{\end{definition}}
\newcommand{\bc}{\begin{corollary}}
\newcommand{\ec}{\end{corollary}}
\newcommand{\bp}{\begin{proof}}
\newcommand{\ep}{\end{proof}}
\newcommand{\bx}{\begin{example}}
\newcommand{\ex}{\end{example}}
\newcommand{\bi}{\begin{exercise}}
\newcommand{\ei}{\end{exercise}}
\newcommand{\bo}{\begin{proposition}}
\newcommand{\eo}{\end{proposition}}
\newcommand{\br}{\begin{remark}}
\newcommand{\er}{\end{remark}}
\newcommand{\beq}{\begin{equation}}
\newcommand{\eeq}{\end{equation}}
\newcommand{\ba}{\begin{align}}
\newcommand{\ea}{\end{align}}
\newcommand{\bn}{\begin{enumerate}}
\newcommand{\en}{\end{enumerate}}
\newcommand{\bg}{\begin{align*}}
\newcommand{\bcs}{\begin{cases}}
\newcommand{\ecs}{\end{cases}}

\newcommand{\bean}{\begin{eqnarray*}}
\newcommand{\eean}{\end{eqnarray*}}

\newtheorem{Thm}{Theorem}[section]
\newtheorem{Lem}[Thm]{Lemma}
\newtheorem{Cor}[Thm]{Corollary}

\theoremstyle{definition}

\theoremstyle{remark}

\def\R{\mathbb{R}}

\def\bd{\mathrm{bd}\,}

\numberwithin{equation}{section}
\begin{document}
\begin{CJK*}{GBK}{song}
\title{\bf  Positive normalized solutions to nonlinear elliptic systems in $\R^4$ with critical Sobolev exponent\thanks{Luo and Yang are supported by NSFC (No. 11901147) and the Fundamental Research Funds for the Central Universities of China (No. JZ2020HGTB0030). Zou is supported by NSFC (No. 20171301826, 20181301532).}}

\date{}
\author{
{\bf        Xiao Luo$^1$,\;\; Xiaolong Yang$^2$,\;\; Wenming Zou$^3$}\\
\footnotesize \it 1. School of Mathematics, Hefei University of Technology, Hefei, 230009, P. R. China.\\
\footnotesize \it 2. School of Mathematics and Statistics, Central China Normal University, Wuhan, 430079, P. R. China. \\
\footnotesize \it 3. Department of Mathematical Sciences, Tsinghua University, Beijing, 100084, P. R. China.}

\maketitle

\begin{center}
\begin{minipage}{120mm}
\begin{center}{\bf Abstract}\end{center}
In this paper, we consider the existence and asymptotic behavior on mass of the positive solutions to the following system:
\begin{equation}\label{eqA0.1}\nonumber
\begin{cases}
-\Delta u+\lambda_1u=\mu_1u^3+\alpha_1|u|^{p-2}u+\beta v^2u\quad&\hbox{in}~\R^4,\\
-\Delta v+\lambda_2v=\mu_2v^3+\alpha_2|v|^{p-2}v+\beta u^2v\quad&\hbox{in}~\R^4,\\
\end{cases}
\end{equation}
under the mass constraint $$\int_{\R^4}u^2=a_1^2\quad\text{and}\quad\int_{\R^4}v^2=a_2^2,$$
where $a_1,a_2$ are prescribed, $\mu_1,\mu_2,\beta>0$; $\alpha_1,\alpha_2\in \R$, $p\!\in\! (2,4)$ and $\lambda_1,\lambda_2\!\in\!\R$ appear as Lagrange multipliers. Firstly, we establish  a non-existence result for the repulsive interaction case, i.e.,  $\alpha_i<0(i=1,2)$. Then  turning to  the case of $\alpha_i>0 (i=1,2)$, if $2<p<3$, we show that the problem admits a ground state and an excited state, which are characterized respectively by a local minimizer and a mountain-pass critical point of the corresponding energy functional.  Moreover, we give a precise asymptotic behavior of these two solutions as $(a_1,a_2)\to (0,0)$ and $a_1\sim a_2$. This seems to be the first contribution regarding the multiplicity as well as the synchronized mass collapse behavior of  the normalized solutions to Schr\"{o}dinger systems with Sobolev critical exponent. When $3\leq p<4$, we prove an  existence as well as non-existence ($p=3$) results  of  the ground states, which are characterized by constrained mountain-pass critical points of the corresponding energy functional. Furthermore, precise asymptotic behaviors of the ground states are obtained when the masses of whose two components vanish and cluster to a upper bound (or infinity), respectively.

\vskip0.23in



\vskip0.23in

\end{minipage}
\end{center}
\vskip0.26in
\newpage
%%%

\section{Introduction}

\setcounter{equation}{0}
In this paper, for prescribed $a_1,a_2>0$, we study the existence of solutions $(\lambda_1,\lambda_2,u,v)\in \R^2\times H^1(\R^4,\R^2)$ satisfying
\begin{equation}\label{eq1.1}
\begin{cases}
-\Delta u+\lambda_1u=\mu_1u^3+\alpha_1|u|^{p-2}u+\beta v^2u,\\
-\Delta v+\lambda_2v=\mu_2v^3+\alpha_2|v|^{p-2}v+\beta u^{2}v,\\
\end{cases}
\end{equation}
and
\begin{equation}\label{eq1.11}
\int_{\R^4}u^2=a_1^2,\quad\int_{\R^4}v^2=a_2^2,
\end{equation}
where $\mu_1,\mu_2, \beta>0$ and $\alpha_1,\alpha_2\in \R$.

\vskip0.1in

Equation \eqref{eq1.1} is closely related to the
following focusing time-dependent nonlinear Schr\"{o}dinger system
\begin{equation}\label{eq1.2}
\begin{cases}
-i\partial_t\Phi_1=\Delta \Phi_1+\mu_1|\Phi_1|^2\Phi_1+\alpha_1|\Phi_1|^{p-2}\Phi_1+\beta |\Phi_2|^2\Phi_1\quad&\hbox{in}~\R^4,\\
-i\partial_t\Phi_2=\Delta \Phi_2+\mu_2|\Phi_2|^2\Phi_2+\alpha_2|\Phi_2|^{p-2}\Phi_2+\beta |\Phi_1|^{2}\Phi_2\quad&\hbox{in}~\R^4.\\
%\int_{\R^4}|\Psi_1(t,x)|^2=a_1^2,\quad\text{and}\quad\int_{\R^4}|\Psi_2(t,x)|^2=a_2^2.
\end{cases}
\end{equation}
Indeed, setting $\Phi_1(t,x)=e^{i\lambda_1t}u(x)$ and $\Phi_2(t,x)=e^{i\lambda_2t}v(x)$, if $(u,v)$ is a nontrivial solution of \eqref{eq1.1}, then $(\Phi_1,\Phi_2)$ is a solitary wave solution of \eqref{eq1.2}. Physically, system \eqref{eq1.2} has the nature of conservation of mass, i.e.,  the $L^2$-norms $\|\Phi_1(t,\cdot)\|^2_2$ and $\|\Phi_2(t,\cdot)\|^2_2$ of solutions are independent of $t\in \R$. Problems with prescribed masses
(the former constraint) appear in nonlinear optics, where the mass represents the power supply, and in the theory of Bose-Einstein condensates, where it represents the total number of particles (see \cite{AA,EGB,FD,HME,TV}). In particular, the positive sign of $\mu_1,\mu_2, \beta$ stays for attractive interaction, while the negative sign stays for repulsive interaction. %In this paper, we study the case of positive parameters $\mu_1,\mu_2, \beta>0$.

\vskip0.1in

When $\alpha_1=\alpha_2=0$, system \eqref{eq1.1} becomes the following so-called Gross-Pitaevskii equation
\begin{equation}\label{eq1.21}
\begin{cases}
-\Delta u+\lambda_1u=\mu_1u^3+\beta v^2u\quad&\hbox{in}~\R^N,\\
-\Delta v+\lambda_2v=\mu_2v^3+\beta u^2v\quad&\hbox{in}~\R^N,\\
\end{cases}
\end{equation}
which governs various physical phenomena, for instance binary mixtures of Bose-Einstein condensates, or the propagation of mutually incoherent wave packets in nonlinear optics, see \cite{HME,TV}.
Recently, normalized solutions to system \eqref{eq1.21} have attracted much attention of researchers e.g. \cite{BJS,BS,BS-CV,BXZ}. Bartsch, Jeanjean, Soave \cite{BJS} and Bartsch, Soave \cite{BS} dealt with $\beta > 0$ and $\beta < 0$ respectively, they found solutions to \eqref{eq1.21}-\eqref{eq1.11} for specific range of $\beta$ depending on $a_1, a_2$. In addition, Bartsch, Zhong and Zou \cite{BXZ} studied  \eqref{eq1.11}-\eqref{eq1.21}   for $\beta > 0$ belonging to much better ranges independent of the masses $a_1$ and $a_2$, they adopted a new approach based on bifurcation theory and the continuation method. Bartsch and Soave \cite{BS-CV} proved the existence of infinitely many normalized solutions to \eqref{eq1.21} with $\mu_{1}=\mu_{2}>0$ and $\beta\leq -\mu_{1}$ by using a suitable minimax argument. For normalized solutions of nonlinear elliptic systems with more general nonlinearities, we refer the reader to \cite{BarJean,GouJean,LiZou,MS}.

\vskip2mm
For $\beta=0$, system \eqref{eq1.1} reduces to the following single energy (Sobolev) critical Schr\"{o}dinger equation
\begin{equation}\label{eq2}
-\Delta u=\lambda u+\mu |u|^{q-2}u+|u|^{2^*-2}u, \ \  \ u\in H^{1}(\R^N),
\end{equation}
where $\mu\in\R$ and $2<q<2^*=\frac{2N}{(N-2)^+}$. If $\mu>0$, it was proved in \cite{Soave2} that if the product of $\mu$ and the mass is well controlled, there is a normalized ground state to \eqref{eq2}. In addition, Soave \cite{Soave2} raised an open question on how to obtain  the second normalized solution to \eqref{eq2} since the associated energy functional constrained on the $L^2$-sphere admits a convex-concave geometry. It is worth pointing out that, Jeanjean, Le \cite{JTT} solved this open question for $N \ge 4$. Recently, Wei and Wu \cite{WW} proved the existence of normalized solutions of mountain-pass type of \eqref{eq2} for $N\ge 3$ and $2<q<2+\frac{4}{N}$. Moreover, in \cite{WW}, they also got the precisely asymptotic behaviors of ground states and mountain-pass solutions as $\mu\to 0$ or  $\mu$ goes to its upper bound. These studies answer some open questions proposed by Soave in \cite{Soave2}.

\vskip3mm

Elliptic systems with the doubly critical exponent have  been investigated by many authors (c.f. \cite{CZ,CZ2,LL,PPW,PYF,WuZ,TS} and the references therein). However, much less is known when the $L^2$-norms $\|u\|_2,\|v\|_2$ are prescribed. Recently, Li, Zou in \cite{LiZou} and Mederski, Schino in \cite{MS} considered normalized solutions to coupled Schr\"{o}dinger systems with a wide class of nonlinearities including the Sobolev critical terms. Here, we have to emphasize that the system considered by Li and Zou in \cite{LiZou} contains only a single critical nonlinearity, and coupled terms of the systems considered by Mederski and Schino in \cite{MS} must be Sobolev subcritical.
%Here, we also want to mention that Li and Zou \cite{LiZou} and Mederski et al. \cite{MS} consider the semilinear elliptic systems with a single critical nonlinearities.
In this paper, we study the existence of normalized solutions to problem \eqref{eq1.1} with doubly Sobolev critical exponent involving coupling terms.
%This study seems to be the first contribution to the existence of normalized solutions for two-component elliptic system in $\R^4$ with the critical Sobolev exponent.

\vskip1mm
One can get solutions to \eqref{eq1.1}-\eqref{eq1.11} by looking for critical points $(u,v)\in H^1(\R^4,\R^2)$ of the energy functional
\begin{equation}\label{eq1.4}
I(u,v)=\int_{\R^4}\frac{1}{2}\big(|\nabla u|^2+|\nabla v^2|\big)-\frac{1}{4}\big(\mu_1|u|^4+\mu_2|v|^4\big)-\frac{\alpha_1}{p}|u|^p-\frac{\alpha_2}{p}|v|^p-\frac{\beta}{2}|u|^{2}|v|^{2}
\end{equation}
constrained on the sphere $S(a_1)\times S(a_2)$, where
\begin{equation}\label{eq1.5}
S(a):=\big\{u\in H^1(\R^4):\int_{\R^4}|u|^2=a^2\big\},
\end{equation}
and $\lambda_1,\lambda_2$ in \eqref{eq1.1} appear as the Lagrange multipliers with respect to the mass constraints.

\vskip3mm
If $\alpha_i<0(i=1,2)$, we have the following non-existence result.
\begin{Thm}\label{th1.3}
Let $\mu_i,a_i,\beta>0(i=1,2)$, $2<p<4$. If $\alpha_i<0(i=1,2)$, then problem \eqref{eq1.1}-\eqref{eq1.11} has no positive solution $(u,v)\in H^1(\R^4,\R^2)$ and no non-trivial radial solution for $p>\frac{8}{3}$.
\end{Thm}
%For $u\in H^1(\R^4)$ and $2<r<4$, the Gagliardo-Nirenberg inequality is
%\begin{equation}\label{eq1.6}
%\|u\|_r\leq  C(r)\|\nabla u\|_2^{\gamma_r}\|u\|_2^{1-\gamma_r},\quad \forall u\in H^1(\R^4),
%\end{equation}
%where $\gamma_r=\frac{2(r-2)}{r}$.

Next, we consider the case $\alpha_i>0(i=1,2)$. By the $L^2$-norm preserving dilations $s\star u(x)=e^{2s}u(e^sx)$ with $s\!>\!0$, it is easy to know that
$\bar{p}:=3$ is the $L^2$-critical exponent of \eqref{eq1.1}.
We say that $(\tilde{u},\tilde{v})$ is a normalized ground state of \eqref{eq1.1}-\eqref{eq1.11} if $dI|_{S(a_1)\times S(a_2)}(\tilde{u},\tilde{v})=0$ and
$$
I(\tilde{u},\tilde{v})=\inf \big\{I(u,v): ~~dI|_{S(a_1)\times S(a_2)}(u,v)=0\ \text{and}\ (u,v)\in S(a_1)\times S(a_2)\big\}.
$$
If $(u,v) \!\in\! H^1(\mathbb{R}^{4},\R^2)$ is a weak solution of \eqref{eq1.1}, then the Pohozaev identity holds true:
\begin{equation}\label{c31}
\begin{aligned}
&P_{a_1,a_2}(u,v) =:||\nabla u||^2_2+\|\nabla v\|^2_2-\mu_1\|u\|^{4}_{4}-\mu_2\|v\|^{4}_{4}\\
&\quad\quad\quad \quad\quad\quad  -\gamma_p\alpha_1\|u\|^p_p-\gamma_p\alpha_2\|v\|^p_p-
2\beta\|uv\|^2_2
\end{aligned}
\end{equation}
for $\gamma_p=\frac{2(p-2)}{p}$. Since $2<p<4$, we have $\inf\limits_{ S(a_1)\times S(a_2)}I(u,v)=-\infty$. In spirit of Bartsch and Soave in \cite{BS}, we introduce the following Pohozaev set:
\begin{equation}\label{c1}
\mathcal{P}_{a_1,a_2}=\big\{(u,v) \in S(a_1)\times S(a_2) : P_{a_1,a_2}(u,v)=0\big\}.
\end{equation}
For $u \in S(a)$ and $s\in \mathbb{R}$, %let
%\begin{equation*} \label{c3}
% s\star u(x) :=e^{2s} u\left(e^s x\right),~~~~\mbox{for}~~~~\mbox{a.e.}~~~~x \in \mathbb{R}^{4}.
%\end{equation*}
it follows that $s\star u \in S(a)$.
Define $s\star (u,v)=(s\star u, s\star v)$. The Pohozaev set $\mathcal{P}_{a_1,a_2}$  is quite related to the fiber maps
\begin{equation} \label{c4}
\begin{aligned}
&\Psi_{(u,v)}(s)\\
& :=I(s\star (u,v))\\
&=\frac{e^{2s}}{2}\int_{\R^4}\big(|\nabla u|^2+|\nabla v|^2\big)-\frac{e^{4s}}{4}\int_{\R^4}\big(\mu_1|u|^4+\mu_2|v|^4+2\beta|u|^{2}|v|^{2}\big)\\
&\quad\ \ -\frac{e^{p\gamma_ps}}{p}\alpha_1\int_{\R^4}|u|^p-\frac{e^{p\gamma_ps}}{p}\alpha_2\int_{\R^4}|v|^p.\\
\end{aligned}
\end{equation}
For $(u,v) \in S(a_1)\times S(a_2)$ and $s \in (0,\infty)$, we have
\begin{equation} \label{c2}
\Psi'_{u,v}(s)=P_{a_1,a_2}\big(s \star (u,v)\big),
\end{equation}
where $P_{a_1,a_2}$ is defined by \eqref{c31}.
We shall see that the critical points  of $\Psi_{u,v}(s)$ allow to project a function on $\mathcal{ P}_{a_1,a_2}$. Thus,  the monotonicity and convexity properties of $\Psi_{u,v}(s)$ strongly affect  the structure of $\mathcal{ P}_{a_1,a_2}$. %(and in turn the geometry of $I|_{S(a_1)\times S(a_2)}$), and also have a strong impact on properties of equation (\ref{eq1.1}).
Therefore, $\mathcal{P}_{a_1, a_2}$ can be divided into the disjoint union $\mathcal{ P}_{a_1,a_2}=\mathcal{ P}_{a_1,a_2}^+\cup \mathcal{ P}_{a_1,a_2}^0\cup \mathcal{ P}_{a_1,a_2}^-$, where
\begin{equation}\label{c41}
\begin{aligned}
	\mathcal{ P}_{a_1,a_2}^+&:=\big\{(u,v)\in \mathcal{ P}_{a_1,a_2} : \Psi_{(u,v)}''(0)>0\big\},\\
	\mathcal{ P}_{a_1,a_2}^0&:=\big\{(u,v)\in \mathcal{ P}_{a_1,a_2} : \Psi_{(u,v)}''(0)=0\big\},\\
	\mathcal{ P}_{a_1,a_2}^-&:=\big\{(u,v)\in \mathcal{ P}_{a_1,a_2} : \Psi_{(u,v)}''(0)<0\big\}.
\end{aligned}
\end{equation}
Let
\begin{equation*}
m^{\pm}(a_1,a_2)=\inf_{(u,v)\in \mathcal{P}^{\pm}_{a_1,a_2}} I(u,v).
\end{equation*}

In the spirit of Soave \cite{Soave2} and Jeanjean, Le \cite{JTT},  we see that the presence of the mass subcritical terms $\alpha_1|u|^{p-2} u, \alpha_2|v|^{p-2}v$   causes  a convex-concave geometry of $I|_{S(a_1)\times S(a_2)}$ if $\alpha_1,\alpha_2>0$ and $a_1,a_2>0$ are small.
In view of this, if $\alpha_1,\alpha_2>0$, it is natural to introduce a suitable local minimization problem:
\begin{equation}\label{b11}
\begin{aligned}
&m^+(a_1,a_2):=\inf_{(u,v)\in\mathcal{P}^+_{a_1,a_2}}I(u,v)\\
&=\inf_{(u,v)\in\mathcal{P}_{a_1,a_2}}I(u,v)=\inf_{(u,v)\in V(a_1,a_2)}I
(u,v)<0,
\end{aligned}
\end{equation}
where
$$
V(a_1,a_2):=\big\{(u,v)\in S(a_1)\times S(a_2): \big[\|\nabla u\|^2_2+\|\nabla v\|^2_2\big]^{\frac{1}{2}}<\rho_0 \big\},
$$
see Section 3 for details. %Moreover, we define
%\begin{equation*}
%m^-(a_1,a_2):=\inf_{(u,v)\in \mathcal{P}^-_{a_1,a_2}}I(u,v)>0.
%\end{equation*}

\vskip3mm
Our main results are as follows. If $\alpha_i,\mu_i,\beta>0(i=1,2)$ and $2<p<4$, we first consider the existence and asymptotic behaviour of ground states to \eqref{eq1.1}-\eqref{eq1.11}. Let $w_p$ be the unique positive solution of %$-\Delta w+w-|w|^{p-2}w=0$.
\begin{equation}\label{g2}
-\Delta w + w = |w|^{p-2}w \ \ \text{in}\ \R^4.
\end{equation}
For $2<p<4$, the Gagliardo-Nirenberg inequality (see \cite{WeM}) is
\begin{equation}\label{b2}
\|u\|_p\leq  \mathcal{C}_p\|\nabla u\|_2^{\gamma_p}\|u\|_2^{1-\gamma_p},\quad \forall u\in H^1(\R^4),
\end{equation}
where $\mathcal{C}_p>0$ is a constant and $\gamma_p=\frac{2(p-2)}{p}$.
Define
\begin{equation}\label{r1}
T(a_1,a_2):=\alpha_1a_1^{4-p}+\alpha_2a_2^{4-p}\quad
\end{equation}
and
\begin{equation}\label{r1-zwm}
\gamma_1\!:=\!\frac{p}{2(4-p)|\mathcal{C}_p|^{p}}
\Big[\!\frac{2(3\!-\!p)S^{2}_{\mu_1,\mu_2, \beta}}{4\!-\!p}\Big]^{3-p},
\end{equation}
where
%Denote
\begin{equation}\label{f4}
\mathcal{S}_{\mu_1,\mu_2, \beta}:=\inf _{(u, v) \in\left[D^{1,2}\left(\mathbb{R}^{4}\right)\right]^{2} \backslash\{(0,0)\}} \frac{\int_{\mathbb{R}^{4}}\left(|\nabla u|^{2}+|\nabla v|^{2}\right)}{\left[\int_{\mathbb{R}^{4}}\left(\mu_1|u|^{4}+\mu_2|v|^{4}+2\beta|u|^{2}|v|^{2}\right)\right]^{\frac{1}{2}}}
.
\end{equation}
%$\mathcal{C}_p$ and $S_{\mu_1,\mu_2, \beta}$ are given in \eqref{b2} and \eqref{f4}, respectively.

To state  the next theorem, we  use the notation  $c\sim d$ which means that $C'd\le c \le Cd$ and $c\lesssim d$ means $c\le C d$.  We have
\begin{Thm}\label{th1.1}
Let $\mu_i,\alpha_i,a_i>0(i=1,2)$, $\beta>0$ and $2<p<3$, and $T(a_1,a_2)\le\gamma_1$.
\begin{enumerate}
\item
Then \eqref{eq1.1}-\eqref{eq1.11} has a positive ground state  $(u_{a_1},v_{a_2})$, which is a local minimizer of $I$ on $V(a_1,a_2)$.
\item For any ground state $(u_{a_1},v_{a_2})\in V(a_1,a_2)$, we have
\begin{equation*}
\Big(\big(\frac{L_1}{\alpha_1}\big)^{-\frac{1}{p-2}}u_{a_1}(L_1^{-\frac{1}{2}}x), \big(\frac{L_2}{\alpha_2}\big)^{-\frac{1}{p-2}}v_{a_2}(L_2^{-\frac{1}{2}}x)\Big)\to (w_p,w_p)
 \ \  \text{in}\ H^{1}(\R^4,\R^2),
\end{equation*}
as $(a_1,a_2)\to (0,0)$ and $a_1\sim a_2$, where $L_1=\big(\frac{a^2_1}{\|w_{p}\|^2_2}\alpha_1^{\frac{2}{p-2}} \big)^{\frac{p-2}{6-2p}}$ and $L_2=\big(\frac{a^2_2}{\|w_{p}\|^2_2}\alpha_2^{\frac{2}{p-2}} \big)^{\frac{p-2}{6-2p}}$.
\end{enumerate}
\end{Thm}

\vskip3mm

Next, for $2<p<3$, we study the existence and asymptotic behaviour of   the  second solution of \eqref{eq1.1}-\eqref{eq1.11}.
Denote
$$\mathcal{S}=\inf _{u \in {D}^{1,2}(\mathbb{R}^{4})\setminus \{0\} }    \frac{\left\| \nabla u\right\|_2^{2}}{||u||_{4}^{2}}.$$
From \cite{GT}, we know that $\mathcal{S}$ is attained by the Aubin-Talanti bubbles
\begin{align}\label{L1}
U_{\varepsilon}(x) :=\frac{2\sqrt{2}\varepsilon}{\varepsilon^2+|x|^2}, \quad \varepsilon > 0, \quad x\in\R^4.
\end{align}
Then $U_\varepsilon$ satisfies $-\Delta u=u^3$ and $\int_{\R^4}|\nabla U_\varepsilon|^2=\int_{\R^4}|U_\varepsilon|^4=\mathcal{S}^2$.
On the other hand, if $0<\beta<\min\{\mu_1,\mu_2\}$ or $\beta>\max\{\mu_1,\mu_2\}$ (see Lemma \ref{lem2.2}),  the pair $$\Big(\sqrt{\frac{\beta-\mu_2}{\beta^2-\mu_1\mu_2}}U_{\varepsilon},\sqrt{\frac{\beta-\mu_1}{\beta^2-\mu_1\mu_2}}U_{\varepsilon}\Big)$$  is the least energy solutions to the following elliptic system:
\begin{equation}\label{a1}
\begin{cases}
-\Delta u=\mu_1u^3+\beta v^{2} u, & x \in \mathbb{R}^{4}, \\
-\Delta v=\mu_2v^3+\beta u^{2} v, & x \in \mathbb{R}^{4}, \\
u,v \in D^{1,2}(\R^4),
\end{cases}
\end{equation}
we refer to Section 2 for more  details.

\vskip4mm

\begin{Thm}\label{th1.2}
Let $\mu_i,\alpha_i,a_i>0(i=1,2)$, $\beta\in\big(0, \min\{\mu_1,\mu_2\}\big)\cup \big(\max\{\mu_1,\mu_2\},\infty\big)$, $2<p<3$, and $T(a_1,a_2)\le\gamma_{1}$.
\begin{enumerate}
\item  Then there exists a positive mountain pass type solution
$(u^-_{a_1},v^-_{a_2})$ to \eqref{eq1.1}-\eqref{eq1.11}, and

\begin{equation*}
I(u^-_{a_1},v^-_{a_2})=m^-(a_1,a_2)\to\frac{2\beta-\mu_1-\mu_2}{4(\beta^2-\mu_1\mu_2)}\mathcal{S}^2\ \ \text{as} \ \ (a_1,a_2)\to (0,0).
\end{equation*}
\item  Moreover, there exists $\varepsilon_{a_1,a_2}>0$ such that
\begin{equation*}
\big(\varepsilon_{a_1,a_2} u^-_{a_1}(\varepsilon_{a_1,a_2} x),\varepsilon_{a_1,a_2} v^-_{a_2}(\varepsilon_{a_1,a_2} x)\big)
\end{equation*}

\begin{equation*}\to \Big(\sqrt{\frac{\beta-\mu_2}{\beta^2-\mu_1\mu_2}}U_{\varepsilon_0},\sqrt{\frac{\beta-\mu_1}{\beta^2-\mu_1\mu_2}}U_{\varepsilon_0}\Big)
\end{equation*}
in $D^{1,2}(\R^4,\R^2)$, for some $\varepsilon_0>0$ as $(a_1,a_2)\to (0,0)$ and $a_1\sim a_2$ up to a subsequence.
\end{enumerate}
\end{Thm}

\vskip1mm

\begin{remark}
Theorems \ref{th1.1} and \ref{th1.2} indicate that when the lower power perturbation terms are mass subcritical with  a product of the perturbation coefficients and masses being controlled from above, problem \eqref{eq1.1} possesses at least two normalized solutions, one ground state and one excited state (whose energy is strictly larger than that of ground state). Furthermore, by making appropriate scalings, two components of the ground state both converge to the unique positive solution of $-\Delta w + w = |w|^{p-2}w$; two components of the excited state both converge to the Aubin-Talanti bubble in related Sobolev spaces, as the masses of two components vanish at the same rate. As far as we know, this is the first result on the  multiplicity as well as detailed synchronized mass collapse behavior of normalized solutions to Schr\"{o}dinger systems with Sobolev critical exponent. The condition $T(a_1,a_2)\le\gamma_1$ in Theorems \ref{th1.1} and \ref{th1.2} not only ensures that the corresponding energy functional $I$ admits a convex-concave geometry, but also guarantees that the Pohozaev manifold $\mathcal{P}_{a_1,a_2}$ is a natural constraint, on which the critical points of $I$ are indeed normalized solutions to problem \eqref{eq1.1}. The value range of coupling coefficient $\beta$ in Theorem \ref{th1.2} (1) is used to prove
$$m^-(a_1,a_2)<\frac{k_1+k_2}{4}\mathcal{S}^2+m^+(a_1,a_2)$$
 i.e.,
%that the corresponding limiting system \eqref{a1} has a least energy solution $\big(\frac{\beta-\mu_2}{\beta^2-\mu_1\mu_2}U_\varepsilon,
%\frac{\beta-\mu_1}{\beta^2-\mu_1\mu_2}U_\varepsilon\big)$. Furthermore, using the radial superposition of this least energy solution to \eqref{a1} and the ground state of \eqref{eq1.1} obtained in Theorem \ref{th1.1}, we can prove that
the mountain pass energy level is less than the usual critical threshold plus the ground state energy, and thus ensures the compactness of PS sequence and a mountain pass type solution follows.
The condition $\beta\in \big(0,\min\{\mu_1,\mu_2\}\big)\cup\big(\max\{\mu_1,\mu_2\},\infty\big)$ in Theorem \ref{th1.2} (2) also guarantees that the corresponding limiting system \eqref{a1} has a unique ground state solution up to translation and dilation. With the help of this uniqueness result, precise synchronized mass collapse behavior of the mountain pass type solution can be obtained.
\end{remark}

\vskip3mm

If $p=3$, $w_3$ is the unique positive solution of \eqref{g2}, then we get the following existence and non-existence results.

\begin{Thm}\label{th1.4}
Let $\mu_i,\alpha_i,a_i>0(i=1,2)$, $p=3$, there exists $\beta_0>0$ such that  if $\beta>\beta_0$ and $\beta\in \big(0,\min\{\mu_1,\mu_2\}\big)\cup\big(\max\{\mu_1,\mu_2\},\infty\big)$, then
the following conclusions hold.
\begin{enumerate}
\item
If $0<\alpha_ia_i<\|w_3\|_2(i=1,2)$,
then \eqref{eq1.1}-\eqref{eq1.11} has a positive ground state solution. However, for $\alpha_1a_1\ge\|w_3\|_2$ and $\alpha_2a_2\ge\|w_3\|_2$, then \eqref{eq1.1}-\eqref{eq1.11} has no ground states.

\item
If $0<\alpha_ia_i<\|w_3\|_2(i=1,2)$,  the for any ground state $(u^-_{a_1},v^-_{a_2})$ of \eqref{eq1.1}-\eqref{eq1.11}, there exist $\nu_1,\nu_2>0$ such that
\begin{equation*}
\Big(\big(\frac{a_1}{\|w_3\|_2} \big)r^{2}_1u^-_{a_1}(\frac{a_1}{\|w_3\|_2}r_1x), \big(\frac{a_2}{\|w_3\|_2} \big)r^{2}_2v^-_{a_2}(\frac{a_2}{\|w_3\|_2}r_2x)\Big)
\end{equation*}
\begin{equation*}
\to \big(\nu_1w_3(\sqrt{\nu_1}x), \nu_2w_3(\sqrt{\nu_2}x)\big),
\end{equation*}
in $H^1(\R^4,\R^2)$ as $(a_1,a_2)\to (\frac{\|w_3\|_2}{\alpha_1},\frac{\|w_3\|_2}{\alpha_2})$ and $\Big(1-\frac{\alpha_1a_1}{\|w_3\|_2}\Big)\sim \Big(1-\frac{\alpha_2a_2}{\|w_3\|_2}\Big)$, up to a subsequence, where
$r_1=\big(1-\frac{\alpha_1a_1}{\|w_3\|_2}\big)^{-\frac{1}{2}}$ and $r_2=\big(1-\frac{\alpha_2a_2}{\|w_3\|_2}\big)^{-\frac{1}{2}}$.
\item
Moreover, there exists $\varepsilon_{a_1,a_2}>0$ such that
\begin{equation*}
\big(\varepsilon_{a_1,a_2} u^-_{a_1}(\varepsilon_{a_1,a_2} x),\varepsilon_{a_1,a_2} v^-_{a_2}(\varepsilon_{a_1,a_2} x)\big)
\end{equation*}
\begin{equation*}
\to \Big(\sqrt{\frac{\beta-\mu_2}{\beta^2-\mu_1\mu_2}}U_{\varepsilon_0},\sqrt{\frac{\beta-\mu_1}{\beta^2-\mu_1\mu_2}}U_{\varepsilon_0}\Big)
\end{equation*}
in $D^{1,2}(\R^4,\R^2)$, for some $\varepsilon_0>0$ as $(a_1,a_2)\to (0,0)$ and $a_1\sim a_2$, up to a subsequence.
\end{enumerate}
\end{Thm}

\begin{remark}
For single Sobolev critical nonlinear Schr\"{o}dinger equation \eqref{eq2} with $q=2+\frac{4}{N}$, there exists a critical mass  $a^*>0$ such that if $0<a<a^*$ then \eqref{eq2} has normalized ground state solutions, while if $a^*\le a$, \eqref{eq2} has no normalized ground states (see \cite{LXF,Soave2,WW}). However, for coupled system \eqref{eq1.1}, such a mass threshold is not obtained in Theorem \ref{th1.4}. The main reason   is that there is a competitive relationship between the two components of system \eqref{eq1.1}.
In proving the non-existence result of Theorem \ref{th1.4}, the monotonicity approach developed in \cite{WW} cannot be used directly, because the monotonicity of $m^-(a_1,a_2)$ with $a_1$ and $a_2$ seems difficult to prove. Motivated by \cite{LXF,WW}, we choose appropriate functions to prove that $m^-(a_1,a_2)=0$ when $\alpha_1a_1\ge\|w_3\|_2$ and $\alpha_2a_2\ge \|w_3\|_2$, and then by contradiction \eqref{eq1.1}-\eqref{eq1.11} has no normalized ground state solution. In addition, similar as in the proof of Theorem \ref{th1.2} (2), the precisely asymptotic behaviors of the normalized ground states solutions are also proved.
\end{remark}

\vskip0.12in

If $3<p<4$,  then $\mathcal{P}_{a_1,a_2}=\mathcal{P}^-_{a_1,a_2}$, and we could find at least one normalized solution to \eqref{eq1.1}-\eqref{eq1.11} on $\mathcal{P}^-_{a_1,a_2}$.
%we prove the existence of ground states for any $a_1,a_2>0$.
Moreover, we give the asymptotic behavior of the normalized ground state  solutions as $(a_1,a_2)\to (+\infty,+\infty)$ and $a_1\sim a_2$.

\begin{Thm}\label{th1.5}
Let $\mu_i,\alpha_i,a_i>0(i=1,2)$, $3<p<4$, and there exists $\beta_1>0$ such that when  $\beta>\beta_1$ and $\beta\in \big(0,\min\{\mu_1,\mu_2\}\big)\cup\big(\max\{\mu_1,\mu_2\},\infty\big)$,  then we have  the following results.
\begin{enumerate}
\item
The  problem \eqref{eq1.1}-\eqref{eq1.11} has a positive ground state solution.
\item
For $(a_1,a_2)\to (+\infty,+\infty)$ and $a_1\sim a_2$, then for any ground state $(u^-_{a_1},v^-_{a_2})\in S(a_1)\times S(a_2)$, we have
\begin{equation*}
\Big(\big(\frac{L_1}{\alpha_1}\big)^{-\frac{1}{p-2}}u^-_{a_1}(L_1^{-\frac{1}{2}}x), \big(\frac{L_2}{\alpha_2}\big)^{-\frac{1}{p-2}}v^-_{a_2}(L_2^{-\frac{1}{2}}x)\Big)\to (w_p,w_p)
 \ \  \text{in}\ H^{1}(\R^4,\R^2),
\end{equation*}
where $L_1=\big(\frac{a^2_1}{\|w_{p}\|^2_2}\alpha_1^{\frac{2}{p-2}} \big)^{-\frac{p-2}{2p-6}}$ and $L_2=\big(\frac{a^2_2}{\|w_{p}\|^2_2}\alpha_2^{\frac{2}{p-2}} \big)^{-\frac{p-2}{2p-6}}$.

\item
In addition, there exists $\varepsilon_{a_1,a_2}>0$ such that
\begin{equation*}
\Big(\varepsilon_{a_1,a_2} u^-_{a_1}(\varepsilon_{a_1,a_2} x),\;\; \varepsilon_{a_1,a_2} v^-_{a_2}(\varepsilon_{a_1,a_2} x)\Big)
\end{equation*}
\begin{equation*}
\to \Big(\sqrt{\frac{\beta-\mu_2}{\beta^2-\mu_1\mu_2}}U_{\varepsilon_0},\sqrt{\frac{\beta-\mu_1}{\beta^2-\mu_1\mu_2}}U_{\varepsilon_0}\Big) ,
\end{equation*}
in $D^{1,2}(\R^4,\R^2)$, for some $\varepsilon_0>0$ as $(a_1,a_2)\to (0,0)$ and $a_1\sim a_2$,  up to a subsequence.
\end{enumerate}
\end{Thm}

\begin{remark}
Theorems \ref{th1.4} and \ref{th1.5} indicate that problem \eqref{eq1.1}  possesses at least one normalized ground state solution once the system has a lower order perturbation of
the  mass critical term. Furthermore, a precise asymptotic behavior of the normalized ground state is obtained when the masses of the  two components cluster  either to an upper bound or  to infinity at the same rate. The conditions $\beta>\beta_0>0$ and $\beta>\beta_1>0$ in these two theorems are used to prove that the ground state energy $m^-(a_1,a_2)$ of coupled system \eqref{eq1.1} is less than that of the  related   single equations. This revision of the energy threshold excludes the semitrivial solutions and guarantees the compactness of minimizing sequence at the energy level $m^-(a_1,a_2)$.
\end{remark}

\vskip1mm

\begin{remark}
The arguments in this paper can be applied to
\begin{equation*}
\begin{cases}
-\Delta u_i+\lambda_iu_i=\alpha_i|u_i|^{p_i-2}u_i+\sum^{k}_{j=1}\beta_{i,j} u^2_ju_i\quad&\hbox{in}~\R^4,\\
u_i\in H^1(\R^4),\qquad i\in \{1,\ldots,k\}, \\
\end{cases}
\end{equation*}
under the constraint $\int_{\R^4}u^2_i=a_i^2$ where $a_i(i\in \{1,\ldots,k\})$ is prescribed, $\beta_{i,j}=\beta_{j,i}\in\R$, $\beta_{i,j}>0(\forall i,j)$, $2<p_i<4$, $\alpha_i\in \R$;   where  $\lambda_i\!\in\!\R$ appear as the  Lagrange multipliers.   In addition, $p_i $ is divided into three cases
\begin{equation*}
(H_0)\ 2<p_i<3; \quad\quad (H_1)\ p_i=3;  \quad\quad  (H_2) \ 3<p_i<4,
\end{equation*}
for all $i\in \{1,\ldots,k\}$.
\end{remark}
\vskip0.12in

This paper is organized as follows, in Section 2, we give some preliminary results. The proofs  of Theorems \ref{th1.3} and \ref{th1.1} are given in Section 3. In Sections 4, 5 and 6, we prove Theorems \ref{th1.2}, \ref{th1.4} and \ref{th1.5}, respectively.

\vskip3mm

\noindent\textbf{Notations:}
$L^{p}=L^{p}(\mathbb{R}^{N})(1<p\leq\infty)$ is the Lebesgue space with the standard norm $||u||_{p}=\big(\int_{{\mathbb{R}^N}} {{|u(x)|}^p dx}\big)^{{1}/{p}}$. We use $``\rightarrow"$ and $``\rightharpoonup"$ to denote the strong and weak convergence in the related function spaces respectively. $C$, $C'$ and $C_{i}$ will denote positive constants.

\section{Preliminaries}

In this section, we give some preliminaries. We have the Sobolev inequality
	$$\mathcal{S}\|u\|_{4}^2\leq \|\nabla u\|_2^2,\quad \forall u\in D^{1,2}(\R^4),$$
where $D^{1,2}(\R^4)$ is the completion of $C_c^{\infty}(\R^4)$ with respect to the norm $||u||_{D^{1,2}}:=\|\nabla u\|_2$.
If $p=4$, $\gamma_{4}=1$, then $\mathcal{S}= |\mathcal{C}_4|^{-2}$.
For nonlinear Schr\"odinger equation with critical nonlinearities
\begin{equation}\label{a2}
\begin{cases}
-\Delta u+\lambda u=\alpha |u|^{p-2}u+ \mu u^3, \ \  \ u\in H^{1}(\R^4),\\
\int_{\R^4}|u|^2=a^2,
\end{cases}
\end{equation}
where $\mu,\alpha>0$, $\lambda\in \R$.
The solutions of \eqref{a2} are critical points of the functional
\begin{equation*}
\mathcal{A}_{p,\mu,\alpha}(u):=\frac{1}{2}\int_{\mathbb{R}^{4}}|\nabla u|^{2}-\frac{\alpha}{p}\int_{\mathbb{R}^{4}}|u|^{p}-\frac{\mu}{4}\int_{\mathbb{R}^{4}}|u|^{4}
\end{equation*}
on $S(a)$. Moreover, we introduce the Pohozaev manifold for the scalar  equation:
\begin{equation}\label{a21}
\mathcal{T}_{a,p,\mu,\alpha}:=\Big\{u\in S(a): \int_{\mathbb{R}^{4}}|\nabla u|^{2}=\mu\int_{\mathbb{R}^{4}}|u|^{4}+\alpha \gamma_{p} \int_{\mathbb{R}^{4}}|u|^{p} \Big\}.
\end{equation}
The $L^2$-Pohozaev manifold $\mathcal{T}_{a,p,\mu,\alpha}$ is  closely  related to the fibering maps
\begin{equation*}
\overline{\Psi}_{u}(s)=\mathcal{A}_{p,\mu,\alpha}(s\star u)=\frac{e^{2s}}{2}\int_{\R^4}|\nabla u|^2-\frac{\alpha e^{p\gamma_ps}}{p}\int_{\R^4}|u|^p-\frac{\mu e^{4s}}{4}\int_{\R^4}|u|^4.
\end{equation*}
Similar to the partition of $L^2$-Pohozaev manifold $\mathcal{P}_{a_1,a_2}$, $\mathcal{T}_{a,p,\mu,\alpha}$ can also be divided into the disjoint union $\mathcal{T}_{a,p,\mu,\alpha}=\mathcal{T}^{+}_{a,p,\mu,\alpha}\cup \mathcal{T}^{0}_{a,p,\mu,\alpha}\cup \mathcal{T}^{-}_{a,p,\mu,\alpha}$, where
\begin{equation*}
\begin{aligned}
	\mathcal{T}^{+}_{a,p,\mu,\alpha}&:=\big\{(u,v)\in \mathcal{T}_{a,p,\mu,\alpha} : \overline{\Psi}_{u}''(0)>0\big\},\\
	\mathcal{T}^{0}_{a,p,\mu,\alpha}&:=\big\{(u,v)\in \mathcal{T}_{a,p,\mu,\alpha} : \overline{\Psi}_{u}''(0)=0\big\},\\
	\mathcal{T}^{-}_{a,p,\mu,\alpha}&:=\big\{(u,v)\in \mathcal{T}_{a,p,\mu,\alpha} : \overline{\Psi}_{u}''(0)<0\big\}.
\end{aligned}
\end{equation*}
Define %the constraint minimization problem
\begin{equation}\label{c121}
m^{\pm}_{p,\mu,\alpha}(a):=\inf_{u\in \mathcal{T}^{\pm}_{a,p,\mu,\alpha} }\mathcal{A}_{p,\mu,\alpha}(u).
\end{equation}
Let
\begin{equation}\label{c12}
m^{\pm}(a_1,0)=m^{\pm}_{p,\mu_1,\alpha_1}(a_1)\quad\text{and}\quad m^{\pm}(0,a_2)=m^{\pm}_{p,\mu_2,\alpha_2}(a_2).
\end{equation}

\vskip3mm
We have the following already known results.
\begin{Lem}\label{lem2.1}(\cite[Theorem 1.1]{Soave2};\cite[Theorem 1.1]{WW};\cite[Theorem 1.4]{LXF})
Let $2<p<3$ and $\mu,a,\alpha>0$ in \eqref{a2}, there exists $\tau(\alpha,p)>0$ such that if $\mu a^{p-p\gamma_p}\le\tau(\alpha,p)$ then $ m^{\pm}_{p,\mu,\alpha}(a)=\inf\limits_{u\in \mathcal{T}^{\pm}_{a,p,\mu,\alpha} }\mathcal{A}_{p,\mu,\alpha}(u)<0$ and it can be attained by some $u^{\pm}_{a,\mu,\alpha}$ which is real valued, positive, radially symmetric and radially decreasing. Moreover, $m^{\pm}_{p,\mu,\alpha}(a)$ is strictly decreasing with respect to $a>0$.
\end{Lem}

\begin{Lem}\label{lem2.12}(\cite[Corollary B.1]{LiZou})
Suppose $(u,v)\in H^1(\R^4,\R^2)$ is a nonnegative solution of \eqref{eq1.1}-\eqref{eq1.11} with $2<p<4$, then $(u,v)$ is a smooth solution.
\end{Lem}

\begin{Lem}\label{lem2.11}(\cite[Lemma 2.3]{LiZou})
Suppose $\alpha_1,\alpha_2>0$ and $(u,v)\in H^1(\R^4,\R^2)$ is a nonnegative solution of \eqref{eq1.1}-\eqref{eq1.11} with $2<p<4$, then $u>0$ implies $\lambda_1>0$; $v>0$ implies $\lambda_2>0$.
\end{Lem}

\begin{Lem}\label{lem2.13}(\cite[Lemma 2.4]{GouJean})             %(\cite{CZ2,GouJean})
If $(u_n,v_n)\rightharpoonup (u,v)$ in $H^1(\R^4,\R^2)$, then up to a subsequence,
\begin{equation*}
\int_{\R^4}|u_n|^2|v_n|^2-|u_n-u|^2|v_n-v|^2=\int_{\R^4}|u|^2|v|^2+o(1).
\end{equation*}
\end{Lem}

\begin{Lem}\label{lem2.14}(\cite[Chapter 3]{LiM})
Let $u^*,v^*$ denote the symmetric decreasing rearrangement of $u,v$ respectively, and $u,v\in H^1(\R^4)$. Then we have
\begin{equation*}
\|\nabla u^*\|_2\le \|\nabla u\|_2, \ \ \|u^*\|_p=\|u\|_p\ \ \text{and}\ \ \int_{\R^4}|u^*|^2|v^*|^2\ge \int_{\R^4}|u|^2|v|^2.
\end{equation*}
\end{Lem}

\vskip1mm

From \cite{CZ,PPW,PYF,WWY}, %for $u, v \in D^{1,2}(\R^4)$, if $0<\beta<\min\{\mu_1,\mu_2\}$ or $\beta>\max\{\mu_1,\mu_2\}$,
we get the least energy solutions to \eqref{a1}.
Note that \eqref{a1} has semi-trivial solutions $(\mu^{-\frac{1}{2}}_1 U_{\varepsilon}, 0)$ and $(0, \mu^{-\frac{1}{2}}_2 U_{\varepsilon})$. Here, we are only interested in  the nontrivial solutions of \eqref{a1}, this is $(\sqrt{k_1}U_{\varepsilon},\sqrt{k_2}U_{\varepsilon})$, where $U_\varepsilon$ is defined in \eqref{L1}, $k_1=\frac{\beta-\mu_2}{\beta^2-\mu_1\mu_2}$ and $k_2=\frac{\beta-\mu_1}{\beta^2-\mu_1\mu_2}$. The main results in this aspect are summarized below. Next, we turn to the related limiting elliptic system \eqref{a1}.

\begin{Lem}\label{lem2.2}(\cite[Theorems 1.5 and 4.1]{CZ}). Let  $\mu_1,\mu_2>0$.
\begin{enumerate}
\item If $\beta<0$, then  $\mathcal{S}_{\mu_1,\mu_2,\beta}$ is not attained.
\item If $0<\beta<\min\{\mu_1,\mu_2\}$ or $\beta>\max\{\mu_1,\mu_2\}$,   %then  $(\sqrt{k_1}U_{\varepsilon},\sqrt{k_2}U_{\varepsilon})$ is a positive least energy solution of \eqref{a1}
%for some $\varepsilon>0$. %and $z\in \R^4$, If $\beta>\max\{\mu_1,\mu_2\}$,
then any positive least energy solution $(u,v)$ of \eqref{a1} must be of the form
$$(u,v)=\big(\sqrt{k_1}U_{\varepsilon}, \sqrt{k_2}U_{\varepsilon}\big).$$

\item If $\beta\in [\min\{\mu_1,\mu_2\},\max\{\mu_1,\mu_2\}]$ and $\mu_1\neq\mu_2$, \eqref{a1} does not have a nontrivial nonnegative solution.
\end{enumerate}
\end{Lem}

\vskip0.3in

Solutions to \eqref{a1} correspond to critical points of the functional
\begin{equation*}
J(u,v)=\int_{\R^4}\frac{1}{2}\big(|\nabla u|^2+|\nabla v^2|\big)-\frac{1}{4}\big(\mu_1|u|^4+\mu_2|v|^4+2\beta|u|^{2}|v|^{2}\big),\ \ u,v\in D^{1,2}(\R^4).
\end{equation*}
From Lemma 2.1 of \cite{PPW} or \cite{CZ}, if $0<\beta<\min\{\mu_1,\mu_2\}$ or $\beta>\max\{\mu_1,\mu_2\}$, $(u_0,v_0)$ is a least energy solution of \eqref{a1}, we have
\begin{equation*}
J(u_0,v_0)=\frac{k_1+k_2}{4}\mathcal{S}^2 \ \  \text{and} \ \ \mathcal{S}_{\mu_1,\mu_2, \beta}=\sqrt{k_1+k_2}\mathcal{S}.
\end{equation*}

We observe that the functional $I(u,v)$ defined in \eqref{eq1.4} is well defined.
Throughout this paper, we denote
\begin{equation}\label{b3}
    D_1=\frac{1}{4\mathcal{S}^2_{\mu_1,\mu_2, \beta}}, \ \quad
	D_2=\frac{\alpha_1}{p} |\mathcal{C}_p|^{p}a_1^{4-p}\ \ \text{and} \ \
	D_3=\frac{\alpha_2}{p} |\mathcal{C}_p|^{p}a_2^{4-p}.
\end{equation}
Then we have
\begin{equation}\label{b4}
\begin{aligned}
\quad\ \frac{1}{4}\int_{\R^4}\mu_1|u|^4+\mu_2|v|^4+2\beta|u|^{2}|v|^{2}\leq D_1\Big(\int_{\R^4}|\nabla u|^2+|\nabla v^2|\Big)^2,
\end{aligned}
\end{equation}

%\begin{equation*}
%\int_{\R^4}\frac{1}{4}(|u|^4+|v|^4)+\frac{\beta}{2}|u|^{2}|v|^{2}\le D_1\big(\|\nabla u\|^2_2+\|\nabla v\|^2_2\big)^2,
%\end{equation*}

\begin{equation}\label{b5}
	\frac{1}{p}\int_{\R^4}\alpha_1|u|^p\leq D_2\|\nabla u\|_2^{2(p-2)}\quad \text{and}\quad \frac{1}{p}\int_{\R^4}\alpha_2|v|^p\leq D_3\|\nabla v\|_2^{2(p-2)}.
\end{equation}
Substituting \eqref{b4}-\eqref{b5} into \eqref{eq1.4}, we obtain
\begin{equation}\label{b7}
\begin{aligned}
	I(u,v)&\geq\frac{1}{2}\big(\|\nabla u\|^2_2+\|\nabla v\|^2_2\big)-D_1\big(\|\nabla u\|^2_2+\|\nabla v\|^2_2\big)^{2}\\
	&\quad -D_2\|\nabla u\|_2^{2(p-2)}-D_3\|\nabla v\|_2^{2(p-2)}\\
		&\geq h\big([\|\nabla u\|^2_2+\|\nabla v\|^2_2]^{\frac{1}{2}}\big),
\end{aligned}
\end{equation}
where $h(\rho):(0,+\infty)\to\R$ is defined by
\begin{equation}\label{n11}
	h(\rho)=\frac{1}{2}\rho^2-D_1\rho^{4}-(D_2+D_3)\rho^{2(p-2)}, \ \ \forall \rho\ge0.
\end{equation}

Recall the definitions of $T(a_1,a_2)$ and $\gamma_1$ in \eqref{r1}.
\begin{Lem}\label{lem2.3}
Let $\mu_i,\alpha_i,a_i(i=1,2)$, $\beta>\!0$, $2\!<\!p\!<\!3$ and $\rho_0\!=\big[\!\frac{3-p}{2(4-p)D_1}\big]^{\frac{1}{2}}$.

\indent $(i)$ If $T(a_1,a_2)\!<\!\gamma_1$,  then $h(\rho)$ has a local minimum at negative level and a global maximum at positive level.  Moreover, there exist $R_0=R_0(a_1,a_2)$ and $R_1=R_1(a_1,a_2)$ such that
$$0< R_0 <{\rho}_0< R_1, \quad  h(R_0)=0=h(R_1);   \;\; \quad  h(\rho)\!>\!0 \Leftrightarrow \rho\!\in\! (R_0,R_1).$$
\indent $(ii)$ If $T(a_1,a_2)\!=\!\gamma_1$,  then $h(\rho)$ has a local minimum at negative level and a global maximum at level $0$. Moreover, we have
$h(\rho^*)=0$ and $h(\rho)\!<\!0 \Leftrightarrow \rho\!\in\!(0, \rho^*) \cup (\rho^*,+\infty)$.
\end{Lem}
\begin{proof}
$(i)$ We first prove that $h$ has exactly two critical points. In fact,
$$
h'(\rho)=0 \Longleftrightarrow \phi_1(\rho)=2(p-2)(D_2+D_3), ~~~~\mbox{with}~~~~\phi_1(\rho)=\rho^{6-2p}-4D_1\rho^{8-2p}.
$$
We have, $\phi_1(\rho) \nearrow$ on $[0,\bar{\rho})$, $\searrow$ on $(\bar{\rho},+\infty)$, where $\bar{\rho}\!=\big[\!\frac{3-p}{4(4-p)D_1}\big]^{\frac{1}{2}}$.
Since $2\!<\!p\!<\!3$, we get $\max\limits_{\rho\geq0}\phi_1(\rho)\!=\!\phi_1(\bar{\rho})
\!=\!\frac{1}{4-p}\bar{\rho}^{6\!-\!2p}
\!>\!2(p-2)(D_2+D_3)$ if and only if
\begin{equation*} \label{c^*}
\alpha_1a_1^{4-p}+\alpha_2a_2^{4-p}\!<\!\gamma_0\!:=\!\frac{p}{2(p-2)(4-p)|\mathcal{C}_p|^{p}}
\Big[\!\frac{(3\!-\!p)S^{2}_{\mu_1,\mu_2, \beta}}{4\!-\!p}\Big]^{3-p}.
\end{equation*}
By $\phi_1(0^+)\!=\!0^+$ and $\phi_1(+\infty)\!=\!-\infty$, we see  that $h$ has exactly two critical points if $T(a_1,a_2):=\alpha_1a_1^{4-p}+\alpha_2a_2^{4-p}\!<\!\gamma_0$.

\vskip3mm
Notice that
$$
h(\rho)>0 \Longleftrightarrow \psi_1(\rho)>D_2+D_3, ~~~~\mbox{with}~~~~\psi_1(\rho)=\frac{1}{2}\rho^{6-2p}-D_1\rho^{8-2p}.
$$
It is not difficult to check that $\psi_1(\rho) \nearrow$ on $[0,\rho_0)$, $\searrow$ on $(\rho_0,+\infty)$, where $\rho_0\!=\big[\!\frac{3-p}{2(4-p)D_1}\big]^{\frac{1}{2}}$.
We have $\max\limits\limits_{\rho\geq0}\psi_1(\rho)\!=\!\psi_1(\rho_0)
\!=\!\frac{1}{2(4-p)}\rho_0^{6\!-\!2p}
\!>\!D_2+D_3$ provided
\begin{equation*}
\alpha_1a_1^{4-p}+\alpha_2a_2^{4-p}\!<\!\gamma_1\!:=\!\frac{p}{2(4-p)|\mathcal{C}_p|^{p}}
\Big[\!\frac{2(3\!-\!p)S^{2}_{\mu_1,\mu_2, \beta}}{4\!-\!p}\Big]^{3-p}.
\end{equation*}
We have $h(\rho)\!>\!0$ on an open interval $(R_0,R_1)$ if and only if $T(a_1,a_2):=\alpha_1a_1^{4-p}+\alpha_2a_2^{4-p}\!<\!\gamma_1$. It follows from Lemma 5.2 of \cite {Soave1} that
$$(p-2)2^{3-p}<1\Leftrightarrow \gamma_1<\gamma_0.$$
Combined with $h(0^+)=0^{-}$ and $h(+\infty)=-\infty$, we see that $h$ has a local minimum point at negative level in $(0,R_0)$ and a global maximum point at positive level in $(R_0,R_1)$.
%Finally, the fact that $h(\rho)\!\leq\!g(\rho)=\frac{1}{2}\rho^{2}-\frac{\mu}{q}C^{q}_{opt} a^{q(1-\gamma_q)}\rho^{q\gamma_q}$ leads to $R_0\!>\!\tilde{\rho}_a$, where $\tilde{\rho}_a:=\Big[\frac{2\mu}{q}C^{q}_{opt} a^{q(1-\gamma_q)}\Big]^{\frac{1}{2-q\gamma_q}}$.

\vskip1mm
$(ii)$  Similar to the proof of (i), we deduce that
\begin{equation*}
R_0=\bar{\rho}=\rho_{0}=R_1,\quad \psi_{1}(\rho_{0})=D_2+D_3,\quad \phi_{1}(\bar{\rho})>2(p-2)(D_2+D_3).
\end{equation*}
\end{proof}

\begin{Lem}\label{lem2.4} Let $\mu_i,\alpha_i,a_i(i=1,2)$, $\beta>\!0$ and $2\!<\!p\!<\!3$, $T(a_1,a_2)\le\gamma_{1}$.   Then, for fixed $(u,v)\in S(a_1)\times S(a_2)$,     $\Psi_{(u,v)}(s)$ has exactly two critical points in $(0,\infty)$.
\end{Lem}
\begin{proof}
By \eqref{c4}, we have
\begin{equation}\label{c11}
\begin{aligned}
\Psi_{(u,v)}(s) &=\frac{e^{2s}}{2}\int_{\R^4}(|\nabla u|^2+|\nabla v|^2)-\frac{e^{4s}}{4}\int_{\R^4}\big(\mu_1|u|^4+\mu_2|v|^4+2\beta|u|^{2}|v|^{2}\big)\\
&\quad-\frac{e^{p\gamma_ps}}{p}\int_{\R^4}\alpha_1|u|^p-\frac{e^{p\gamma_ps}}{p}\int_{\R^4}\alpha_2|v|^p.
\end{aligned}
\end{equation}
Let $t=e^s$, $A_1=\int_{\R^4}(|\nabla u|^2+|\nabla v|^2)$, $A_2=\frac{1}{4}\int_{\R^4}\big(\mu_1|u|^4+\mu_2|v|^4+2\beta|u|^{2}|v|^{2}\big)$, $A_3=\frac{1}{p}\int_{\R^4}\alpha_1|u|^p$ and $A_4=\frac{1}{p}\int_{\R^4}\alpha_2|v|^p$. Moreover, from \eqref{b4}-\eqref{b5}, we have $A_2\le D_1|A_1|^2$, $A_3\le D_2 |A_1|^{p-2}$ and $A_4\le D_3 |A_1|^{p-2}$. Then we can rewrite \eqref{c11} as
\begin{equation*}
\begin{aligned}
l(t)&=\frac{A_1}{2}t^2-A_2t^4-(A_3+A_4)t^{2(p-2)}\\
    &\ge \frac{1}{2}(A^{\frac{1}{2}}_1t)^2-D_1(A^{\frac{1}{2}}_1t)^4-(D_2+D_3)(A^{\frac{1}{2}}_1t)^{2(p-2)}.
\end{aligned}
\end{equation*}
Similar to the calculation in Lemma \ref{lem2.3}, we can prove $\Psi_{(u,v)}(s)$ has exactly two critical points in $(0,\infty)$.
\end{proof}

\vskip0.3in

\section{Proof of Theorems 1.1 and 1.2 }
In this section, we first study the case $\alpha_i<0(i=1,2)$, $2<p<4$, $\mu_i>0(i=1,2)$ and $\beta>0$.
\vskip1mm

\noindent \textbf{Proof of Theorem 1.1.}
If $\alpha_i<0(i=1,2)$, let $(u,v)$ be a constraint critical point of $I$ on $S(a_1)\times S(a_2)$ and $u,v>0$. Therefore, $(u,v)$ solves \eqref{eq1.1}-\eqref{eq1.11} for some $\lambda_1,\lambda_2\in \R$. %It follows from Lemma \ref{lem2.2} that $\lambda_1,\lambda_2>0$. From the maximum principle, we conclude that $u,v>0$.
By the Pohozaev identity $P_{a_1,a_2}(u,v)=0$, we have
\begin{equation*}\label{y1}
\lambda_1\|u\|^2_2+\lambda_2\|v\|^2_2=\alpha_1(1-\gamma_p)\|u\|^p_p+\alpha_2(1-\gamma_p)\|v\|^p_p.
\end{equation*}
% Thus, we have a contradiction with \eqref{y1}.
Then one of the $\lambda_1,\lambda_2$ is negative. Without  loss of generality, we assume $\lambda_1<0$, $\alpha_1<0$. It follows from Lemma \ref{lem2.12} that $(u,v)$ is smooth, and is in $L^{\infty}(\R^4)\times L^{\infty}(\R^4)$; thus, $|\Delta u|, |\Delta v|\in L^{\infty}(\R^4)$ as well, and stand gradient estimates for the Poisson equation (see formula (3,15) in \cite{Tru}) imply that $|\nabla u|, |\nabla v|\in L^{\infty}(\R^4)$. Combining the fact that $u,v\in L^2(\R^4)$, we get $u(x), v(x)\to 0$ as $|x|\to \infty$. Thus, we have
\begin{equation*}
-\Delta u=\big(-\lambda_1+\alpha_1|u|^{p-2}+\mu_1u^2+\beta v^2\big)u\ge \frac{-\lambda_1}{2} u>0 \ \
\end{equation*}
for $|x|>R_0$, with $R_0>0$ large enough, and then $u$ is superharmonic at infinity. From the Hadamard three spheres theorem \cite[Chapter 2]{ProW}, this implies that the function $m(r):= \min\limits_{|x|=r} u(x)$ satisfies
\begin{equation*}
m(r) \ge \frac{m(r_1)\big(r^{-2}-r_2^{-2}\big) + m(r_2)\big( r_1^{-2}- r^{-2}\big)}{r_1^{-2}- r_2^{-2}} \qquad \forall R_0 < r_1 < r < r_2.
\end{equation*}
Since $u$ decays at infinity, we have that $m(r_2) \to 0$ as $r_2 \to +\infty$, it is not difficult to see that $r \mapsto r^{2} m(r)$ is monotone non-decreasing for $r>R_0$. Moreover, $m(r) >0$ for every $r>0$ because $u>0$ in $\R^4$. Thus,
\begin{equation*}
m(r) \ge m(R_0) R_0^{2}\cdot r^{-2} \qquad \forall r>R_0,
\end{equation*}
and we deduce that
\begin{equation*}
\begin{split}
\|u\|_{2}^2
 \ge C \int_{R_0}^{+\infty} |m(r)|^2 r^{3} dr  \ge C \int_{R_0}^{+\infty} \frac{1}{r} dr=+\infty,
\end{split}
\end{equation*}
with $C>0$.  This is a contradiction. Then we know that  \eqref{eq1.1}-\eqref{eq1.11} has no positive solution for $\alpha_i<0(i=1,2)$.
\vskip3mm

If $u\in H^1(\R^4)$ is a radial function, by \cite[Radial Lemma A.II]{BHL}, there exist $C>0$ and $R_1>0$ such that
\begin{equation*}
|u(x)|\le C |x|^{-\frac{3}{2}}\ \ \text{for}\ \ |x|\ge R_1.
\end{equation*}
If $\alpha_i<0(i=1,2)$, let $(u,v)$ be a non-trivial radial solution of \eqref{eq1.1}-\eqref{eq1.11}.
Similarly, we assume $\lambda_1<0$.
Setting $p(x)=-\alpha_1|u|^{p-2}-\mu_1u^2-\beta v^2$, then
\begin{equation}\label{z7}
-\Delta u+p(x)u= -\lambda_1 u.
\end{equation}
Since $p >\frac{8}{3}$,
\begin{equation*}
\lim_{n\to +\infty}|x||p(x)|\le \lim_{n\to +\infty}\big[C|x|^{-\frac{3p}{2}+4}+C|x|^{-2} \big]=0.
\end{equation*}
By Kato's result \cite{KT}, i.e.,  Schr\"{o}dinger operator $H = -\Delta+ p(x)$ has no positive eigenvalue with an $L^2$-eigenfunction if $p(x) = o(|x|^{-1})$, then \eqref{z7} has no solution. We  show that  \eqref{eq1.1}-\eqref{eq1.11} has no non-trivial radial solution for $p >\frac{8}{3}$ and $\alpha_i<0(i=1,2)$.
\qed

\vskip4mm
In the following, we deal with the case $2<p<3$, $\alpha_i,\mu_i,a_i>0(i=1,2)$, $\beta>0$ and prove Theorem \ref{th1.1}.
%In order to prove that $I$ has a local strict minimum at negative level, we give some lemmas.
The first step consists in showing that $\mathcal{P}_{a_1,a_2}$ is a manifold.
\begin{Lem}\label{lem3.1}     If $T(a_1,a_2)\le\gamma_{1}$, then $\mathcal{P}^0_{a_1,a_2}=\emptyset$, and the set $\mathcal{P}_{a_1,a_2}$ is a $C^1$-submanifold of codimension 1 in $S(a_1)\times S(a_2)$, hence it is a $C^1$-submanifold of codimension 3 in $H^1(\R^4,\R^2)$.
\end{Lem}
\begin{proof}
It is sufficient to prove that $\mathcal{P}^0_{a_1,a_2}=\emptyset$ holds. Indeed, if $\mathcal{P}^0_{a_1,a_2}=\emptyset$, we show that $\mathcal{P}_{a_1,a_2}$ is a $C^1$-submanifold of codimension 1 in $S(a_1)\times S(a_2)$. From \eqref{c1}, we know that $\mathcal{P}_{a_1,a_2}$ is defined by $P_{a_1,a_2}(u,v)=0$, $G_1(u)=0$ and $G_2(v)=0$, where
\begin{equation*}
G_1(u)=\|u\|^2_2-a^2_1 \quad\text{and}\quad  G_2(v)=\|v\|^2_2-a^2_2.
\end{equation*}
Since $P_{a_1,a_2}$, $G_1$ and $G_2$ are class of $C^1$, the proof is completed by showing that
$d(P_{a_1,a_2},G_1,G_2): H^1(\R^4,\R^2)\to \R^3 $ is a surjective. If this is not true, then $dP_{a_1,a_2}$ has to be linearly dependent from $dG_1$ and $dG_2$, i.e., there exist $\bar{\nu}_1,\bar{\nu}_2\in\R$ such that $(u, v)$ is a solution to
\begin{equation*}
\begin{cases}
-\Delta u+\bar{\nu}_1 u=2\mu_1u^3+\frac{p\gamma_p}{2}\alpha_1|u|^{p-2}u+2\beta v^2u\ \ &\text{in}\ \R^4,\\
-\Delta v+\bar{\nu}_2 v=2\mu_2v^3+\frac{p\gamma_p}{2}\alpha_2|v|^{p-2}v+2\beta u^2v\ \ &\text{in}\ \R^4,\\
\|u\|^2_2=a^2_1,\ \ \ \|v\|^2_2=a^2_2.
\end{cases}
\end{equation*}
By the Pohozaev identity for above equation, we get
\begin{equation*}
2(||\nabla u||^2_2+\|\nabla v\|^2_2)=4\mu_1\|u\|^{4}_{4}+4\mu_2\|v\|^{4}_{4}+p\gamma^2_p\alpha_1\|u\|^p_p+p\gamma^2_p\alpha_2\|v\|^p_p+
8\beta\|uv\|^2_2,
\end{equation*}
which contradicts with $(u,v)\in \mathcal{P}^0_{a_1,a_2}=\emptyset$.

\vskip1mm
Next, we only need to prove that $\mathcal{P}^0_{a_1,a_2}=\emptyset$ provided that $T(a_1,a_2)\le\gamma_{1}$. We argue by contradiction, suppose that there is a pair $(u,v)\in \mathcal{P}^0_{a_1,a_2}$. Let
\begin{equation*}
\begin{aligned}
E(t):&=t\Psi'_{(u,v)}(0)-\Psi''_{(u,v)}(0)\\
&=(t-2)\int_{\R^4}|\nabla u|^2+|\nabla v|^2-(t-4)\mu_1\int_{\R^4}|u|^{4}-(t-4)\mu_2\int_{\R^4}|v|^{4}\\
&\quad-(t-p\gamma_p)\gamma_p\alpha_1\int_{\R^4}|u|^p-(t-p\gamma_p)\gamma_p\alpha_2\int_{\R^4}|v|^p-2\beta(t-4)\int_{\R^4}|u|^2|v|^2\\
&=0.
\end{aligned}
\end{equation*}
%and
%\begin{equation*}
%2(||\nabla u||^2_2+\|\nabla %v\|^2_2)=4\alpha_1\|u\|^{4}_{4}+4\alpha_2\|v\|^{4}_{4}+p\gamma^2_p\mu_1\|u\|^p_p+q\gamma^2_q\mu_2\|v\|^q_q+
%8\beta\|uv\|^2_2.
%\end{equation*}
If $2<p<3$, we have $p\gamma_p<2$. It follows from $E(p\gamma_p)=0$ that
\begin{equation}\label{c5}
\begin{aligned}
(2-p\gamma_p)\int_{\R^4}|\nabla u|^2+|\nabla v|^2&=(4-p\gamma_p)\mu_1\int_{\R^4}|u|^{4}+(4-p\gamma_p)\mu_2\int_{\R^4}|v|^{4}\\
&\quad+2\beta(4-p\gamma_p)\int_{\R^4}|u|^2|v|^2.
\end{aligned}
\end{equation}
Let $\rho=\big[\|\nabla u\|^2_2+\|\nabla v\|^2_2\big]^{\frac{1}{2}}$, by \eqref{b4} and \eqref{c5}, we can infer that
$\rho\ge\big(\frac{3-p}{4(4-p)D_1}\big)^{\frac{1}{2}}$.
%$(2-q\gamma_q)\rho^2\le (4-q\gamma_q)D_1\rho^4$.
Since $E(4)=0$, we get
\begin{equation}\label{c6}
\begin{aligned}
2&=(4-p\gamma_p)\gamma_p\alpha_1\rho^{-2}\int_{\R^4}|u|^p+(4-p\gamma_p)\gamma_p\alpha_2\rho^{-2}\int_{\R^4}|v|^p\\
&\le (4-p\gamma_p)p\gamma_pD_2\rho^{2(p-2)-2}+(4-p\gamma_p)p\gamma_pD_3\rho^{2(p-2)-2}\\
&\le C'(p) \big(\alpha_1a^{4-p}_1+ \alpha_2a^{4-p}_2\big),
\end{aligned}
\end{equation}
where $C'(p)=\frac{4(4-p)(p-2)|\mathcal{C}_p|^p}{p}\Big(\frac{4-p}{(3-p)\mathcal{S}^2_{\mu_1,\mu_2,\beta}}\Big)^{3-p}$ and $p\gamma_p=2(p-2)$.
Since $T(a_1,a_2)=\alpha_1a^{4-p}_1+ \alpha_2a^{4-p}_2\le \gamma_{1}<\gamma_0$ with $\gamma_{0}=\frac{2}{C'(p)}$, then this contradicts with \eqref{c6}. Thus, $\mathcal{P}^0_{a_1,a_2}=\emptyset$.
\end{proof}

\begin{Lem}\label{lem3.2} If $T(a_1,a_2)\le\gamma_{1}$, for $(u,v)\in S({a_1})\times S({a_2})$, then the function $\Psi_{(a_1,a_2)}(s)$ has exactly two critical points $s_{(u,v)}<t_{(u,v)}\in\R$ and two zeros $c_{(u,v)}<d_{(u,v)}$ with $s_{(u,v)}<c_{(u,v)}<t_{(u,v)}<d_{(u,v)}$. Moreover,
\vskip1mm

\begin{itemize}
\item  [$(1)$]   $s_{(u,v)}\star(u,v)\in\mathcal{P}^{+}_{a_1,a_2}$ and $t_{(u,v)}\star(u,v)\in\mathcal{P}^{-}_{a_1,a_2}$. Moreover,  if $s\star(u,v)\in \mathcal{P}_{(a_1,a_2)}$, then either $s=s_{(u,v)}$   or $s=t_{(u,v)}$.
\item  [$(2)$]  $s_{(u,v)}<\log\frac{R_0}{[\|\nabla u\|^2_2+\|\nabla v\|^2_2]^{\frac{1}{2}}}$ and
\begin{equation*}
\Psi_{(u,v)}(s_{(u,v)})=\inf\Big\{\Psi_{(u,v)}(s):s\in\big(-\infty,\log\frac{R_0}{[\|\nabla u\|^2_2+\|\nabla v\|^2_2]^{\frac{1}{2}}}\big)\Big\}<0.
\end{equation*}
%$\Psi_{u}^{\mu}$ is strictly decreasing and concave on $\left(t_{u},+\infty\right)$, and $t_u<0$ implies $P_{\mu}(u)<0$.\\
\item  [$(3)$]  $I\big(t_{(u,v)}\star(u,v)\big)=\max\limits_{s\in\R}I\big(s\star(u,v)\big)>0$.
\item  [$(4)$] The maps $(u,v) \mapsto s_{(u,v)} \in \mathbb{R}$
and $(u,v) \mapsto t_{(u,v)} \in \mathbb{R}$ are of class $C^1$.
\end{itemize}

\end{Lem}
\begin{proof}
Let $(u,v)\in S(a_1)\times S(a_2)$, from \eqref{c2}, we have $s\star(u,v)\in \mathcal{P}_{(a_1,a_2)}$ if and only if $\Psi'_{(u,v)}(s)=0$. %We first show that $\Psi_{(u,v)}(s)$ has at least two critical points.
By \eqref{b7}, we get
\begin{equation*}
\Psi_{(u,v)}(s)=I\big(s\star(u,v)\big)\ge h\big(e^s[\|\nabla u\|^2_2+\|\nabla v\|^2_2]^{\frac{1}{2}}\big).
\end{equation*}
Therefore, $\Psi_{(u,v)}(s)$ is positive on $\Big(\log\frac{R_0}{[\|\nabla u\|^2_2+\|\nabla v\|^2_2]^{\frac{1}{2}}},\log\frac{R_1}{[\|\nabla u\|^2_2+\|\nabla v\|^2_2]^{\frac{1}{2}}} \Big)$ and $\Psi_{(u,v)}(-\infty)=0^-$, $\Psi_{(u,v)}(+\infty)=-\infty$, thus we can see that $\Psi_{(u,v)}(s)$ has a local minimum point $s_{(u,v)}$ at negative level in $\Big(-\infty,\log\frac{R_0}{[\|\nabla u\|^2_2+\|\nabla v\|^2_2]^{\frac{1}{2}}} \Big)$ and has a global maximum point $t_{(u,v)}$ at positive level in $\Big(\log\frac{R_0}{[\|\nabla u\|^2_2+\|\nabla v\|^2_2]^{\frac{1}{2}}},\log\frac{R_1}{[\|\nabla u\|^2_2+\|\nabla v\|^2_2]^{\frac{1}{2}}} \Big)$. It follows from Lemma \ref{lem2.4} that $\Psi_{(u,v)}(s)$ has no other critical points.

Since $\Psi''_{(u,v)}(s_{(u,v)})\ge 0$, $\Psi''_{(u,v)}(t_{(u,v)})\le 0$ and $\mathcal{P}^0_{a_1,a_2}=\emptyset$, we know that $s_{(u,v)}\in \mathcal{P}^+_{a_1,a_2}$ and $t_{(u,v)}\in \mathcal{P}^-_{a_1,a_2}$. By the monotonicity and the behavior at infinity of $\Psi_{(u,v)}(s)$, we get that $\Psi_{(u,v)}(s)$ has exactly zeros $c_{(u,v)}<d_{(u,v)}$ with $s_{(u,v)}<c_{(u,v)}<t_{(u,v)}<d_{(u,v)}$. Thus, the conclusions $(1)$-$(3)$ follows from these facts above. Finally, applying the Implicit Function Theorem to the $C^1$ function $g_{(u,v)} : \R \times S(a_1)\times S(a_2) \mapsto \R$ defined by $g_{(u,v)}(s) = \Psi_{(u,v)}'(s)$. Therefore, we have that
$(u,v) \mapsto s_{(u,v)}$ is of class $C^1$ because $g_{(u,v)}(s_{(u,v)}) =0$ and $\partial_s g_{(u,v)}(s_{(u,v)}) = \Psi_{(u,v)}''(s_{(u,v)})>0$. Similarly, we can prove that $(u,v) \mapsto t_{(u,v)} \in \mathbb{R}$ is of class $C^1$.
\end{proof}

\vskip0.3in

Let $T(a_1,a_2)>0$ and $\rho_0$ be given in Lemma \ref{lem2.3}. Note that by Lemma \ref{lem2.3}, we have that $f(a_1,a_2, \rho_0)=h(\rho_0) \ge 0$ for all $T(a_1,a_2)\le\gamma_1$.
Define
\begin{align*}
B_{\rho_0} := \big\{(u,v) \in H^{1}(\R^4,\R^2) : [\|\nabla u\|^2_2+ \|\nabla v\|^2_2 ]^{\frac{1}{2}}< \rho_0\big\},
\end{align*}
and
$$V(a_1,a_2) := \big(S(a_1)\times S(a_2)\big) \cap B_{\rho_0}.$$
Thus, we can define the local minimization problem:  for any $T(a_1,a_2)\le\gamma_1$,
\begin{equation*}
m(a_1,a_2) := \inf_{(u,v) \in V(a_1,a_2)} I(u,v).
\end{equation*}

\begin{Cor}\label{Cor3.1}
If $T(a_1,a_2)\le\gamma_{1}$, the set $\mathcal{P}^+_{a_1,a_2}$ is contained in $V(a_1,a_2)$ and
\begin{equation*}
m(a_1,a_2)=m^+(a_1,a_2)=\inf_{(u,v)\in\mathcal{P}_{a_1,a_2}}I(u,v)=\inf_{(u,v)\in\mathcal{P}^+_{a_1,a_2}}I(u,v).
\end{equation*}
\end{Cor}
\begin{proof}
For $(u,v)\in V(a_1,a_2)$, we have
\begin{equation*}
I(u,v)\ge h\big([\|\nabla u\|^2_2+\|\nabla v\|^2_2]^{\frac{1}{2}}\big)\ge \min_{\rho\in [0,\rho_0]}h(\rho)>-\infty.
\end{equation*}
Let $(u,v)\in S(a_1)\times S(a_2)$, there exists $s_0\ll -1$ such that $s_0\star (u,v)\in B_{\rho_0}$ and $I(s_0\star (u,v))<0$. Hence, we get $m(a_1,a_2)\in (-\infty, 0)$. From Lemma \ref{lem3.2}, we have $\mathcal{P}^+_{a_1,a_2}\subset V(a_1,a_2)$, and then $m(a_1,a_2)\le \inf\limits_{(u,v)\in\mathcal{P}^+_{a_1,a_2}}I(u,v)$.

\vskip0.123in

In addition, if $(u,v)\in V(a_1,a_2)$, $s_{(u,v)}\star (u,v)\in \mathcal{P}^+_{a_1,a_2}\subset V(a_1,a_2)$, we get
\begin{equation*}
I\big(s_{(u,v)}\star(u,v)\big)=\min\big\{I(s\star(u,v)): s\in \R\ \ \text{and}\ \ s\star(u,v)\in V(a_1,a_2)\big\}\le I(u,v),
\end{equation*}
and it follows that $\inf\limits_{\mathcal{P}^+_{a_1,a_2}}I(u,v)\le m(a_1,a_2)$. By Lemma \ref{lem3.2}, $I(u,v)>0$ on $\mathcal{P}^-_{a_1,a_2}$, we conclude that $m(a_1,a_2)=\inf\limits_{(u,v)\in\mathcal{P}_{a_1,a_2}}I(u,v)=\inf\limits_{(u,v)\in\mathcal{P}^+_{a_1,a_2}}I(u,v).$
\end{proof}

\vskip0.23in

\begin{Lem}\label{lem3.3}
If $T(a_1,a_2)\le\gamma_{1}$, then
\begin{equation*}
m(a_1,a_2)<\min\big\{m^+(a_1,0),m^+(0,a_2)\big\}<0,
\end{equation*}
where $m^+(a_1,0)$ and $m^+(0,a_2)$ are defined in \eqref{c12}.
\end{Lem}
\begin{proof}
By Lemma \ref{lem2.1}, $m^+(a_1,0)<0$, $m^+(a_1,0)$ can be achieved by $u^*\in S(a_1)$ and $u^*$ is radially symmetric and decreasing. We choose a proper test function $v\in S(a_2)$ such that $(u^*,s\star v)\in S(a_1)\times S(a_2)$ with $s\in\R$. It follows from \eqref{n11} that
\begin{equation*}
h(\rho)\!<\!h_1(\rho)=\frac{1}{2}\rho^{2}-\frac{\mu_1}{4\mathcal{S}^2}\rho^{4}-\frac{\alpha_1}{p}|\mathcal{C}_{p}|^p a^{4-p}_1\rho^{2(p-2)},
\end{equation*}
where we have used the fact that $\mathcal{S}^2_{\mu_1,\mu_2,\beta}< \min\{\frac{1}{\mu_1},\frac{1}{\mu_2}\} \mathcal{S}^2$. %By the similar arguments of Lemma 2.4 in \cite{LY},
By direct calculations, there exists $0\!<\!\hat{\rho}<R_0$ such that $h_1(\hat{\rho})=0$.
In \cite{JJL,JTT}, it has been proved that
\begin{equation*}
m^+(a_1,0)=\inf_{u\in \mathcal{T}^+_{a_1,p,\mu_1,\alpha_1}} \mathcal{A}_{p,\mu_1,\alpha_1}(u)=\inf_{u\in S(a_1)\cap B_{\hat{\rho}}}\mathcal{A}_{p,\mu_1,\alpha_1}(u).
\end{equation*}
Then, we have
\begin{equation*}
\|\nabla u^*\|_2\le \hat{\rho}<R_0< \rho_0.
\end{equation*}
Since $h(R_0)=h(R_1)=0$ and the monotonicity of $h(\rho)$, we deduce that $(u^*,s\star v)\in V(a_1,a_2)$ for $s\ll -1$.
It follows that
\begin{equation*}
\begin{aligned}
m(a_1,a_2)&=\inf_{(u,v) \in V(a_1,a_2)} I(u,v)\le I(u^*,s\star v)\\
&\le \mathcal{A}_{p,\mu_1,\alpha_1}(u^*)+\frac{e^{2s}}{2}\|\nabla v\|^2_2-\frac{e^{p\gamma_ps}}{p}\|v\|^p_p-\frac{e^{4s}}{4}\|v\|^4_4-\frac{\beta}{2}\|u^*(s\star v)\|^2_2\\
%&\le  \mathcal{A}_{\mu_1,p,\alpha_1}(u_0)+\frac{e^{2s}}{2}\|\nabla v\|^2_2-\frac{e^{q\gamma_qs}}{q}\|v\|^q_q-\frac{e^{4s}}{4}\|v\|^4_4-C\|u_0\|^2_{\infty}a^2_2\\
&<  \mathcal{A}_{p,\mu_1,\alpha_1}(u^*)=m^+(a_1,0),
\end{aligned}
\end{equation*}
because $p\gamma_p<2$. On the other hand, we similarly get that $m(a_1,a_2)<m^+(0,a_2)$. Hence, the proof is completed.
\end{proof}

The main aim of this section is the following results.
\begin{Lem}\label{lem3.4}
If $T(a_1,a_2)\le\gamma_{1}$ and
\begin{equation*}
m^+(a_1,a_2)<\min\big\{m^+(a_1,0),m^+(0,a_2)\big\}<0,
\end{equation*}
then $m^+(a_1,a_2)$ can be achieved by some function $(u_{a_1},v_{a_2})\in V(a_1,a_2)$ which is real valued, positive, radially symmetric and radially decreasing.
\end{Lem}

\begin{proof}
Let $\{(u_n,v_n)\}\subset \mathcal{P}^+_{a_1,a_2}$ be a minimizing sequence of $m^+(a_1,a_2)$. From Lemma \ref{lem2.14}, then by taking $(|u_n|,|v_n|)$ and adapting the Schwarz symmetrization to $(|u_n|,|v_n|)$ if necessary, we can obtain a new minimizing sequence (up to a subsequence), such that $(u_n,v_n)$ are all real valued, nonnegative, radially symmetric and decreasing in $r=|x|$. We first prove that $\{(u_n,v_n)\}$ is bounded in $H^1(\R^4,\R^2)$. Since $P_{a_1,a_2}(u_n,v_n)\to 0$, we have
$$
||\nabla u_n||^2_2+\|\nabla v_n\|^2_2-\mu_1\|u_n\|^{4}_{4}-\mu_2\|v_n\|^{4}_{4}-\gamma_p\alpha_1\|u_n\|^p_p$$
$$-\gamma_p\alpha_2\|v_n\|^p_p-
2\beta\|u_nv_n\|^2_2
=o_n(1).
$$
By Gagliardo-Nirenberg inequality and \eqref{b5}, we get
\begin{equation}\label{g5}
\begin{aligned}
I(u_n,v_n)&=\frac{1}{4}\int_{\R^4}|\nabla u_n|^2+|\nabla v_n|^2-\frac{\alpha_1}{p}\big(1-\frac{p\gamma_p}{4}\big)\int_{\R^4}|u_n|^p\\
&\quad \quad -\frac{\alpha_2}{p}\big(1-\frac{p\gamma_p}{4}\big)\int_{\R^4}|v_n|^p\\
&\ge \frac{1}{4}\big(\|\nabla u_n\|^2_2+\|\nabla v_n\|^2_2\big)-\big(1-\frac{p\gamma_p}{4}\big)D_2\|\nabla u_n\|^{p\gamma_p}_2\\
&\quad\quad -\big(1-\frac{p\gamma_p}{4}\big)D_3\|\nabla v_n\|^{p\gamma_p}_2.\\
\end{aligned}
\end{equation}
Therefore,
\begin{equation*}
\frac{1}{4}(\|\nabla u_n\|^2_2+\|\nabla v_n\|^2_2)\end{equation*}
\begin{equation*}
\le \big(1-\frac{p\gamma_p}{4}\big)D_2\|\nabla u_n\|^{p\gamma_p}_2+\big(1-\frac{p\gamma_p}{4}\big)D_3\|\nabla v_n\|^{p\gamma_p}_2+m^{+}(a_1,a_2)+o_n(1),
\end{equation*}
which implies that $\{(u_n,v_n)\}$ is bounded. By the compact embedding of $H^1_r(\R^4)$ into $L^{p_1}(\R^4)$ for $2<p_1<4$, there exists $(u,v)$ such that $(u_n,v_n)\rightharpoonup (u,v)$ in $H^1(\R^4,\R^2)$, $u_n\to u$ in $L^p(\R^4)$ and $v_n\to v$ in $L^p(\R^4)$. Therefore, $u,v\ge 0$ are radial functions. By the Lagrange multiplier's rule (see \cite[ Lemma 3]{BL}), we know that there exists a sequence $\{(\lambda_{1,n},\lambda_{2,n})\}\subset \R\times\R$ such that
\begin{equation}\label{d1}
\int_{\R^4}\nabla u_n\nabla\phi+\lambda_{1,n}u_n\phi-\mu_1(u_n)^3\phi-\alpha_1|u_n|^{p-2}u_n\phi
-\beta |v_n|^2 u_n\phi=o_n(1)\|\phi\|,
\end{equation}
\begin{equation}\label{d2}
\int_{\R^4}\nabla v_n\nabla\varphi+\lambda_{2,n}v_n\varphi-\mu_2(v_n)^3\varphi-\alpha_2|v_n|^{p-2}v_n\varphi-\beta |u_n|^2 v_n\varphi=o_n(1)\|\varphi\|,
\end{equation}
as $n\to \infty$, for every $(\phi,\varphi)\in H^1(\R^4,\R^2)$. In particular, if we take $(\phi,\varphi)=(u_n,v_n)$, $(\lambda_{1,n},\lambda_{2,n})$ is bounded. Up to a subsequence, $(\lambda_{1,n},\lambda_{2,n})\to (\lambda_{1},\lambda_{2})\in \R\times\R$. Passing to the limit in \eqref{d1}-\eqref{d2}, then $(u,v)$ satisfies
\begin{equation*}
\begin{cases}
-\Delta u+\lambda_1u=\mu_1u^3+\alpha_1|u|^{p-2}u+\beta v^2u,\\
-\Delta v+\lambda_2v=\mu_2v^3+\alpha_2|v|^{p-2}v+\beta u^{2}v.\\
\end{cases}
\end{equation*}
By $P_{a_1,a_2}(u,v)=0$,
\begin{equation}\label{d3}
\lambda_1\|u\|^2_2+\lambda_2\|v\|^2_2=(1-\gamma_p)\int_{\R^4}\alpha_1|u|^p+(1-\gamma_p)\int_{\R^4}\alpha_2|v|^p.
\end{equation}
From \eqref{d1}-\eqref{d2} and $P_{a_1,a_2}(u_n,v_n)=o_n(1)$, we obtain
\begin{equation}\label{d4}
\begin{aligned}
\lambda_1a^2_1+\lambda_2a^2_2&=\lim_{n\to\infty}\lambda_1\|u_n\|^2_2+\lambda_2\|v_n\|^2_2\\
&=\lim_{n\to\infty}(1-\gamma_p)\int_{\R^4}\alpha_1|u_n|^p+(1-\gamma_p)\int_{\R^4}\alpha_2|v_n|^p\\
&=(1-\gamma_p)\int_{\R^4}\alpha_1|u|^p+(1-\gamma_p)\int_{\R^4}\alpha_2|v|^p.\\
\end{aligned}
\end{equation}
\vskip1mm
We claim that $u\neq 0$ and $v\neq0$.
\vskip1mm
{\bf Case 1.} If $u=0$ and $v=0$, then $\int_{\R^4}\alpha_1|u_n|^p\to 0$, $\int_{\R^4}\alpha_2|v_n|^p\to 0$, we have
\begin{equation*}
P_{a_1,a_2}(u_n,v_n)=||\nabla u_n||^2_2+\|\nabla v_n\|^2_2-\mu_1\|u_n\|^{4}_{4}-\mu_2\|v_n\|^{4}_{4}-
2\beta\|u_nv_n\|^2_2+o_n(1)=o_n(1).
\end{equation*}
Therefore,
\begin{equation*}
m^+(a_1,a_2)+o_n(1)=I(u_n,v_n)=\frac{1}{4}\int_{\R^4}|\nabla u_n|^2+|\nabla v_n|^2+o_n(1)\ge0,
\end{equation*}
this contradicts the fact that $m^+(a_1,a_2)<0$.
\vskip1mm
{\bf Case 2.} If $u\neq0$ and $v=0$, then $v_n\to 0$ in $L^p(\R^4)$. Let $\bar{u}_n=u_n-u$, $\bar{u}_n\to 0$ in $L^p(\R^4)$. By the maximum principle (see \cite[Theorem 2.10]{hanq}), $u$ is a positive solution of \eqref{a2}. It follows from Lemma \ref{lem2.1} that $m^+(a_1,0)\le m^+(\|u\|_2,0)$.
By the Br\'{e}zis-Lieb Lemma \cite{WM}, we deduce that
\begin{equation*}
\begin{aligned}
P_{a_1,a_2}(u_n,v_n)&=||\nabla u_n||^2_2+\|\nabla v_n\|^2_2-\mu_1\|u_n\|^{4}_{4}-\mu_2\|v_n\|^{4}_{4}-\gamma_p\alpha_1\|u_n\|^p_p\\
&\quad -2\beta\|u_nv_n\|^2_2+o_n(1)\\
&=||\nabla \bar{u}_n||^2_2+\|\nabla v_n\|^2_2-\mu_1\|\bar{u}_n\|^{4}_{4}-\mu_2\|v_n\|^{4}_{4}-2\beta\|\bar{u}_nv_n\|^2_2\\
&\quad+||\nabla u||^2_2-\mu_1\|u\|^{4}_{4}-\gamma_p\alpha_1\|u\|^p_p+o_n(1)\\
%&=P_{a_1,a_2}(\bar{u}_n,v_n)+P_{a_1,a_2}(u,0)+o_n(1)\\
&=||\nabla \bar{u}_n||^2_2+\|\nabla v_n\|^2_2-\mu_1\|\bar{u}_n\|^{4}_{4}-\mu_2\|v_n\|^{4}_{4}-2\beta\|\bar{u}_nv_n\|^2_2+o_n(1).
\end{aligned}
\end{equation*}
Thus,
\begin{equation*}
\begin{aligned}
m^+(a_1,a_2)+o_n(1)&=I(u_n,v_n)=I(\bar{u}_n,v_n)+I(u,0)+o_n(1)\\
&\ge \frac{1}{4}\int_{\R^4}(|\nabla \bar{u}_n|^2+|\nabla v_n|^2)
%-\frac{1}{4}\int_{\R^4}\mu_1|\bar{u}_n|^{4}+\mu_2|v_n|^{4}+\frac{\beta}{2}|\bar{u}_nv_n|^2\\
+m^+(\|u\|_2,0)+o_n(1)\\
&\ge m^+(a_1,0),
\end{aligned}
\end{equation*}
which contradicts our assumption on $m^+(a_1,a_2)$.
\vskip1mm
{\bf Case 3.} If $u=0$ and $v\neq0$. The analysis similar to that in the proof of Case 2 shows that $m^+(a_1,a_2)\ge m^+(0,a_2)+o_n(1)$, we also have a contradiction.
\vskip1mm

Therefore, $u\neq0$ and $v\neq0$. It remains to show that $m^+(a_1,a_2)$ is achieved.
%$\{(u_n,v_n)\}\subset H^1(\R^4,\R^2)$ converges strongly.
By the maximum principle (see \cite[Theorem 2.10]{hanq}), $u,v>0$. Then, from Lemma \ref{lem2.11}, we get $\lambda_1,\lambda_2>0$. Moreover, combining \eqref{d3} with \eqref{d4}, we have
\begin{equation}\label{d5}
\lambda_1a^2_1+ \lambda_2a^2_2=\lambda_1\|u\|^2_2+\lambda_2\|v\|^2_2.
\end{equation}
Since $\|u\|^2_2\le a^2_1$ and $\|v\|^2_2\le a^2_2$, it follows from \eqref{d5} that $\|u\|^2_2= a^2_1$ and $\|v\|^2_2= a^2_2$, and hence $(u,v)\in \mathcal{P}_{a_1,a_2}$. Let $(\bar{u}_n,\bar{v}_n)=(u_n-u,v_n-v)$, we obtain
\begin{equation*}
\begin{aligned}
o_n(1)&=P_{a_1,a_2}(u_n,v_n)=P_{a_1,a_2}(\bar{u}_n,\bar{v}_n)+P_{a_1,a_2}(u,v)+o_n(1)\\
&=||\nabla \bar{u}_n||^2_2+\|\nabla \bar{v}_n\|^2_2-\mu_1\|\bar{u}_n\|^{4}_{4}-\mu_2\|\bar{v}_n\|^{4}_{4}-2\beta\|\bar{u}_n\bar{v}_n\|^2_2+o_n(1).\\
\end{aligned}
\end{equation*}
Therefore,
\begin{equation*}
\begin{aligned}
m^+(a_1,a_2)+o_n(1)&=I(u_n,v_n)=I(\bar{u}_n,\bar{v}_n)+I(u,v)+o_n(1)\\
&\ge \frac{1}{4}\int_{\R^4}(|\nabla \bar{u}_n|^2+|\nabla v_n|^2)+m^+(a_1,a_2)+o_n(1)\\
&\ge m^+(a_1,a_2)+o_n(1).
\end{aligned}
\end{equation*}
We then conclude that $I(u,v)=m^+(a_1,a_2)$ and $(u_n,v_n)\to (u,v)$ in $H^1(\R^4,\R^2)$. Then, we have proved that $m^+(a_1,a_2)$ can be attained by some $(u_{a_1},v_{a_2})$ which is real valued, positive, radially symmetric and decreasing in $r=|x|$.
\end{proof}

\vskip0.3in

In the following, we derive a better upper bound of $m^+(a_1,a_2)$. We consider the problem
\begin{equation}\label{g1}
\begin{cases}
-\Delta u +\lambda u =\alpha_1|u|^{p-2}u \ \ \text{in}\ \R^4,\\
\int_{\R^4}|u|^2=a^2_1,
\end{cases}
\end{equation}
%under the constraint $S(a_1)$,
where $\alpha_1>0$ and $2<p<3$. Define $I_0(u)=\frac{1}{2}\| \nabla u\|^2_{2}-\frac{\alpha_1}{p}\|u\|^{p}_{p}$, then solutions $u$ of \eqref{g1} can be found as minimizers of
\begin{equation*}
m_0(a_1)=\inf_{u\in S(a_1)} I_0(u)>-\infty,
\end{equation*}
where $\lambda$ is a Lagrange multiplier. From \cite{KM}, we obtain that \eqref{g1} has a unique positive solution $(\lambda,u_{p,\alpha_1})$ given by
\begin{equation}\label{z6}
\lambda=\big(\frac{a^2_1}{\|w_{p}\|^2_2}\alpha_1^{\frac{2}{p-2}} \big)^{\frac{p-2}{6-2p}},\quad \ \  \ \ u_{p,\alpha_1}=\big(\frac{\lambda}{\alpha_1}\big)^{\frac{1}{p-2}}w_p(\lambda^{\frac{1}{2}}x),
\end{equation}
where $w_{p}$ is the unique positive solution of \eqref{g2}.
%\begin{equation}\label{g2}
%-\Delta w + w = |w|^{p-2}w \ \ \text{in}\ \R^4.
%\end{equation}

\vskip1mm
By direct calculations, we have the following corollaries.
\begin{Cor} \label{cor4.1}
Let $a_1>0$, $\alpha_1>0$ and  $p\in(2,3)$. Then, we have
$$m_0(a_1)=I_0(u_{p,\alpha_1})=-K_{p,\alpha_1}\cdot a_1^\frac{4-p}{3-p}<0,$$
where $K_{p,\alpha_1}:=\frac{3-p}{4-p}\|w_{p}\|^{\frac{2-p}{3-p}}_2\alpha_1^{\frac{1}{3-p}}>0$.
\end{Cor}

If $\alpha_1:=\alpha_2$ and $a_1:=a_2$ in \eqref{g1}, the analogues of $u_{p,\alpha_1}$ will be denoted by
$u_{p,\alpha_2}$ respectively.
Combining Lemma \ref{lem3.3} with Corollary \ref{cor4.1}, we have
\begin{Cor} \label{cor4.2}
Let $a_i,\alpha_i>0(i=1,2)$ and  $2<p<3$. If $T(a_1,a_2)\le\gamma_{1}$, then we have
$$m^+(a_1,a_2)<\min\big\{-K_{p,\alpha_1}\cdot a^\frac{4-p}{3-p}_1,-K_{p,\alpha_2} \cdot a^\frac{4-p}{3-p}_2\big\},$$
where $K_{p,\alpha_1}:=\frac{3-p}{4-p}\|w_{p}\|^{\frac{2-p}{3-p}}_2\alpha^{\frac{1}{3-p}}_1>0$ and $K_{p,\alpha_2}:=\frac{3-p}{4-p}\|w_p\|^{\frac{2-p}{3-p}}_2\alpha^{\frac{1}{3-p}}_2>0$.
\end{Cor}

Next, we prove the asymptotic behaviour of  the ground states to \eqref{eq1.1}-\eqref{eq1.11}.
\begin{Lem} \label{lem3.8}
For any fixed $\alpha_i,\mu_i,a_i >0(i=1,2)$, $\beta>0$ and $2<p<3$. Letting $(a_1,a_2)\to (0,0)$ and $a_1\sim a_2$, then for any ground state $(u_{a_1},v_{a_2})$ of \eqref{eq1.1}-\eqref{eq1.11}, we have
%\begin{equation*}
%%\frac{m(a)}{a_1^{\frac{4-p}{3-p}}}\to -K_{N,q},\ \
%%\frac{\lambda_{a}}{a^{\frac{p-2}{3-p}}}\to
%%-\frac{2q(1-\gamma_q)}{2-q\gamma_q}K_{N,q},\ \
%%\frac{\|u_{a}\|^2_{\dot{H}^{1/2}}}{a^{\frac{2q(1-\gamma_q)}{2-q\gamma_q}}}\to \frac{2q\gamma_q}{2-q\gamma_q}K_{N,q}.
%\quad \frac{\|u_{a_k}\|^{2^*}_{2^*}}{A}\to 0,
%%\end{equation*}
%\begin{equation*}
%\frac{\|u_{a_k}\|^2_{\dot{H}^{1/2}}}{A}\to \frac{2N(q-1)}{2-N(q-1)}K_{N,q},\ \quad %\frac{\|u_{a_k}\|^{q+1}_{q+1}}{A}\to -\frac{2(q+1)}{\mu\big[N(q-1)-2\big]}K_{N,q},
%\end{equation*}
%where $L=a^{\frac{N+q+1-Nq}{2-N(q-1)}}$ and $K_{N,q}=\frac{N+2-Nq}{2(N+q+1-Nq)}\cdot \|Q\|^{\frac{2(q-1)}{N(q-1)-2}}_2\mu^{\frac{2}{2-N(q-1)}}>0$.
%Moreover, we have
\begin{equation*}
\Big(\big(\frac{L_1}{\alpha_1}\big)^{-\frac{1}{p-2}}u_{a_1}(L_1^{-\frac{1}{2}}x), \big(\frac{L_2}{\alpha_2}\big)^{-\frac{1}{p-2}}v_{a_2}(L_2^{-\frac{1}{2}}x)\Big)\to (w_p,w_p)
 \ \  \text{in}\ H^{1}(\R^4,\R^2),
\end{equation*}
where $w_p$ is defined as above,
$L_1=\big(\frac{a^2_1}{\|w_{p}\|^2_2}\alpha_1^{\frac{2}{p-2}} \big)^{\frac{p-2}{6-2p}}$ and $L_2=\big(\frac{a^2_2}{\|w_{p}\|^2_2}\alpha_2^{\frac{2}{p-2}} \big)^{\frac{p-2}{6-2p}}$.
%and $\gamma_p=2(p-2)$, $\gamma_q=2(q-2)$.
%$F_\mu$ restricted to $S(a)$ has a ground state. This ground state is a local minimizer of $F_\mu$ in the set $V(a)$.
\end{Lem}
\begin{proof}
For $\mu_i,\alpha_i,\beta>0(i=1,2)$ fixed, let $(a_{1,k},a_{2,k})\to (0^+,0^+)$ as $k\to +\infty$ and $(u_{a_{1,k}},v_{a_{2,k}})\in V(a_{1,k},a_{2,k})$ be a minimizer of $m^+(a_{1,k},a_{2,k})$ for each $k\in\mathbb{N}$, where $V(a_{1,k},a_{2,k})=\big\{(u,v)\in S(a_{1,k})\times S(a_{2,k}) : [\|\nabla u_{a_{1,k}}\|^2_2+\|\nabla v_{a_{2,k}}\|^2_2]^{\frac{1}{2}}< \rho_0 \big\}$.
By Lemma $\ref{lem3.4}$, we can suppose that $\big\{(u_{a_{1,k}},v_{a_{2,k}})\big\}$ is positive and radially symmetric, i.e., $0< u_{a_{1,k}},v_{a_{2,k}}\in H^{1}_r(\R^4)$.
%By the strong maximum principle, $ 0<u_{a_{1,k}},v_{a_{2,k}}$.
From Lemma \ref{lem2.3}, we have $\rho_0\to 0$ as $(a_{1,k},a_{2,k})\to (0,0)$. So $\|\nabla u_{a_{1,k}}\|^2_2+\|\nabla v_{a_{2,k}}\|^2_2\to 0$ as well. In addition,
\begin{equation*}
0>m^+(a_{1,k},a_{2,k})+o_k(1)=I(u_{a_{1,k}},v_{a_{2,k}})\ge h\big([\|\nabla u_{a_{1,k}}\|^2_2+\|\nabla v_{a_{2,k}}\|^2_2]^{\frac{1}{2}}\big)\to 0,
\end{equation*}
it follows that $m^+(a_{1,k},a_{2,k})\to 0$ as $(a_{1,k},a_{2,k})\to (0,0)$.
Since $P_{a_{1,k},a_{2,k}}(u_{a_{1,k}},v_{a_{2,k}})=0$, we get
\begin{equation*}
\begin{aligned}
I(u_{a_{1,k}},v_{a_{2,k}})&=\frac{1}{4}\int_{\R^4}|\nabla u_{a_{1,k}}|^2+|\nabla v_{a_{2,k}}|^2-\big(\frac{1}{p}-\frac{\gamma_p}{4}\big)\int_{\R^4}\alpha_1|u_{a_{1,k}}|^p\\
&\quad -\big(\frac{1}{p}-\frac{\gamma_p}{4}\big)\int_{\R^4}\alpha_2|v_{a_{2,k}}|^p\\
&=\big(\frac{1}{2}-\frac{1}{p\gamma_p}\big)\int_{\R^4}|\nabla u_{a_{1,k}}|^2+|\nabla v_{a_{2,k}}|^2\\
&\quad -\big(\frac{1}{4}-\frac{1}{p\gamma_p}\big)\int_{\R^4}\big(\mu_1|u_{a_{1,k}}|^4+\mu_2|v_{a_{2,k}}|^4\\
&\quad +2\beta|u_{a_{1,k}}|^2|v_{a_{2,k}}|^2\big)\\%-\big(\frac{1}{q}-\frac{\gamma_q}{p\gamma_p}\big)\int_{\R^4}\alpha_2|v_{a_{2,k}}|^q\\
&\le \min\big\{-K_{p,\alpha_1}\cdot a^\frac{4-p}{3-p}_{1,k},-K_{p,\alpha_2} \cdot a^\frac{4-p}{3-p}_{2,k}\big\}.\\
\end{aligned}
\end{equation*}
%It follows from $p\le q$ that $\frac{4-p}{3-p}\le \frac{4-q}{3-q}$.
Then, there exist $C_i,C'_i>0(i=1,2)$ such that
\begin{equation}\label{g4}
C_1 a^\frac{4-p}{3-p}_{1,k}+C_2 a^\frac{4-p}{3-p}_{2,k}\le\int_{\R^4}|\nabla u_{a_{1,k}}|^2+|\nabla v_{a_{2,k}}|^2 \le C'_1 a^\frac{4-p}{3-p}_{1,k}+C'_2 a^\frac{4-p}{3-p}_{2,k}.
\end{equation}
Moreover, from \eqref{g4}, we have
\begin{equation*}
\begin{aligned}
&\int_{\R^4}\big(\mu_1|u_{a_{1,k}}|^4+\mu_2|v_{a_{2,k}}|^4+2\beta|u_{a_{1,k}}|^2|v_{a_{2,k}}|^2\big)\\
&\le \mathcal{S}^2_{\mu_1,\mu_2, \beta}\Big(\int_{\R^4}|\nabla u_{a_{1,k}}|^2+|\nabla v_{a_{2,k}}|^2\Big)^2\\
&\le C a_{1,k}^\frac{2(4-p)}{3-p}+C'a_{2,k}^\frac{2(4-p)}{3-p}.
\end{aligned}
\end{equation*}
%By Corollary \ref{cor4.2}, we obtain
%\begin{equation*}
%\begin{aligned}
%-K_{p,\alpha_1}\cdot a^\frac{4-p}{3-p}_{1,k}-K_{p,\alpha_2}\cdot a^\frac{4-p}{3-p}_{2,k}&>m(a_{1,k},a_{2,k})=I(u_{a_{1,k}},v_{a_{2,k}})\\
%%&=\frac{1}{2}\big[\|\nabla u_{a_{1,k}}\|^2_2+\|\nabla v_{a_{2,k}}\|^2_2\big]-\frac{1}{4}\big[\mu_1\|u_{a_{1,k}}\|^4_4+\mu_2\|v_{a_{2,k}}\|^4_4 \\ &\quad+2\beta\|u_{a_{1,k}}v_{a_{2,k}}\|^2_2\big]-\frac{1}{p}\|u_{a_{1,k}}\|^q_q-\frac{1}{q}\|v_{a_{2,k}}\|^q_q\\
%&\ge \inf_{S(a_1)\times S(a_2)}\Big\{\frac{1}{2}\big[\|\nabla u\|^2_2+\|\nabla v\|^2_2\big]-\frac{1}{p}\|u\|^p_p-\frac{1}{p}\|v\|^p_p\Big\}-\\
%&\quad -\frac{1}{4}\big[\mu_1\|u_{a_{1,k}}\|^4_4+\mu_2\|v_{a_{2,k}}\|^4_4  +2\beta\|u_{a_{1,k}}v_{a_{2,k}}\|^2_2\big]\\
%&=-K_{p,\alpha_1}\cdot a^\frac{4-p}{3-p}_{1,k}-K_{p,\alpha_2}\cdot a^\frac{4-p}{3-p}_{2,k}\\
%&\quad-\frac{1}{4}\big[\mu_1\|u_{a_{1,k}}\|^4_4+\mu_2\|v_{a_{2,k}}\|^4_4  +2\beta\|u_{a_{1,k}}v_{a_{2,k}}\|^2_2\big].\\
%\end{aligned}
%\end{equation*}
The Lagrange multipliers rule implies the existence of some $\lambda_{1,k},\lambda_{2,k} \in \R$ such that
\begin{align*}\label{g3}
\int_{\R^4}\nabla u_{a_{1,k}} \nabla\phi+\lambda_{1,k}\int_{\R^4}u_{a_{1,k}}\phi&=
\alpha_1\int_{\R^4}|u_{a_{1,k}}|^{p-2}u_{a_{1,k}}\phi+\mu_1\int_{\R^4}(u_{a_{1,k}})^{3}
\phi\\
&\quad +\beta\int_{\R^4}|v_{a_{2,k}}|^2u_{a_{1,k}}\phi,
\end{align*}
\begin{align*}
\int_{\R^4}\nabla v_{a_{2,k}} \nabla\psi+\lambda_{2,k}\int_{\R^4}v_{a_{2,k}}\psi&=
\alpha_2\int_{\R^4}|v_{a_{2,k}}|^{p-2}v_{a_{2,k}}\psi+\mu_2\int_{\R^4}(v_{a_{2,k}})^{3}
\psi\\
& \quad  +\beta\int_{\R^4}|u_{a_{1,k}}|^2v_{a_{2,k}}\psi,
\end{align*}
for each $\phi,\psi\in H^{1}(\R^4)$. %By Corollary \ref{cor4.2}, we can deduce that $m^+(a_1,a_2)<\min\big\{ -K_{p,\alpha_1}\cdot a^\frac{4-p}{3-p}_1, -K_{p,\alpha_2} \cdot a^\frac{4-p}{3-p}_2\big\}$.
Taking $\phi=u_{a_{1,k}}$ and $\psi=v_{a_{2,k}}$, we get
\begin{equation*}
\lambda_{1,k}a^2_{1,k}=-\big[\|\nabla u_{a_{1,k}}\|^2_2-\mu_1\|u_{a_{1,k}}\|^4_4 -\beta\|u_{a_{1,k}}v_{a_{2,k}}\|^2_2-\alpha_1\|u_{a_{1,k}}\|^p_p\big],
\end{equation*}
and
\begin{equation*}
\lambda_{2,k}a^2_{2,k}=-\big[\|\nabla v_{a_{2,k}}\|^2_2-\mu_2\|v_{a_{2,k}}\|^4_4 -\beta\|v_{a_{2,k}}u_{a_{1,k}}\|^2_2-\alpha_2\|v_{a_{2,k}}\|^p_p\big],
\end{equation*}
hence that
\begin{equation}\label{x1}
2K_{p,\alpha_1}\cdot a^\frac{4-p}{3-p}_{1,k}\le \lambda_{1,k}a^2_{1,k}+\lambda_{2,k}a^2_{2,k}\ \ \text{or}\ \ 2K_{p,\alpha_2}\cdot a^\frac{4-p}{3-p}_{2,k}\le \lambda_{1,k}a^2_{1,k}+\lambda_{2,k}a^2_{2,k},
\end{equation}
and finally that
$\lambda_{1,k} \gtrsim  a^\frac{p-2}{3-p}_{1,k}$ and $\lambda_{2,k}\gtrsim a^\frac{p-2}{3-p}_{2,k}$. In addition, by $P_{a_{1,k},a_{2,k}}(u_{a_{1,k}},v_{a_{2,k}})=0$,
\begin{equation}\label{x2}
\begin{aligned}
&C a^\frac{4-p}{3-p}_{1,k}+C' a^\frac{4-p}{3-p}_{2,k}\\
&\le
\lambda_{1,k}a^2_{1,k}+\lambda_{2,k}a^2_{2,k}\\
&=-(1-\frac{1}{\gamma_p})\big[\|\nabla u_{a_{1,k}}\|^2_2+\|\nabla v_{a_{2,k}}\|^2_2]+\big(1-\frac{1}{\gamma_p}\big)\big[\mu_1\|u_{a_{1,k}}\|^4_4\\
&\quad+\mu_2\|v_{a_{2,k}}\|^4_4  +2\beta\|u_{a_{1,k}}v_{a_{2,k}}\|^2_2\big]\\
&\le -(1-\frac{1}{\gamma_p})\big[\|\nabla u_{a_{1,k}}\|^2_2+\|\nabla v_{a_{2,k}}\|^2_2]\le C_1a^\frac{4-p}{3-p}_{1,k}+ C'_1a^\frac{4-p}{3-p}_{2,k}.\\
\end{aligned}
\end{equation}
%Similarly, if $p\ge q$, \eqref{x2} also holds.
%Note that the above arguments does not depend on the choice of $a_{1,k}\sim a_{2,k}$. That holds for any  $(a_{1,k},a_{2,k})\to (0,0)$ and $2<p,q<3$.
Therefore, from \eqref{x1}-\eqref{x2}, there  are  $\lambda_{1,k}\sim a^\frac{p-2}{3-p}_{1,k}$ and $\lambda_{2,k}\sim a^\frac{p-2}{3-p}_{2,k}$ as $(a_{1,k},a_{2,k})\to (0,0)$ and $a_{1,k}\sim a_{2,k}$. Moreover, we also have $\|u_{a_{1,k}}\|^p_p\sim a^\frac{4-p}{3-p}_{1,k}$ and $\|v_{a_{2,k}}\|^p_p\sim a^\frac{4-p}{3-p}_{2,k}$.
%which implies that
%\begin{equation*}
%\frac{m(a_{1,k},a_{2,k})}{a^\frac{4-p}{3-p}_{1,k}}\to -K_{p,\alpha_1}\ \ \text{as}\ \ k\to \infty.
%\end{equation*}
Define
\begin{equation*}
\tilde{u}_{a_{1,k}}=\frac{1}{\theta_{1,k}}u_{a_{1,k}}\big(\frac{1}{\gamma_{1,k}}x\big)\ \ \text{and}\ \ \tilde{v}_{a_{2,k}}=\frac{1}{\theta_{2,k}}v_{a_{2,k}}\big(\frac{1}{\gamma_{2,k}}x\big),
\end{equation*}
where
\begin{equation*}
\theta_{1,k}=\big(\frac{a^2_{1,k}}{\|w_{p}\|^2_2}\big)^{\frac{1}{6-2p}}\alpha_1^{\frac{1}{3-p}}, \quad \theta_{2,k}=\big(\frac{a^2_{2,k}}{\|w_{p}\|^2_2}\big)^{\frac{1}{6-2p}}\alpha_2^{\frac{1}{3-p}},
\end{equation*}
and
\begin{equation*}
\gamma_{1,k}=\big(\frac{a^2_{1,k}}{\|w_{p}\|^2_2}\big)^{\frac{p-2}{4(3-p)}}\alpha_1^{\frac{1}{6-2p}},\quad \gamma_{2,k}=\big(\frac{a^2_{2,k}}{\|w_{p}\|^2_2}\big)^{\frac{p-2}{4(3-p)}}\alpha_2^{\frac{1}{6-2p}}.
\end{equation*}
Then,
\begin{equation}\label{y2}
\|\nabla \tilde{u}_{a_{1,k}}\|^2_2=\frac{\gamma^2_{1,k}}{\theta^2_{1,k}}\|\nabla u_{a_{1,k}}\|^2_2=\Big[\frac{a_{1,k}}{\|w_{p}\|_2}\Big]^{\frac{p-4}{3-p}}\frac{1}{\alpha^{3-p}_1}\|\nabla u_{a_{1,k}}\|^2_2,
\end{equation}
\vskip1mm
\begin{equation}\label{y3}
\|\nabla \tilde{v}_{a_{2,k}}\|^2_2=\frac{\gamma^2_{2,k}}{\theta^2_{2,k}}\|\nabla v_{a_{2,k}}\|^2_2=\Big[\frac{a_{2,k}}{\|w_{p}\|_2}\Big]^{\frac{p-4}{3-p}}\frac{1}{\alpha^{3-p}_2}\|\nabla v_{a_{2,k}}\|^2_2,
\end{equation}
\begin{equation*}
\| \tilde{u}_{a_{1,k}}\|^2_2=\frac{\gamma^4_{1,k}}{\theta^2_{1,k}}\| u_{a_{1,k}}\|^2_2=\|w_p\|^2_2,\ \quad
\| \tilde{v}_{a_{2,k}}\|^2_2=\frac{\gamma^4_{2,k}}{\theta^2_{2,k}}\| v_{a_{2,k}}\|^2_2=\|w_p\|^2_2,
\end{equation*}
and
\begin{equation*}
\| \tilde{u}_{a_{1,k}}\|^p_p=\frac{\gamma^4_{1,k}}{\theta^p_{1,k}}\| u_{a_{1,k}}\|^p_p=\frac{\|w_p\|^{\frac{4-p}{3-p}}_2}{\alpha^{\frac{p-2}{3-p}}_1}\frac{\| u_{a_{1,k}}\|^p_p}{a^{\frac{4-p}{3-p}}_{1,k}},\ \quad
\end{equation*}

\begin{equation*}
\| \tilde{v}_{a_{2,k}}\|^p_p=\frac{\gamma^4_{2,k}}{\theta^p_{2,k}}\| v_{a_{2,k}}\|^2_2=\frac{\|w_p\|^{\frac{4-p}{3-p}}_2}{\alpha^{\frac{p-2}{3-p}}_2}\frac{\| v_{a_{2,k}}\|^p_p}{a^{\frac{4-p}{3-p}}_{2,k}}.
\end{equation*}

%\begin{equation*}
%\|\tilde{u}_{a_{1,k}}\|^4_4=
%\frac{\gamma^4_{1,k}}{\theta^4_{1,k}}\|u_{a_{1,k}}\|^4_4=
%\end{equation*}
We note that such an additional assumption is appropriate for $a_{1,k}\sim a_{2,k}$ as $(a_{1,k}, a_{2,k}) \to (0,0)$. We see that $(\tilde{u}_{a_{1,k}},\tilde{v}_{a_{2,k}})$ solves
\begin{equation}\label{g7}
\begin{cases}
-\Delta\tilde{u}_{a_{1,k}}+\frac{\lambda_{1,k}}{\gamma^2_{1,k}}\tilde{u}_{a_{1,k}}=\mu_1\frac{\theta^2_{1,k}}{\gamma^2_{1,k}}\tilde{u}^3_{a_{1,k}}+\alpha_1\frac{\theta^{p-2}_{1,k}}{\gamma^2_{1,k}}|\tilde{u}_{a_{1,k}}|^{p-2}\tilde{u}_{a_{1,k}}+\beta\frac{\theta^{2}_{2,k}}{\gamma^2_{1,k}}\tilde{v}^2_{a_{2,k}}\tilde{u}_{a_{1,k}},\\
-\Delta\tilde{v}_{a_{2,k}}+\frac{\lambda_{2,k}}{\gamma^2_{2,k}}\tilde{v}_{a_{2,k}}=\mu_2\frac{\theta^2_{2,k}}{\gamma^2_{2,k}}\tilde{v}^3_{a_{2,k}}+\alpha_2\frac{\theta^{p-2}_{2,k}}{\gamma^2_{2,k}}|\tilde{v}_{a_{2,k}}|^{p-2}\tilde{v}_{a_{2,k}}+\beta\frac{\theta^{2}_{1,k}}{\gamma^2_{2,k}}\tilde{u}^2_{a_{1,k}}\tilde{v}_{a_{2,k}}.\\
\end{cases}
\end{equation}
The assumption $a_{1,k}\sim a_{2,k}$ is used to ensure that $\frac{\theta^{2}_{2,k}}{\gamma^2_{1,k}}\to 0$ and $\frac{\theta^{2}_{1,k}}{\gamma^2_{2,k}}\to 0$ as $(a_{1,k}, a_{2,k}) \to (0,0)$.
From the asymptotic properties of $\lambda_{1,k},\lambda_{2,k},\theta_{1,k},\theta_{2,k},\gamma_{1,k},\gamma_{2,k}$, there exist $\lambda^*_1, \lambda^*_2>0$ such that
\begin{equation*}
\begin{aligned}
&\frac{\lambda_{1,k}}{\gamma^2_{1,k}}\to \lambda^*_1,\qquad \frac{\theta^2_{1,k}}{\gamma^2_{1,k}}\to 0,\qquad \alpha_1\frac{\theta^{p-2}_{1,k}}{\gamma^2_{1,k}}\to 1,\qquad \frac{\theta^{2}_{2,k}}{\gamma^2_{1,k}}\to 0;\\
&\frac{\lambda_{2,k}}{\gamma^2_{2,k}}\to \lambda^*_2,\qquad \frac{\theta^2_{2,k}}{\gamma^2_{2,k}}\to 0,\qquad \alpha_2\frac{\theta^{p-2}_{2,k}}{\gamma^2_{2,k}}\to 1,\qquad \frac{\theta^{2}_{1,k}}{\gamma^2_{2,k}}\to 0,\\
\end{aligned}
\end{equation*}
as $k\to \infty$.
Since \eqref{g4} and \eqref{y2}-\eqref{y3}, there exists $(u_0,v_0)\in H^1(\R^4,\R^2)$ such that
\begin{equation*}
(\tilde{u}_{a_{1,k}},\tilde{v}_{a_{2,k}})\rightharpoonup (u_0,v_0) \quad \text{in}\ H^{1}(\R^4,\R^2).
\end{equation*}
Therefore, $(u_0,v_0)$ solves
\begin{equation}\label{g6}
\begin{cases}
-\Delta u+\lambda^*_1u-|u|^{p-2}u=0,\\
-\Delta v+\lambda^*_2v-|v|^{p-2}v=0.
\end{cases}
\end{equation}
Since $0 \le u_0,v_0$ and $(u_0,v_0)\neq (0,0)$, it follows from the strong maximum principle that $0<u_0,v_0$. We know that
\begin{equation*}
u_0=(\lambda^*_1)^{\frac{1}{p-2}}w_p(\sqrt{\lambda^*_1}x),\ \quad v_0=(\lambda^*_2)^{\frac{1}{p-2}}w_p(\sqrt{\lambda^*_2}x)
\end{equation*}
is the unique positive solution of \eqref{g6} (see \cite{LN,KM}).
Testing \eqref{g7} and \eqref{g6} with $\tilde{u}_{a_{1,k}}-u_0,\tilde{v}_{a_{2,k}}-v_0$ respectively,
%testing \eqref{g6} with $(\tilde{u}_{a_{1,k}}-u_0,\tilde{v}_{a_{2,k}}-v_0)$ as well,
we obtain
\begin{equation*}
\begin{aligned}
&\int_{\R^4}|\nabla (\tilde{u}_{a_{1,k}}-u_0)|^2+\int_{\R^4}\big(\frac{\lambda_{1,k}}{\gamma^2_{1,k}}\tilde{u}_{a_{1,k}}-\lambda^*_1u_0\big)(\tilde{u}_{a_{1,k}}-u_0)\\
&\quad=\int_{\R^4}|\nabla (\tilde{u}_{a_{1,k}}-u_0)|^2+\lambda^*_1|\tilde{u}_{a_{1,k}}-u_0|^2+o_k(1)=o_k(1),\\
\end{aligned}
\end{equation*}
\begin{equation*}
\begin{aligned}
&\int_{\R^4}|\nabla (\tilde{v}_{a_{2,k}}-v_0)|^2+\int_{\R^4}\big(\frac{\lambda_{2,k}}{\gamma^2_{2,k}}\tilde{v}_{a_{2,k}}-\lambda^*_2v_0\big)(\tilde{v}_{a_{2,k}}-v_0)\\
&\quad=\int_{\R^4}|\nabla (\tilde{v}_{a_{2,k}}-v_0)|^2+\lambda^*_2|\tilde{v}_{a_{2,k}}-v_0|^2+o_k(1)=o_k(1).\\
\end{aligned}
\end{equation*}
%We deduce from the same computations in \eqref{d4} and \eqref{d5} that
%$$\lambda^*_1\|w_p\|^2_2+\lambda^*_2\|w_q\|^2_2=\lambda^*_1\|u_0\|^2_2+\lambda^*_2\|v_0\|^2_2.$$
We deduce that $\lim\limits_{k\to\infty}\|\tilde{u}_{a_{1,k}}\|^2_2=\|u_0\|^2_2=\|w_p\|^2_2$ and $\lim\limits_{k\to\infty}\|\tilde{v}_{a_{2,k}}\|^2_2=\|v_0\|^2_2=\|w_p\|^2_2$, hence that
\begin{equation*}
(\lambda^*_1)^{\frac{6-2p}{p-2}}\|w_p\|^2_2=\|u_0\|^2_2=\|w_p\|^2_2,\quad (\lambda^*_2)^{\frac{6-2p}{p-2}}\|w_p\|^2_2=\|v_0\|^2_2=\|w_p\|^2_2,
\end{equation*}
and finally that $\lambda^*_1=1$ and $\lambda^*_2=1$. Since $(w_p,w_p)$ is unique, $(\tilde{u}_{a_{1,k}},\tilde{v}_{a_{2,k}})\to (w_p,w_p)$ in $H^1(\R^4,\R^2)$ as $(a_{1,k},a_{2,k})\to (0,0)$ and $a_{1,k}\sim a_{2,k}$. Then, by direct calculations, there exist $L_1=\big(\frac{a^2_1}{\|w_{p}\|^2_2}\alpha_1^{\frac{2}{p-2}} \big)^{\frac{p-2}{6-2p}}$ and $L_2=\big(\frac{a^2_2}{\|w_{p}\|^2_2}\alpha_2^{\frac{2}{p-2}} \big)^{\frac{p-2}{6-2p}}$ such that
$$(\tilde{u}_{a_{1,k}},\tilde{v}_{a_{2,k}})=\Big(\big(\frac{L_1}{\alpha_1}\big)^{-\frac{1}{p-2}}u_{a_1}(L_1^{-\frac{1}{2}}x), \big(\frac{L_2}{\alpha_2}\big)^{-\frac{1}{p-2}}v_{a_2}(L_2^{-\frac{1}{2}}x)\Big).$$
\end{proof}

\noindent \textbf{Proof of Theorem 1.2.}
It follows from Lemma \ref{lem3.1} and Lemma \ref{lem3.4} that there is a local minimizer of $I$ on $V(a_1,a_2)$.
The item 2 of Theorem \ref{th1.1} follows from Lemma \ref{lem3.8}.
\qed

\section{Proof of Theorem 1.3}
In this section, we follow the approach of \cite{WW}. By Theorem \ref{th1.1}, $\mu_i,\alpha_i>0(i=1,2)$, $\beta>0$ and $T(a_1,a_2)\le\gamma_1$, $(\bar{u},\bar{v})$ is a local minimum point of $I$ restricted on $S(a_1)\times S(a_2)$, i.e., $I(\bar{u},\bar{v})=m^+(a_1,a_2)$. We can suppose that $(\bar{u},\bar{v})$ is a positive and radially symmetric decreasing function. Using the ground state $(\bar{u},\bar{v})$, we can not only obtain a good energy estimate $m^-(a_1,a_2)$ for $2<p<3$, but also recover the compactness of minimizing sequence at energy $m^-(a_1,a_2)$.

%From Lemma \ref{lem3.2}, we define
%\begin{equation*}
%m^-(a_1,a_2)=\inf_{(u,v)\in \mathcal{P}^-_{a_1,a_2}} I(u,v)
%\end{equation*}

\begin{Lem}\label{lem4.1}
Let $2<p<3$, $\mu_i,\alpha_i,a_i>0(i=1,2)$, and $\beta\in\big(0,\min\{\mu_1,\mu_2\}\big)\cup\big(\max\{\mu_1,\mu_2\},\infty\big)$. Then for $T(a_1,a_2)\le\gamma_{1}$,
\begin{equation*}
0<m^-(a_1,a_2)=\inf_{(u,v)\in \mathcal{P}^-_{a_1,a_2}} I(u,v)<m^+(a_1,a_2)+\frac{k_1+k_2}{4}\mathcal{S}^2,
\end{equation*}
where $k_1=\frac{\beta-\mu_2}{\beta^2-\mu_1\mu_2}$ and $k_2=\frac{\beta-\mu_1}{\beta^2-\mu_1\mu_2}$.
\end{Lem}

\begin{proof}
From \eqref{L1}, $U_{\varepsilon}=\frac{2\sqrt{2}\varepsilon}{\varepsilon^2+|x|^2}$, taking a radially decreasing cut-off function $\xi \in C_0^{\infty}(\R^4)$ such that $\xi \equiv 1$ in $B_1$, $\xi \equiv 0$ in $\R^4 \backslash B_2$, and let $W_{\varepsilon}(x) = \xi(x) U_{\varepsilon}(x)$. We have (see \cite{ABC}),
\begin{equation}\label{f1}
\|\nabla W_\varepsilon\|^2_2=\mathcal{S}^2+O(\varepsilon^2),\quad \|W_\varepsilon\|^4_4=\mathcal{S}^2+O(\varepsilon^4),%\quad\| W_\varepsilon\|^p_p=O(\varepsilon^{4-p}),
\end{equation}
and
\begin{equation}\label{f11}
 \| W_\varepsilon\|^p_p=O(\varepsilon^{4-p}), \quad \|W_\varepsilon\|^2_2=O(\varepsilon^2|\ln\varepsilon|).
\end{equation}
%\vskip1mm
Setting $$(\widehat{W}_{\varepsilon,t_1},\widehat{V}_{\varepsilon,t_2})=(\bar{u}+t_1W_{\varepsilon},\bar{v}+t_2W_{\varepsilon}), \big(\overline{W}_{\varepsilon,t_1},\overline{V}_{\varepsilon,t_2}\big)=\big(s_1\widehat{W}_{\varepsilon,t_1}(s_1x),s_2\widehat{V}_{\varepsilon,t_2}(s_2x)\big),$$ where $t_1,t_2>0$ and $\frac{t_1}{a^2_1}\int_{\R^4}\bar{u}W_\varepsilon=\frac{t_2}{a^2_2}\int_{\R^4}\bar{v}W_\varepsilon$.
It holds that
\begin{equation}\label{f2}
\begin{aligned}
&\|\nabla \overline{W}_{\varepsilon,t_1}\|_{2}^2 =\|\nabla \widehat{W}_{\varepsilon,t_1}\|_{2}^2, \qquad \|\overline{W}_{\varepsilon,t_1}\|_{4}^4 = \|\widehat{W}_{\varepsilon,t_1}\|_{4}^4, \\
&\|\overline{W}_{\varepsilon,t_1}\|_{2}^{2} = (s_1)^{-2} \|\widehat{W}_{\varepsilon,t_1}\|_{2}^{2}, \quad \|\overline{W}_{\varepsilon,t_1}\|_{p}^{p} =(s_1)^{p-4}\|\widehat{W}_{\varepsilon,t_1}\|_{p}^{p}.\\
\end{aligned}
\end{equation}
Similarly,
\begin{equation}\label{f21}
\begin{aligned}
&\|\nabla \overline{V}_{\varepsilon,t_2}\|_{2}^2 =\|\nabla \widehat{V}_{\varepsilon,t_2}\|_{2}^2, \qquad \|\overline{V}_{\varepsilon,t_2}\|_{4}^4 = \|\widehat{V}_{\varepsilon,t_2}\|_{4}^4, \\
&\|\overline{V}_{\varepsilon,t_2}\|_{2}^{2} = (s_2)^{-2} \|\widehat{V}_{\varepsilon,t_2}\|_{2}^{2}, \quad \|\overline{V}_{\varepsilon,t_2}\|_{p}^{p} =(s_2)^{p-4}\|\widehat{V}_{\varepsilon,t_2}\|_{p}^{p}.\\
\end{aligned}
\end{equation}
In particular,
\begin{equation}\label{f3}
\|\overline{W}_{\varepsilon,t_1}\overline{V}_{\varepsilon,t_2}\|_{2}^{2}=s^2_1s^2_2\int_{\R^4}|\widehat{W}_{\varepsilon,t_1}(s_1 x)|^2|\widehat{V}_{\varepsilon,t_2}({s_2}x)|^2.
\end{equation}
We choose $s_1=\frac{\|\widehat{W}_{\varepsilon,t_1}\|_{2}}{a_1}$ and $s_2=\frac{\|\widehat{V}_{\varepsilon,t_2}\|_{2}}{a_2}$, then $(\overline{W}_{\varepsilon,t_1},\overline{V}_{\varepsilon,t_2})\in S(a_1)\times S(a_2)$. From Lemma \ref{lem3.2}, there exists $\tau_{\varepsilon}\in\R$ such that $\tau_{\varepsilon}\star(\overline{W}_{\varepsilon,t_1},\overline{V}_{\varepsilon,t_2})\in \mathcal{P}^-_{a_1,a_2}$.
Then,
\begin{equation}\label{f12}
\begin{aligned}
&e^{2\tau_{\varepsilon}}\big[\|\nabla \overline{W}_{\varepsilon,t_1}\|_{2}^2+\|\nabla \overline{V}_{\varepsilon,t_2}\|_{2}^2\big]\\
&=e^{4\tau_{\varepsilon}}\big[\mu_1\|\overline{W}_{\varepsilon,t_1}\|_{4}^4+\mu_2\|\overline{V}_{\varepsilon,t_2}\|_{4}^4+2\beta\|\overline{W}_{\varepsilon,t_1}\overline{V}_{\varepsilon,t_2}\|^2_2\big]\\
&\quad +e^{p\gamma_p\tau_{\varepsilon}}\big[\alpha_1\gamma_p\|\overline{W}_{\varepsilon,t_1}\|_{p}^{p}+\alpha_2\gamma_p\|\overline{V}_{\varepsilon,t_2}\|_{p}^{p}\big].
\end{aligned}
\end{equation}
Since $(\bar{u},\bar{v})\in \mathcal{P}^+_{a_1,a_2}$, we have that $\tau_{\varepsilon}>0$ for $t_1=t_2=0$. By \eqref{f1}-\eqref{f11} and \eqref{f12}, we get $\tau_{\varepsilon}\to -\infty$ as $t_1,t_2\to +\infty$ uniformly for $\varepsilon>0$ sufficiently small. It follows from Lemma \ref{lem3.2} that $\tau_{\varepsilon}$ is unique and continuous for $(\overline{W}_{\varepsilon,t_1},\overline{V}_{\varepsilon,t_2})$, and then there exist $t_{1}(\varepsilon),t_{2}(\varepsilon)>0$ such that $\tau_{\varepsilon}=0$.
Therefore, we have
\begin{equation*}
 m^-(a_1,a_2)\le \sup_{t_1,t_2\ge 0}I\big(\overline{W}_{\varepsilon,t_1},\overline{V}_{\varepsilon,t_2}\big).
\end{equation*}
By \eqref{f1}--\eqref{f3}, there exists $t_0>0$ such that
\begin{equation*}
\begin{aligned}
I\big(\overline{W}_{\varepsilon,t_1},\overline{V}_{\varepsilon,t_2}\big)&=\frac{1}{2}\big(\|\nabla \widehat{W}_{\varepsilon,t_1}\|_{2}^2+\|\nabla \widehat{V}_{\varepsilon,t_2}\|_{2}^2\big)-\frac{1}{4}\big( \mu_1\|\widehat{W}_{\varepsilon,t_1}\|_{4}^4+\mu_2\|\widehat{V}_{\varepsilon,t_2}\|_{4}^4\\
&\quad+2\beta s^2_1s^2_2\|\widehat{W}_{\varepsilon,t_1}(s_1x)\widehat{V}_{\varepsilon,t_2}({s_2}x)\|^2_2\big)
-\frac{\alpha_1}{p}s^{p\gamma_p-p}_1\|\widehat{W}_{\varepsilon,t_1}\|_{p}^{p}\\
&\quad -\frac{\alpha_2}{p}s_2^{p\gamma_p-p}\|\widehat{V}_{\varepsilon,t_2}\|_{p}^{p}\\
%&\quad-2\beta\Big(s^2_1s^2_2\|\widehat{W}_{\varepsilon,t_1}(s_1x)\widehat{V}_{\varepsilon,t_2}({s_2}x)\|^2_2-\|\widehat{W}_{\varepsilon,t_1}(x)\widehat{V}_{\varepsilon,t_2}\|^2_2\Big)\\
&\le m^+(a_1,a_2)+\frac{k_1+k_2}{4}\mathcal{S}^2-\sigma,
\end{aligned}
\end{equation*}
for $0<t_1,t_2<\frac{1}{t_0}$ and $t_1,t_2>t_0$ with $\sigma>0$.
By Lemma \ref{lem2.12}, we have
\begin{equation*}
\int_{\R^4}\bar{u}W_\varepsilon= O\Big(\varepsilon^3\int^{\frac{1}{\varepsilon}}_{1}\frac{|r|^3}{1+|r|^2}\Big)= O(\varepsilon),\ %\ \text{and}\ \ \int_{\R^4}\bar{u}W^3_\varepsilon=O\Big( \varepsilon\int^{\frac{1}{\varepsilon}}_{1}\frac{|r|^3}{(1+|r|^2)^3}\Big)= O(\varepsilon).
\end{equation*}
and
\begin{equation*}
\int_{\R^4}\bar{u}W^3_\varepsilon=O\big( \varepsilon\int^{\frac{1}{\varepsilon}}_{1}\frac{|r|^3}{(1+|r|^2)^3}\big)= O(\varepsilon).
\end{equation*}
Moreover,
\begin{equation}\label{x11}
s^2_1=\frac{\|\widehat{W}_{\varepsilon,t_1}\|^2_{2}}{a^2_1}=1+\frac{2t_1}{a^2_1}\int_{\R^4}\bar{u}W_\varepsilon+\frac{t^2_1}{a^2_1}\|W_\varepsilon\|^2_2=1+O(\varepsilon),
\end{equation}
and
\begin{equation}\label{x21}
s^2_2=\frac{\|\widehat{V}_{\varepsilon,t_2}\|^2_{2}}{a^2_2}=1+\frac{2t_2}{a^2_2}\int_{\R^4}\bar{v}W_\varepsilon+\frac{t^2_2}{a^2_2}\|W_\varepsilon\|^2_2=1+O(\varepsilon),
\end{equation}
for $\frac{1}{t_0}<t_1,t_2<t_0$.
%Then, we can deduce that
%\begin{equation*}
%\begin{aligned}
%2\beta\Big(s^2_1s^2_2\|\widehat{W}_{\varepsilon,t_1}(s_1x)\widehat{V}_{\varepsilon,t_2}({s_2}x)\|^2_2-\|\widehat{W}_{\varepsilon,t_1}(x)\widehat{V}_{\varepsilon,t_2}\|^2_2\Big)
%\le C\int_{\R^4}|\widehat{V}_{\varepsilon,t_2}|^2=O(\varepsilon^2|\ln\varepsilon|).???????
%\end{aligned}
%\end{equation*}
By direct calculations, we use the fact that $(a+b)^4\ge a^4+b^4+4a^3b+4ab^{3}$ for all $a,b\ge 0$, and $(\bar{u},\bar{v})$ is a positive solution of \eqref{eq1.1} for $\lambda_1,\lambda_2>0$, then
\begin{equation*}
\begin{aligned}
&I\big(\overline{W}_{\varepsilon,t_1},\overline{V}_{\varepsilon,t_2}\big)\\
&\le\frac{1}{2}\big[\|\nabla \widehat{W}_{\varepsilon,t_1}\|_{2}^2+\|\nabla \widehat{V}_{\varepsilon,t_2}\|_{2}^2\big]-\frac{1}{4}\big[ \mu_1\|\widehat{W}_{\varepsilon,t_1}\|_{4}^4+\mu_2\|\widehat{V}_{\varepsilon,t_2}\|_{4}^4\big]\\
&\quad-\frac{\alpha_1}{p}s^{p\gamma_p-p}_1\|\widehat{W}_{\varepsilon,t_1}\|_{p}^{p}\\
&\quad-\frac{\alpha_2}{p}s_2^{p\gamma_p-p}\|\widehat{V}_{\varepsilon,t_2}\|_{p}^{p}-\frac{\beta}{2}s^2_1s^2_2\|\widehat{W}_{\varepsilon,t_1}(s_1x)\widehat{V}_{\varepsilon,t_2}({s_2}x)\|^2_2\\
&\le I(\bar{u},\bar{v})-\frac{\beta}{2}\Big(s^2_1s^2_2\int_{\R^4}|\widehat{W}_{\varepsilon,t_1}(s_1 x)|^2|\widehat{V}_{\varepsilon,t_2}(s_2 x)|^2-\int_{\R^4}|\widehat{W}_{\varepsilon,t_1}( x)|^2|\widehat{V}_{\varepsilon,t_2}(x)|^2\Big)\\
&\quad+I(t_1W_\varepsilon,t_2W_\varepsilon)-\mu_1\int_{\R^4}\bar{u}(t_1W_\varepsilon)^3-\mu_2\int_{\R^4}\bar{v}(t_2W_\varepsilon)^3\\
&\quad+\int_{\R^4}[t_1\nabla\bar{u}+t_2\nabla\bar{v}]\nabla W_\varepsilon\\
&\quad-\int_{\R^4}\big[t_1\alpha_1|\bar{u}|^{p-1}+t_2\alpha_2|\bar{v}|^{p-1}+t_1\mu_1|\bar{u}|^3+t_2\mu_2|\bar{v}|^3+t_1\beta|\bar{v}|^2\bar{u}+t_2\beta|\bar{u}|^2\bar{v}\big]W_{\varepsilon}\\
&\quad+\frac{\alpha_1(1-\gamma_p)}{a^2_1}\|\widehat{W}_{\varepsilon,t_1}\|^p_p\int_{\R^4}\bar{u}t_1W_\varepsilon\\
&\quad+\frac{\alpha_2(1-\gamma_p)}{a^2_2}\|\widehat{V}_{\varepsilon,t_2}\|^p_p\int_{\R^4}\bar{v}t_2W_\varepsilon+O(\varepsilon^2|\ln\varepsilon|)\\
&\le m^+(a_1,a_2)+I(t_1W_\varepsilon,t_2W_\varepsilon)-t_1\big(\lambda_1- \frac{\alpha_1(1-\gamma_p)}{a^2_1}\|\widehat{W}_{\varepsilon,t_1}\|^p_p \big)\int_{\R^4}\bar{u}W_\varepsilon\\
&\quad-t_2\big(\lambda_2-\frac{\alpha_2(1-\gamma_p)}{a^2_2}\|\widehat{V}_{\varepsilon,t_2}\|^p_p \big)\int_{\R^4}\bar{v}W_\varepsilon+O(\varepsilon^2|\ln\varepsilon|)-O(\varepsilon)\\
&\quad -\frac{\beta}{2}\Big(s^2_1s^2_2\int_{\R^4}|\widehat{W}_{\varepsilon,t_1}(s_1 x)|^2|\widehat{V}_{\varepsilon,t_2}(s_2 x)|^2-\int_{\R^4}|\widehat{W}_{\varepsilon,t_1}( x)|^2|\widehat{V}_{\varepsilon,t_2}(x)|^2\Big).\\
%&\le m(a_1,a_2)+I(t_1W_\varepsilon,t_2W_\varepsilon)+C\|W_\varepsilon\|^2_2+O(\varepsilon^2|\ln\varepsilon|).\\
%&\le m(a_1,a_2)+\frac{k_1+k_2}{4}\mathcal{S}^2-O(\varepsilon)+O(\varepsilon^{4-p})\\
%&<m(a_1,a_2)+\frac{k_1+k_2}{4}\mathcal{S}^2
\end{aligned}
\end{equation*}
It follows from Lemma \ref{lem2.2} that
\begin{equation*}
\begin{aligned}
I(t_1W_\varepsilon,t_2W_\varepsilon)&=\frac{t^2_1+t^2_2}{2}\|\nabla W_{\varepsilon}\|_{2}^2-\frac{1}{4}\|W_{\varepsilon}\|_{4}^4 \big(\mu_1t^4_1+\mu_2t^4_2+2\beta t^2_1t^2_2\big)-O(\varepsilon^{4-p})\\
&\le \frac{k_1+k_2}{4}\mathcal{S}^2-O(\varepsilon^{4-p}),
\end{aligned}
\end{equation*}
where $k_1=\frac{\beta-\mu_2}{\beta^2-\mu_1\mu_2}$ and $k_2=\frac{\beta-\mu_1}{\beta^2-\mu_1\mu_2}$. Since $(\bar{u},\bar{v})$ is the ground state solution of \eqref{eq1.1}-\eqref{eq1.11},
\begin{equation*}
\lambda_1a^2_1+\lambda_2a^2_2=\alpha_1(1-\gamma_p)\|\bar{u}\|^p_p+\alpha_2(1-\gamma_p)\|\bar{v}\|^p_p.
\end{equation*}
The assumption $\frac{t_1}{a^2_1}\int_{\R^4}\bar{u}W_\varepsilon=\frac{t_2}{a^2_2}\int_{\R^4}\bar{v}W_\varepsilon$ is applied to ensure that
\begin{equation*}
\begin{aligned}
&\big(\lambda_1a^2_1-\alpha_1(1-\gamma_p)\|\widehat{W}_{\varepsilon,t_1}\|^p_p \big)\frac{t_1}{a^2_1}\int_{\R^4}\bar{u}W_\varepsilon+\big(\lambda_2a^2_2-\alpha_2(1-\gamma_p)\|\widehat{V}_{\varepsilon,t_2}\|^p_p\big)\frac{t_2}{a^2_2}\int_{\R^4}\bar{v}W_\varepsilon\\ %&\le C\Big(\int_{\R^4}|\bar{u}|^{p-1}W_\varepsilon+|\bar{u}|^{p-2}|W_\varepsilon|^2+O(\varepsilon^{4-p})\Big)\int_{\R^4}\bar{u}W_\varepsilon+\Big(\int_{\R^4}|\bar{v}|^{q-1}W_\varepsilon+|\bar{v}|^{q-2}|W_\varepsilon|^2+O(\varepsilon^{4-q})\Big)\int_{\R^4}\bar{v}W_\varepsilon\\
&\le C\Big(\int_{\R^4}|\bar{u}|^{p-1}W_\varepsilon+O(\varepsilon^{4-p})\Big)\int_{\R^4}\bar{u}W_\varepsilon+C\Big(\int_{\R^4}|\bar{v}|^{p-1}W_\varepsilon+O(\varepsilon^{4-p})\Big)\int_{\R^4}\bar{v}W_\varepsilon\\
&=O(\varepsilon^2).
\end{aligned}
\end{equation*}
%similarly,
%\begin{equation*}
%\frac{t_2}{a^2_2}\big(\lambda_2a^2_2-\alpha_2(1-\gamma_q)\|\widehat{V}_{\varepsilon,t_2}\|^q_q \big)\int_{\R^4}\bar{v}W_\varepsilon=O(\varepsilon^2).
%\end{equation*}
In view of \eqref{x11}, \eqref{x21} and using Fatou Lemma, we have
\begin{equation}\label{z5}
\int_{\R^4}|\widehat{W}_{\varepsilon,t_1}( x)|^2|\widehat{V}_{\varepsilon,t_2}(x)|^2\le \liminf_{\varepsilon\to 0}s^2_1s^2_2\int_{\R^4}|\widehat{W}_{\varepsilon,t_1}(s_1 x)|^2|\widehat{V}_{\varepsilon,t_2}(s_2 x)|^2.
\end{equation}
Therefore, for $\varepsilon$ small enough, we have
\begin{equation*}
\begin{aligned}
I\big(\overline{W}_{\varepsilon,t_1},\overline{V}_{\varepsilon,t_2}\big)&\le m^+(a_1,a_2)+\frac{k_1+k_2}{4}\mathcal{S}^2-O(\varepsilon^{4-p})-O(\varepsilon)+O(\varepsilon^2|\ln\varepsilon|)\\
&<m^+(a_1,a_2)+\frac{k_1+k_2}{4}\mathcal{S}^2.
\end{aligned}
\end{equation*}
\end{proof}

\begin{Lem}\label{lem4.21}
Let $2<p<3$, $\mu_i,\alpha_i,a_i>0(i=1,2)$, $\beta\in\big(0,\min\{\mu_1,\mu_2\}\big)\cup\big(\max\{\mu_1,\mu_2\},\infty\big)$.
Then for $T(a_1,a_2)\le\gamma_{1}$,
\begin{equation*}
m^-(a_1,a_2)<\min\big\{m^-(a_1,0),m^-(0,a_2)\big\},
\end{equation*}
where $m^-(a_1,0)$ and $m^-(0,a_2)$ are defined in \eqref{c12}.
\end{Lem}

\begin{proof}
By Lemma \ref{lem2.1}, $m^-(a_1,0)$ can be achieved by $u^*\in S(a_1)$. We choose a proper test function $v\in S(a_2)$ such that $(u^*,s\star v)\in S(a_1)\times S(a_2)$ with $s\in\R$.
In \cite{JTT,WuZ}, it has been proved that
\begin{equation*}
m^-(a_1,0)=\inf_{u\in \mathcal{T}^{-}_{a_1,p,\mu_1,\alpha_1}} \mathcal{A}_{p,\mu_1,\alpha_1}(u).
\end{equation*}
From Lemma \ref{lem3.2}, there exists $t_s\in \R$ such that $t_s\star (u^*,s\star v)\in \mathcal{P}^-_{a_1,a_2}$.
Then, we have
\begin{equation}\label{o1}
\begin{aligned}
&\|\nabla u^*\|^2_2+\|\nabla(s\star v)\|^2_2\\
&=e^{(p\gamma_p-2)t_s}\big[\alpha_1\|u^*\|^p_p+\alpha_2\|s\star v\|^p_p\big]\\
&\quad +e^{2t_s}\big[\mu_1\|u^*\|^4_4+\mu_2\|s\star v\|^4_4+2\beta\|u^*(s\star v)\|^2_2\big].
\end{aligned}
\end{equation}
By \eqref{o1}, we obtain that $e^{t_s}$ is bounded as $s\to -\infty$. If $s\ll -1$, we deduce that
\begin{equation*}
\begin{aligned}
m^-(a_1,&a_2)=\inf_{(u,v)\in \mathcal{P}^-_{a_1,a_2}}I(u,v)\le I\big(t_s\star (u^*,s\star v)\big)\\
&\le \mathcal{A}_{p,\mu_1,\alpha_1}(t_s\star u^*)+\frac{e^{2t_s}}{2}\|\nabla s\star v\|^2_2\\
&\quad -\frac{e^{p\gamma_pt_s}}{p}\|s\star v\|^p_p-\frac{e^{4t_s}}{4}\|s\star v\|^4_4-\frac{\beta e^{4t_s}}{2}\|u^*(s\star v)\|^2_2\\
&<  \mathcal{A}_{p,\mu_1,\alpha_1}(u^*)=m^-(a_1,0),
\end{aligned}
\end{equation*}
because $p\gamma_p<2$. On the other hand, we similarly get that $m^-(a_1,a_2)<m^-(0,a_2)$. Hence, the proof is completed.
\end{proof}

With the help of Lemmas \ref{lem4.1}-\ref{lem4.21}, we can prove the existence of  the  second solution.
\begin{Lem}\label{lem4.2}
Let $2<p<3$, $\mu_i,\alpha_i,a_i>0(i=1,2)$, $\beta\in\big(0,\min\{\mu_1,\mu_2\}\big)\cup\big(\max\{\mu_1,\mu_2\},\infty\big)$.
If $T(a_1,a_2)\le\gamma_{1}$ and
\begin{equation*}
0<m^-(a_1,a_2)<\min \big\{m^+(a_1,a_2)+\frac{k_1+k_2}{4}\mathcal{S}^2,m^-(a_1,0),m^-(0,a_2)\big\},
\end{equation*}
where $k_1=\frac{\beta-\mu_2}{\beta^2-\mu_1\mu_2}$ and $k_2=\frac{\beta-\mu_1}{\beta^2-\mu_1\mu_2}$,
then $m^-(a_1,a_2)$ can be achieved by some function $(u^-_{a_1},v^-_{a_2})\in S(a_1)\times S(a_2)$ which is real valued, positive, radially symmetric and radially decreasing.
\end{Lem}
\begin{proof}
Since the proof is similar to that of Lemma \ref{lem3.4}, we only sketch it. Let $\{(u_n,v_n)\}$ $\subset \mathcal{P}^-_{a_1,a_2}$ be a minimizing sequence. %From Lemma \ref{lem2.14}, then by taking $(|u_n|,|v_n|)$ and adapting the Schwarz symmetrization to $(|u_n|,|v_n|)$ if necessary, we can obtain a new minimizing sequence (up to a subsequence),
We assume that $(u_n,v_n)$ are all real valued, nonnegative, radially symmetric and decreasing in $r=|x|$. Moreover, $\{(u_n,v_n)\}$ is bounded in $H^1(\R^4,\R^2)$, and $(u_n,v_n)\rightharpoonup (u,v)$ in $H^1(\R^4,\R^2)$.
%Since $P_{a_1,a_2}(u_n,v_n)\to 0$, we have
We need to prove that $u\neq 0$ or $v\neq0$.
\vskip1mm
{\bf Case 1.} If $u=0$ and $v=0$, then $\int_{\R^4}\alpha_1|u_n|^p\to 0$, $\int_{\R^4}\alpha_2|v_n|^p\to 0$,
 we have
\begin{equation*}
P_{a_1,a_2}(u_n,v_n)=||\nabla u_n||^2_2+\|\nabla v_n\|^2_2-\mu_1\|u_n\|^{4}_{4}-\mu_2\|v_n\|^{4}_{4}-
2\beta\|u_nv_n\|^2_2=o_n(1).
\end{equation*}
By \eqref{f4}, we have
\begin{equation}\label{f5}
\begin{aligned}
\sqrt{k_1+k_2}\mathcal{S}\big(\mu_1\|u_n\|^{4}_{4}+\mu_2\|v_n\|^{4}_{4}+2\beta\|u_nv_n\|^2_2\big)^{\frac{1}{2}}
&\le ||\nabla u_n||^2_2+\|\nabla v_n\|^2_2.\\
%&=\mu_1\|u_n\|^{4}_{4}+\mu_2\|v_n\|^{4}_{4}+2\beta\|u_nv_n\|^2_2+o_n(1).
\end{aligned}
\end{equation}
Then, we distinguish the two cases:
\begin{equation*}
\text{either}\ \ (i)\ \ ||\nabla u_n||^2_2+\|\nabla v_n\|^2_2\to 0\ \ \text{or}\ \ (ii) \ \ ||\nabla u_n||^2_2+\|\nabla v_n\|^2_2\to l>0.
\end{equation*}
If $(i)$ holds, we have that $I(u_n,v_n)\to 0$ which contradicts the fact that $m^-(a_1,a_2)>0$. If $(ii)$ holds, we deduce from \eqref{f5} that
\begin{equation*}
||\nabla u_n||^2_2+\|\nabla v_n\|^2_2\ge (k_1+k_2)\mathcal{S}^2+o_n(1).
\end{equation*}
Therefore,
\begin{equation*}
\begin{aligned}
m^-(a_1,a_2)+o_n(1)=I(u_n,v_n)&=\frac{1}{4}\int_{\R^4}|\nabla u_n|^2+|\nabla v_n|^2+o_n(1)\\
%&\qquad-\frac{1}{4}\int_{\R^4}\big(\mu_1|u_n|^4+\mu_2|v_n|^4+2\beta|u_n|^2|v_n|^2\big)\\
&\ge \frac{k_1+k_2}{4}\mathcal{S}^2+o_n(1),
\end{aligned}
\end{equation*}
this contradicts $m^-(a_1,a_2)<m^+(a_1,a_2)+\frac{k_1+k_2}{4}\mathcal{S}^2$.
\vskip3mm
{\bf Case 2.} If $u\neq0$ and $v=0$, then $v_n\to 0$ in $L^p(\R^4)$. Let $\bar{u}_n=u_n-u$, $\bar{u}_n\to 0$ in $L^p(\R^4)$. By the maximum principle (see \cite[Theorem 2.10]{hanq}), $u$ is a positive solution of \eqref{a2}, then $m^-(a_1,0)\le m^-(\|u\|_2,0)$.
By the Br\'{e}zis-Lieb Lemma \cite{WM}, we deduce that
\begin{equation*}
\begin{aligned}
&P_{a_1,a_2}(u_n,v_n)\\
&=||\nabla u_n||^2_2+\|\nabla v_n\|^2_2-\mu_1\|u_n\|^{4}_{4}-\mu_2\|v_n\|^{4}_{4}-\gamma_p\alpha_1\|u_n\|^p_p-2\beta\|u_nv_n\|^2_2+o_n(1)\\
%&=||\nabla \bar{u}_n||^2_2+\|\nabla v_n\|^2_2-\mu_1\|\bar{u}_n\|^{4}_{4}-\mu_2\|v_n\|^{4}_{4}-2\beta\|\bar{u}_nv_n\|^2_2\\
%&\quad+||\nabla u||^2_2-\mu_1\|u\|^{4}_{4}-\gamma_p\alpha_1\|u\|^p_p+o_n(1)\\
%&=P_{a_1,a_2}(\bar{u}_n,v_n)+P_{a_1,a_2}(u,0)+o_n(1)\\
&=||\nabla \bar{u}_n||^2_2+\|\nabla v_n\|^2_2-\mu_1\|\bar{u}_n\|^{4}_{4}-\mu_2\|v_n\|^{4}_{4}-2\beta\|\bar{u}_nv_n\|^2_2+o_n(1).
\end{aligned}
\end{equation*}
Again by \eqref{f4}, we have
\begin{equation*}
\begin{aligned}
\sqrt{k_1+k_2}\mathcal{S}\big(\mu_1\|\bar{u}_n\|^{4}_{4}+\mu_2\|v_n\|^{4}_{4}+2\beta\|\bar{u}_nv_n\|^2_2\big)^{\frac{1}{2}}
&\le ||\nabla \bar{u}_n||^2_2+\|\nabla v_n\|^2_2.\\
\end{aligned}
\end{equation*}
Similarly, we distinguish the two cases:
\begin{equation*}
\text{either}\ \ (i)\ \ ||\nabla \bar{u}_n||^2_2+\|\nabla v_n\|^2_2\to 0\ \ \text{or}\ \ (ii) \ \ ||\nabla \bar{u}_n||^2_2+\|\nabla v_n\|^2_2\to l>0.
\end{equation*}
If $(i)$ holds, we have $I(\bar{u}_n,v_n)\to 0$. %which contradicts the fact that $m^-(a_1,a_2)>0$.
Thus,
\begin{equation*}
\begin{aligned}
m^-(a_1,a_2)+o_n(1)&=I(u_n,v_n)=I(\bar{u}_n,v_n)+I(u,0)+o_n(1)\\
                   &=I(u,0)+o_n(1)\ge m^-(a_1,0).
\end{aligned}
\end{equation*}
This is a contradiction.
If $(ii)$ holds,
\begin{equation*}
\begin{aligned}
m^-(a_1,a_2)+o_n(1)&=I(u_n,v_n)=I(\bar{u}_n,v_n)+I(u,0)+o_n(1)\\
&\ge \frac{1}{4}\int_{\R^4}(|\nabla \bar{u}_n|^2+|\nabla v_n|^2)
%-\frac{1}{4}\int_{\R^4}\big(\mu_1|\bar{u}_n|^{4}+\mu_2|v_n|^{4}\\
%&\qquad+2\beta\|\bar{u}_nv_n\|^2_2\big)
%-\frac{1}{4}\int_{\R^4}\mu_1|\bar{u}_n|^{4}+\mu_2|v_n|^{4}+\frac{\beta}{2}|\bar{u}_nv_n|^2\\
+m^+(\|u\|_2,0)+o_n(1)\\
&\ge \frac{k_1+k_2}{4}\mathcal{S}^2+m^+(a_1,0)\\
&\ge \frac{k_1+k_2}{4}\mathcal{S}^2+m^+(a_1,a_2),\\
\end{aligned}
\end{equation*}
which contradicts our assumption on $m^-(a_1,a_2)$.
\vskip3mm
{\bf Case 3.} If $u=0$ and $v\neq0$.   The analysis similar to that in the proof of Case 2,
%shows that $m^-(a_1,a_2)\ge \frac{k_1+k_2}{4}\mathcal{S}^2+m(a_1,a_2)+o_n(1)$,
we also have a contradiction.
\vskip1mm
Therefore, $u\neq0$ and $v\neq0$. It remains to show that $m^-(a_1,a_2)$ is achieved.
%$\{(u_n,v_n)\}\subset H^1(\R^4,\R^2)$ converges strongly.
%Analysis similar to that in the proof of Lemma \ref{lem3.4} shows that \eqref{d3}-\eqref{d4} hold.
%Moreover, combining \eqref{d3} with \eqref{d4}, we have
%By the maximum principle, $u,v>0$. Then, from Lemma \ref{lem2.11}, we get $\lambda_1,\lambda_2>0$.
It follows from the same computations in \eqref{d1}-\eqref{d4} that there exist $\lambda_1,\lambda_2>0$ such that
\begin{equation}\label{d51}
\lambda_1a^2_1+ \lambda_2a^2_2=\lambda_1\|u\|^2_2+\lambda_2\|v\|^2_2.
\end{equation}
Since $\|u\|^2_2\le a^2_1$ and $\|v\|^2_2\le a^2_2$, it follows from \eqref{d51} that $\|u\|^2_2= a^2_1$ and $\|v\|^2_2= a^2_2$, and hence $(u,v)\in \mathcal{P}_{a_1,a_2}$. Let $(\bar{u}_n,\bar{v}_n)=(u_n-u,v_n-v)$, we obtain
\begin{equation*}
\begin{aligned}
o_n(1)&=P_{a_1,a_2}(u_n,v_n)=P_{a_1,a_2}(\bar{u}_n,\bar{v}_n)+P_{a_1,a_2}(u,v)+o_n(1)\\
&=||\nabla \bar{u}_n||^2_2+\|\nabla \bar{v}_n\|^2_2-\mu_1\|\bar{u}_n\|^{4}_{4}-\mu_2\|\bar{v}_n\|^{4}_{4}-2\beta\|\bar{u}_n\bar{v}_n\|^2_2+o_n(1).\\
\end{aligned}
\end{equation*}
Here again we have  the two cases:
\begin{equation*}
\text{either}\ \ (i)\ \ ||\nabla \bar{u}_n||^2_2+\|\nabla \bar{v}_n\|^2_2\to 0\ \ \text{or}\ \ (ii) \ \ ||\nabla \bar{u}_n||^2_2+\|\nabla \bar{v}_n\|^2_2\to l>0.
\end{equation*}
If $(ii)$ holds,
\begin{equation*}
\begin{aligned}
m^-(a_1,a_2)+o_n(1)&=I(u_n,v_n)=I(\bar{u}_n,\bar{v}_n)+I(u,v)+o_n(1)\\
&\ge \frac{1}{4}\int_{\R^4}(|\nabla \bar{u}_n|^2+|\nabla \bar{v}_n|^2)+m^+(a_1,a_2)+o_n(1)\\
&\ge\frac{k_1+k_2}{4}\mathcal{S}^2+ m^+(a_1,a_2)+o_n(1),
\end{aligned}
\end{equation*}
which is a contradiction. Then $(i)$ holds, we conclude that $I(u,v)=m^-(a_1,a_2)$ and $(u_n,v_n)\to (u,v)$ in $H^1(\R^4,\R^2)$. Therefore, we have proved that $m^-(a_1,a_2)$ can be attained by some $(u^-_{a_1},v^-_{a_2})$ which is real valued, positive, radially symmetric and decreasing in $r=|x|$.
\end{proof}

%\section{Asymptotic behavior of solutions}

In the following, we study the asymptotic behavior of the second solution for \eqref{eq1.1}-\eqref{eq1.11} as $(a_1,a_2)\to (0,0)$.

\begin{Lem}\label{lem4.4}
Let $2<p<3$, $\mu_i,\alpha_i,a_i>0(i=1,2)$, $\beta\in\big(0,\min\{\mu_1,\mu_2\}\big)\cup\big(\max\{\mu_1,\mu_2\},\infty\big)$. If $T(a_1,a_2)\le\gamma_{1}$, then
\begin{equation*}
I(u^-_{a_1},v^-_{a_2})=m^-(a_1,a_2)\to\frac{k_1+k_2}{4}\mathcal{S}^2\ \ \text{as} \ \ (a_1,a_2)\to (0,0).
\end{equation*}
Moreover, %if $0<\beta<\min\{\mu_1,\mu_2\}$ or $\beta>\max\{\mu_1,\mu_2\}$,
there exists $\varepsilon_{a_1,a_2}>0$ such that
\begin{equation*}
\big(\varepsilon_{a_1,a_2} u^-_{a_1}(\varepsilon_{a_1,a_2} x),\varepsilon_{a_1,a_2} v^-_{a_2}(\varepsilon_{a_1,a_2} x)\big)\to (\sqrt{k_1}U_{\varepsilon_0},\sqrt{k_2}U_{\varepsilon_0})
\end{equation*}
in $D^{1,2}(\R^4,\R^2)$, for some $\varepsilon_0>0$ as $(a_1,a_2)\to (0,0)$ and $a_1\sim a_2$ up to a subsequence, where %$\varepsilon_{a_1,a_2} =o(a_1)$ and $\varepsilon_{a_1,a_2}=o_(a_2)$ as $(a_1,a_2)\to (0,0)$,
$k_1=\frac{\beta-\mu_2}{\beta^2-\mu_1\mu_2}$ and $k_2=\frac{\beta-\mu_1}{\beta^2-\mu_1\mu_2}$.
\end{Lem}
\begin{proof}
Since $P(u^-_{a_1},v^-_{a_2})=0$, using the Gagliardo-Nirenberg inequality, we get
\begin{equation*}
\begin{aligned}
I(u^-_{a_1},v^-_{a_2})&=\frac{1}{4}\int_{\R^4}|\nabla u^-_{a_1}|^2+|\nabla v^-_{a_2}|^2-\frac{\alpha_1}{p}\big(1-\frac{p\gamma_p}{4}\big)\int_{\R^4}|u^-_{a_1}|^p\\
&\quad -\frac{\alpha_2}{p}\big(1-\frac{p\gamma_p}{4}\big)\int_{\R^4}|v^-_{a_2}|^p,\\
&\ge \frac{1}{4}\big(\|\nabla u^-_{a_1}\|^2_2+\|\nabla v^-_{a_2}\|^2_2\big)-\big(1-\frac{p\gamma_p}{4}\big)D_2\|\nabla u^-_{a_1}\|^{2(p-2)}_2\\
&\quad -\big(1-\frac{p\gamma_p}{4}\big)D_3\|\nabla v^-_{a_2}\|^{2(p-2)}_2.\\
\end{aligned}
\end{equation*}
By Lemma \ref{lem4.1}, $I(u^-_{a_1},v^-_{a_2})<m^+(a_1,a_2)+\frac{k_1+k_2}{4}\mathcal{S}^2$, we have that $\{(u^-_{a_1},v^-_{a_2})\}\subset H^1(\R^4,\R^2)$ is bounded. Using the Gagliardo-Nirenberg inequality again,
\begin{align*}
&\int_{\R^4}|u^-_{a_1}|^p\le |\mathcal{C}_p|^{p}a_1^{4-p}\|\nabla u^-_{a_1}\|^{2(p-2)}_2\to 0,\\
& \int_{\R^4}|v^-_{a_2}|^p\le |\mathcal{C}_p|^{p}a_2^{4-p}\|\nabla v^-_{a_2}\|^{2(p-2)}_2\to 0,
\end{align*}
as $(a_1,a_2)\to (0,0)$. Therefore, it follows from $P(u^-_{a_1},v^-_{a_2})=0$ that
\begin{equation*}
\|\nabla u^-_{a_1}\|^2_2+\|\nabla v^-_{a_2}\|^2_2=\mu_1\|u^-_{a_1}\|^4_4+\mu_2\|v^-_{a_2}\|^4_4+2\beta\|u^-_{a_1}v^-_{a_2}\|^2_2+o(1).
\end{equation*}
From \eqref{f4}, we have
\begin{equation*}
\begin{aligned}
\sqrt{k_1+k_2}\mathcal{S}\big(\mu_1\|u^-_{a_1}\|^{4}_{4}+\mu_2\|v^-_{a_2}\|^{4}_{4}+2\beta\|u^-_{a_1}v^-_{a_2}\|^2_2\big)^{\frac{1}{2}}
&\le ||\nabla u^-_{a_1}||^2_2+\|\nabla v^-_{a_2}\|^2_2.\\
%&=\mu_1\|u_n\|^{4}_{4}+\mu_2\|v_n\|^{4}_{4}+2\beta\|u_nv_n\|^2_2+o_n(1).
\end{aligned}
\end{equation*}
Thus, we distinguish the two cases:
\begin{align*}
\text{either}\ \ (i)\ \ ||\nabla u^-_{a_1}||^2_2+\|\nabla v^-_{a_2}\|^2_2\to 0\ \ \text{or}\ \ (ii) \ \ ||\nabla u^-_{a_1}||^2_2+\|\nabla v^-_{a_2}\|^2_2\to l>0.
\end{align*}
We claim that $(i)$ is impossible. In fact, since $(u^-_{a_1},v^-_{a_2})\in \mathcal{P}^-_{a_1,a_2}$, we get
\begin{align*}
2\|\nabla u^-_{a_1}\|^2_2+2\|\nabla v^-_{a_2}\|^2_2&<4\big[\mu_1\|u^-_{a_1}\|^4_4+\mu_2\|v^-_{a_2}\|^4_4+2\beta\|u^-_{a_1}v^-_{a_2}\|^2_2\big]\\
&\quad+p\gamma^2_p\alpha_1\|u^-_{a_1}\|^p_p+p\gamma^2_p\alpha_2\|v^-_{a_2}\|^p_p.
\end{align*}
%Without loss of generality, we assume $p\le q$,
Then
\begin{equation*}
\begin{aligned}
\|\nabla u^-_{a_1}\|^2_2+\|\nabla v^-_{a_2}\|^2_2&<\frac{4-p\gamma_p}{2-p\gamma_p}\big[\mu_1\|u^-_{a_1}\|^4_4+\mu_2\|v^-_{a_2}\|^4_4+2\beta\|u^-_{a_1}v^-_{a_2}\|^2_2\big]\\
&\le \frac{4-p\gamma_p}{2-p\gamma_p}\frac{1}{(k_1+k_2)\mathcal{S}^2}\big[\|\nabla u^-_{a_1}\|^2_2+\|\nabla v^-_{a_2}\|^2_2\big]^2.
\end{aligned}
\end{equation*}
The claim is proved. Therefore, $(ii)$ holds, and $l\ge(k_1+k_2)\mathcal{S}^2$. Moreover, by \cite{JTT}, $m^-(a_1,0)\to \frac{\mu_1}{4}\mathcal{S}^2$ as $a_1\to 0$, and $m^-(0,a_2)\to \frac{\mu_2}{4}\mathcal{S}^2$ as $a_2\to 0$. We obtain that
\begin{equation*}
\begin{aligned}
\frac{k_1+k_2}{4}\mathcal{S}^2+m^+(a_1,a_2)&>m^-(a_1,a_2)=I(u^-_{a_1},v^-_{a_2})\\
&=\frac{1}{4}\int_{\R^4}|\nabla u^-_{a_1}|^2+|\nabla v^-_{a_2}|^2-\frac{\alpha_1}{p}\big(1-\frac{p\gamma_p}{4}\big)\int_{\R^4}|u^-_{a_1}|^p\\
&\quad-\frac{\alpha_2}{p}\big(1-\frac{p\gamma_p}{4}\big)\int_{\R^4}|v^-_{a_2}|^p\\
&\ge \frac{1}{4}\int_{\R^4}|\nabla u^-_{a_1}|^2+|\nabla v^-_{a_2}|^2+o(1)\\
&\ge \frac{k_1+k_2}{4}\mathcal{S}^2+o(1).
\end{aligned}
\end{equation*}
Taking into account that $m^+(a_1,a_2)\to 0$ as $(a_1,a_2)\to (0,0)$.
%\vskip1mm
%Since $\big\{(u^-_{a_1},v^-_{a_2})\big\}$ is bounded in $H^1(\R^4,\R^2)$, $(u^-_{a_1},v^-_{a_2})\rightharpoonup (u^-_0,v^-_0)$ in $H^1(\R^4,\R^2)$ as $(a_1,a_2)\to (0,0)$.
Thus, we know that
\begin{align*}
&||\nabla u^-_{a_1}||^2_2+\|\nabla v^-_{a_2}\|^2_2\to (k_1+k_2)\mathcal{S}^2,\\
& \mu_1\|u^-_{a_1}\|^{4}_{4}+\mu_2\|v^-_{a_2}\|^{4}_{4}+2\beta\|u^-_{a_1}v^-_{a_2}\|^2_2\to (k_1+k_2)\mathcal{S}^2,
\end{align*}
as $(a_1,a_2)\to (0,0)$. It follows that, up to a subsequence, $\{(u^-_{a_1},v^-_{a_2})\}$ is a minimizing sequence of the minimizing problem \eqref{f4}. From Lemma \ref{lem3.2}, $(u^-_{a_1},v^-_{a_2})$ is radially symmetric. By \cite[Theorem 1.41]{WM} or \cite[Lemma 3.5]{LL}, up to a subsequence, there exists $\sigma_1:=\sigma_1(a_1,a_2)$ such that for some $\varepsilon_0>0$,
\begin{equation}\label{m4}
u_1(x)=\sigma_1 u^-_{a_1}(\sigma_1 x)\to \sqrt{k_1}U_{\varepsilon_0}\ \ \text{and} \ \ v_1(x)=\sigma_1 v^-_{a_2}(\sigma_1 x)\to \sqrt{k_2}U_{\varepsilon_0}%\ \ \text{in}\ D^{1,2}(\R^4)\times D^{1,2}(\R^4),
\end{equation}
in $D^{1,2}(\R^4,\R^2)$, as $(a_1,a_2)\to (0,0)$. Since $\|U_{\varepsilon_0}\|^2_2\not\in L^2(\R^4)$, $\|u_1\|^2_2=\frac{a^2_1}{\sigma^2_1}$ and $\|v_1\|^2_2=\frac{a^2_2}{\sigma^2_1}$, by the
%\vskip0.5mm
Fatou lemma, we have $\sigma_1=o(a_1)$, $\sigma_1=o(a_2)$ and $\|u_1\|^2_2\to +\infty$, $\|v_1\|^2_2\to +\infty$ as $(a_1,a_2)\to (0,0)$.
%we choose $\sigma$ such that $\sigma=\sigma(a_1,a_2)\to 0$ as $(a_1,a_2)\to (0,0)$.

\vskip1mm
Therefore, we get that $(u_1,v_1)$ satisfies
\begin{equation}\label{z2}
\begin{cases}
-\Delta u_1+\lambda_{1}\sigma^2_1 u_1=\mu_1u^3_1+\alpha_1\sigma^{4-p}_1|u_1|^{p-2}u_1+\beta v^2_1u_1\quad&\hbox{in}~\R^4,\\
-\Delta v_1+\lambda_{2}\sigma^2_1 v_1=\mu_2v^3_1+\alpha_2\sigma^{4-p}_1|v_1|^{p-2}v_1+\beta u^2_1v_1\quad&\hbox{in}~\R^4,\\
\end{cases}
\end{equation}
where $\lambda_{1}:=\lambda_{1,a_1,a_2}$ and $\lambda_{2}:=\lambda_{2,a_1,a_2}$.
We note that such an additional assumption is appropriate for the case $a_{1}\sim a_{2}$ as $(a_{1}, a_{2}) \to (0,0)$. %then $\sigma_1\sim \sigma_2$. %$\frac{\sigma^2_1}{\sigma^2_2}=O(1)$, and $\frac{\sigma^2_2}{\sigma^2_1}=O(1)$.
By $P_{a_1,a_2}(u^-_{a_1},v^-_{a_2})=0$, we have
\begin{equation*}
\lambda_{1}\sigma^2_1\|u_1\|^2_2+\lambda_{2} \sigma^2_1 \|v_1\|^2_2=(1-\gamma_p)\alpha_1\sigma^{4-p}_1\|u_1\|^p_p+(1-\gamma_p)\alpha_2\sigma^{4-p}_1\|v_1\|^p_p.
\end{equation*}
%By the Gagliardo-Nirenberg inequality, we have
%\begin{equation*}
%\lambda_{1}\sigma^2_1\|u_1\|^2_2+\lambda_{2} \sigma^2_2 \|v_1\|^2_2\lesssim \alpha_1\sigma^{4-p}_1\|u_1\|^{4-p}_2+\alpha_2\sigma^{4-p}_2\|v_1\|^{4-p}_2.
%\end{equation*}
Since $2<p<3$, it follows that $\lambda_{1}\sigma^2_1=o(\sigma^{4-p}_1)$ and $\lambda_{2}\sigma^{2}_1=o(\sigma^{4-p}_1)$ as $(a_1,a_2)\to (0,0)$ and $a_1\sim a_2$.
This follows by the same methods as in Lemma \ref{lem2.12}, hence
\begin{equation}\label{e11}
(u_1,v_1)\to (\sqrt{k_1}U_{\varepsilon_0},\sqrt{k_2}U_{\varepsilon_0})\ \ \text{in}\ C^2_{loc}(\R^4)\times C^2_{loc}(\R^4).
\end{equation}
In what follows, we follow the ideas in \cite{WuZ} to drive a uniformly upper bound of $(u_1,v_1)$. We define
\begin{equation*}
\bar{u}_1:=\frac{1}{u_1(0)}u_1\big(\sqrt{u_1(0)}x\big),\quad \bar{v}_1:=\frac{1}{v_1(0)}v_1\big(\sqrt{v_1(0)}x\big).
\end{equation*}
Since $(u_1,v_1)$ is radial, a direct computation shows that $(\bar{u}_1,\bar{v}_1)$ solve
\begin{equation}\label{e1}
\begin{cases}
-\bar{u}''_1-\frac{3}{r}\bar{u}'_1=f(\bar{u}_1,\bar{v}_1)\quad&\hbox{in}~\R^4,\\
-\bar{v}''_1-\frac{3}{r}\bar{v}'_1=g(\bar{u}_1,\bar{v}_1)\quad&\hbox{in}~\R^4,\\
\end{cases}
\end{equation}
where
\begin{equation}\label{e2}
\begin{aligned}
f(\bar{u}_1,\bar{v}_1)&=-\lambda_{1}\sigma^2_1u_1(0)\bar{u}_1+\mu_1[u_1(0)]^3\bar{u}^3_1+\alpha_1\sigma^{4-p}_1[u_1(0)]^{p-1}|\bar{u}_1|^{p-2}\bar{u}_1\\
&\quad+\beta[v_1(0)]^{2}u_1(0)\bar{v}^2_1\bar{u}_1,
\end{aligned}
\end{equation}
and
\begin{equation*}
\begin{aligned}
g(\bar{u}_1,\bar{v}_1)&=-\lambda_{2}\sigma^2_1v_1(0)\bar{v}_1+\mu_2[v_1(0)]^3\bar{v}^3_1+\alpha_2\sigma^{4-p}_1[v_1(0)]^{p-1}|\bar{v}_1|^{p-2}\bar{v}_1\\
&\quad+\beta[u_1(0)]^{2}v_1(0)\bar{u}^2_1\bar{v}_1.
\end{aligned}
\end{equation*}
Let
\begin{equation*}
H_1(r)=r^4(\bar{u}'_1)^2+2r^3\bar{u}_1\bar{u}'_1+\frac{1}{2}r^4\bar{u}_1f(\bar{u}_1,\bar{v}_1),
\end{equation*}
and
\begin{equation*}
H_2(r)=r^4(\bar{v}'_1)^2+2r^3\bar{v}_1\bar{v}'_1+\frac{1}{2}r^4\bar{v}_1g(\bar{u}_1,\bar{v}_1).
\end{equation*}
By direct calculations, we have
\begin{equation*}
\begin{aligned}
H'_1(r)&=4r^3(\bar{u}'_1)^2+2r^4\bar{u}'_1\bar{u}''_1+6r^2\bar{u}_1\bar{u}'_1+2r^3[(\bar{u}'_1)^2+\bar{u}_1\bar{u}''_1]+2r^3\bar{u}_1f(\bar{u}_1,\bar{v}_1)\\
&\quad +\frac{1}{2}r^4\big[\bar{u}'_1f(\bar{u}_1,\bar{v}_1)+\bar{u}_1\frac{\partial f(\bar{u}_1,\bar{v}_1)}{\partial r}\big].\\
\end{aligned}
\end{equation*}
%\begin{equation*}
%\begin{aligned}
%H'_2(r)&=4r^3(\bar{v}'_1)^2+2r^4\bar{v}'_1\bar{v}''_1+6r^2\bar{v}_1\bar{v}'_1+2r^3[(\bar{v}'_1)^2+\bar{v}_1\bar{v}''_1]+2r^3\bar{v}_1g(\bar{u}_1,\bar{v}_1)\\
%&\quad +\frac{1}{2}r^4[\bar{v}'_1g(\bar{u}_1,\bar{v}_1)+\bar{v}_1\frac{\partial g(\bar{u}_1,\bar{v}_1)}{\partial r}],\\
%\end{aligned}
%\end{equation*}
From \eqref{e11}-\eqref{e2}, it follows that
\begin{equation*}
\begin{aligned}
H'_1(r)&=r^4u_1(0)\bar{u}_1\bar{u}'_1\big[\lambda_{1}\sigma^2_1-\frac{4-p}{2}\alpha_1\sigma^{4-p}_1[u_1(0)]^{p-2}|\bar{u}_1|^{p-2}\big]\\
&\quad+r^4\beta[v_1(0)]^2u_1(0) \bar{u}_1\bar{v}_1\big[\bar{v}'_1\bar{u}_1-\bar{u}'_1\bar{v}_1\big]\\
&=r^4\sigma^{4-p}_1u_1(0)\bar{u}_1\bar{u}'_1\big[o(1)-\frac{4-p}{2}\alpha_1|\sqrt{k_1}U_{\varepsilon_0}|^{p-2}\big]+o(1),
\end{aligned}
\end{equation*}
for $\varepsilon_0$ small enough.
%Similar to the proof of Lemma 5.1 in \cite{WuZ},
Moreover, $u_1, v_1$ exponentially decays to zero as $r\to +\infty$ (see \cite{WuZ}), then $\lim\limits_{r\to+\infty}H_1(r)=0$. Since $u_1,v_1>0$, $\bar{u}'_1,\bar{v}'_1<0$, there exists $r_1:=r(a_1,a_2)>0$, $H'_1(r_1)>0$ for $0<r<r_1$ and $H'_1(r)<0$ for $r>r_1$. Therefore, $H_1(r)>H_1(0)=0$ for all $r>0$. Let
\begin{equation*}
\Psi(r)=\frac{-\bar{u}'_1}{r|\bar{u}_1|^{2}}.
\end{equation*}
By \eqref{e1}, we deduce that
\begin{equation*}
\Psi'(r)=2r^{-5}\bar{u}^{-3}_1H_1(r)>0, \ \ \forall r>0.
\end{equation*}
It follows from \eqref{e1} once more that
\begin{equation}\label{h1}
\Psi(r)>\Psi(0)=-\bar{u}''_1(0)=\frac{d}{4},
\end{equation}
where
\begin{equation*}
d=-\lambda_1\sigma^2_1u_1(0)+\mu_1[u_1(0)]^3+\alpha_1\sigma^{4-p}_1[u_1(0)]^{p-1}+\beta[v_1(0)]^{2}u_1(0).
\end{equation*}
Let $Z(r)=\frac{1}{1+\frac{d}{8}r^2}$, then we have
\begin{equation}\label{h2}
\frac{-Z'(r)}{[Z(r)]^2}=\frac{d}{4}r.
\end{equation}
From \eqref{h1}-\eqref{h2}, we get that
\begin{equation*}
\frac{\bar{u}'_1(r)}{[\bar{u}_1(r)]^{2}}\le \frac{Z'(r)}{[Z(r)]^2}\ \ \ \ \forall r>0,
\end{equation*}
which together with \eqref{e11}, implies
\begin{equation}\label{n1}
u_1(r)\lesssim \frac{1}{1+r^2}\ \ \ \ \forall r>0,
\end{equation}
uniformly for $a_1,a_2>0$ sufficiently small and $a_1\sim a_2$.

\vskip1mm
Similarly, we can deduce that
\begin{equation}\label{n2}
v_1(r)\lesssim \frac{1}{1+r^2}\ \ \ \ \forall r>0,
\end{equation}
uniformly for $a_1,a_2>0$ sufficiently small and $a_1\sim a_2$.
Then for $2<p<3$, we get
\begin{equation*}
\|u_1\|^p_p\lesssim \int^{+\infty}_{0}\frac{r^3}{(1+r^2)^p}\lesssim 1\ \ \text{and}\ \ \|v_1\|^p_p\lesssim \int^{+\infty}_{0}\frac{r^3}{(1+r^2)^p}\lesssim 1.
\end{equation*}
By the Fatou lemma, $\|u_1\|^p_p\ge \|\sqrt{k_1}U_{\varepsilon_0}\|^p_p\gtrsim 1+o(1)$ and $\|v_1\|^p_p\ge \|\sqrt{k_2}U_{\varepsilon_0}\|^p_p\gtrsim 1+o(1)$. Thus, $\|u_1\|^p_p\sim 1$, $\|v_1\|^p_p\sim 1$, and
\begin{equation*}
\lambda_1a^2_1\sim \sigma^{4-p}_1\quad \text{and} \quad  \lambda_2a^2_2  \sim  \sigma^{4-p}_1,
\end{equation*}
where $(a_1,a_2)\to (0,0)$ and $a_1\sim a_2$.

\vskip1mm
Let us define
\begin{equation*}
V_\varepsilon(x)=\sqrt{k_1}U_{\varepsilon}\phi(R^{-1}_\varepsilon x)\ \  \text{and}\ \  \overline{V}_\varepsilon(x)=\sqrt{k_2}U_{\varepsilon}\phi(\bar{R}^{-1}_\varepsilon x),
\end{equation*}
where $U_\varepsilon$ is defined in \eqref{L1}, $\phi \in C_0^{\infty}(\R^4)$ is a radially decreasing cut-off function  with $\phi \equiv 1$ in $B_1$, $\phi \equiv 0$ in $\R^4 \backslash B_2$, and $R_\varepsilon, \bar{R}_\varepsilon$ are chosen such that $(V_\varepsilon,\overline{V}_\varepsilon)\in S(a_1)\times S(a_2)$. We choose $\varepsilon>0$ sufficiently small and then
\begin{equation}\label{m3}
a^2_1=\int_{\R^4}\big|\sqrt{k_1}U_{\varepsilon}\phi(R^{-1}_\varepsilon x)\big|^2\sim \varepsilon^2\int^{R_{\varepsilon}\varepsilon^{-1}}_{1}\frac{1}{r}\sim \varepsilon^2\ln(R_{\varepsilon}\varepsilon^{-1}),
\end{equation}
\begin{equation}\label{m31}
a^2_2=\int_{\R^4}\big|\sqrt{k_2}U_{\varepsilon}\phi(\bar{R}^{-1}_\varepsilon x)\big|^2\sim \varepsilon^2\int^{\bar{R}_{\varepsilon}\varepsilon^{-1}}_{1}\frac{1}{r}\sim \varepsilon^2\ln(\bar{R}_{\varepsilon}\varepsilon^{-1}),
\end{equation}
which implies $R_{\varepsilon}\varepsilon^{-1}\to +\infty$ and $\bar{R}_{\varepsilon}\varepsilon^{-1}\to +\infty$ as $\varepsilon\to 0$. From %Chapter III of \cite{SM},
we have
\begin{equation*}
\|\nabla V_\varepsilon\|^2_2+\|\nabla \overline{V}_\varepsilon\|^2_2=(k_1+k_2)\mathcal{S}^2+O\big((R_{\varepsilon}\varepsilon^{-1})^{-2}\big)+O\big((\bar{R}_{\varepsilon}\varepsilon^{-1})^{-2}\big),\ \
\end{equation*}
\begin{equation*}
\alpha_1\|V_\varepsilon\|^p_p+\alpha_2\| \overline{V}_\varepsilon\|^p_p+2\beta\|V_\varepsilon \overline{V}_\varepsilon\|^2_2=(k_1+k_2)\mathcal{S}^2+O\big((R_{\varepsilon}\varepsilon^{-1})^{-4}\big)+O\big((\bar{R}_{\varepsilon}\varepsilon^{-1})^{-4}\big),
\end{equation*}
and
\begin{equation*}
\|V_\varepsilon\|^p_p\sim \varepsilon^{4-p}, \ \ \ \ \|\overline{V}_\varepsilon\|^p_p\sim \varepsilon^{4-p},
\end{equation*}
for $\varepsilon$ small enough. By Lemma \ref{lem3.2}, there exists $t_0:=t_0(V_\varepsilon,\overline{V}_\varepsilon)>0$ such that $t_0\star(V_\varepsilon,\overline{V}_\varepsilon)\in \mathcal{P}^-_{a_1,a_2}$. For convenience, we replace $e^{t_{0}}$ by $t_1$, then
\begin{align*}
&||\nabla V_\varepsilon||^2_2+\|\nabla \overline{V}_\varepsilon\|^2_2\\
&=t^2_1\big[\mu_1\|V_\varepsilon\|^{4}_{4}+\mu_2\|\overline{V}_\varepsilon\|^{4}_{4}+2\beta\|V_\varepsilon \overline{V}_\varepsilon\|^2_2\big]+t^{p\gamma_p-2}_1\big[\gamma_p\alpha_1\|V_\varepsilon\|^p_p+\gamma_p\alpha_2\|\overline{V}_\varepsilon\|^p_p\big]
\end{align*}
and
\begin{equation*}
\begin{aligned}
2||\nabla V_\varepsilon||^2_2+2\|\nabla \overline{V}_\varepsilon\|^2_2&<4t^2_1\big[\mu_1\|V_\varepsilon\|^{4}_{4}+\mu_2\|\overline{V}_\varepsilon\|^{4}_{4}+2\beta\|V_\varepsilon \overline{V}_\varepsilon\|^2_2\big]\\
&\quad+t^{p\gamma_p-2}_1\big[p\gamma^2_p\alpha_1\|V_\varepsilon\|^p_p +p\gamma^2_p\alpha_2\|\overline{V}_\varepsilon\|^p_p\big].\\
\end{aligned}
\end{equation*}
Therefore, $\{t_1(V_\varepsilon,\overline{V}_\varepsilon)\}$ is uniformly bounded and bounded from below away from $0$ for all $\varepsilon, a_1, a_2$ sufficiently small. Hence,
\begin{align*}
&(k_1+k_2)\mathcal{S}^2[1-t^2_1]\\
&=\big[\gamma_p\alpha_1\|V_\varepsilon\|^p_p+\gamma_p\alpha_2\|\overline{V}_\varepsilon\|^p_p\big]t^{p\gamma_p-2}_1+O\big((R_{\varepsilon}\varepsilon^{-1})^{-2}\big)+O\big((\bar{R}_{\varepsilon}\varepsilon^{-1})^{-2}\big).
\end{align*}
It follows that
\begin{equation}\label{l3}
t_1=1-(1+o(1))\frac{\gamma_p\alpha_1\|V_\varepsilon\|^p_p+\gamma_p\alpha_2\|\overline{V}_\varepsilon\|^p_p+O\big((R_{\varepsilon}\varepsilon^{-1})^{-2}+O\big((\bar{R}_{\varepsilon}\varepsilon^{-1})^{-2}\big)}{2(k_1+k_2)\mathcal{S}^2}.%-\frac{\gamma_q\alpha_2\|W_\varepsilon\|^q_q}{2(k_1+k_2)\mathcal{S}^2}+
\end{equation}
From \eqref{l3}, we can infer that
\begin{equation}\label{l4}
\begin{aligned}
m^-(a_1,a_2)&\le I\big(t_0\star (V_\varepsilon,\overline{V}_\varepsilon)\big)\\
&\le\frac{t^2_1}{4}(k_1+k_2)\mathcal{S}^2-\frac{\alpha_1}{p}(1-\frac{p\gamma_p}{4})(t_1)^{p\gamma_p}\|V_{\varepsilon}\|^p_p\\
&\quad -\frac{\alpha_2}{p}(1-\frac{p\gamma_p}{4})(t_1)^{p\gamma_p}\|\overline{V}_{\varepsilon}\|^p_p+O\big((R_{\varepsilon}\varepsilon^{-1})^{-2}\big)+O\big((\bar{R}_{\varepsilon}\varepsilon^{-1})^{-2}\big)\\
&=\frac{k_1+k_2}{4}\mathcal{S}^2-\frac{\alpha_1\gamma_p\|V_{\varepsilon}\|^p_p+\alpha_2\gamma_p\|\overline{V}_{\varepsilon}\|^p_p}{4}-\frac{\alpha_1}{p}(1-\frac{p\gamma_p}{4})\|V_{\varepsilon}\|^p_p\\
&\quad\ \ -\frac{\alpha_2}{p}(1-\frac{p\gamma_p}{4})\|\overline{V}_{\varepsilon}\|^p_p+O\big((R_{\varepsilon}\varepsilon^{-1})^{-2}\big)+O\big((\bar{R}_{\varepsilon}\varepsilon^{-1})^{-2}\big)\\
&=\frac{k_1+k_2}{4}\mathcal{S}^2-\frac{\alpha_1}{p}\|V_{\varepsilon}\|^p_p-\frac{\alpha_2}{p}\|\overline{V}_{\varepsilon}\|^p_p\\
&\quad +O\big((R_{\varepsilon}\varepsilon^{-1})^{-2}\big)+O\big((\bar{R}_{\varepsilon}\varepsilon^{-1})^{-2}\big).
\end{aligned}
\end{equation}
Since $m^-(a_1,a_2)=I(u^-_{a_1},v^-_{a_2})=\frac{1}{4}[\|\nabla u^-_{a_1}\|^2_2+\|\nabla v^-_{a_2}\|^2_2
]-\frac{\alpha_1}{p}(1-\frac{p\gamma_p}{4})\|u^-_{a_1}\|^p_p-\frac{\alpha_2}{p}(1-\frac{p\gamma_p}{4})\|v^-_{a_2}\|^p_p$, from \eqref{l4}, we get
\begin{equation}\label{m1}
\begin{aligned}
&\frac{1}{4}\big[\|\nabla u^-_{a_1}\|^2_2+\|\nabla v^-_{a_2}\|^2_2\big]
-\frac{\alpha_1}{p}(1-\frac{p\gamma_p}{4})\|u^-_{a_1}\|^p_p-\frac{\alpha_2}{p}(1-\frac{p\gamma_p}{4})\|v^-_{a_2}\|^p_p\\
&\le\frac{k_1+k_2}{4}\mathcal{S}^2-\frac{\alpha_1}{p}\|V_{\varepsilon}\|^p_p-\frac{\alpha_2}{p}\|\overline{V}_{\varepsilon}\|^p_p+O\big((R_{\varepsilon}\varepsilon^{-1})^{-2}\big)+O\big((\bar{R}_{\varepsilon}\varepsilon^{-1})^{-2}\big).
\end{aligned}
\end{equation}
By $P_{a_1,a_2}(u^-_{a_1},v^-_{a_2})=0$, we know that
\begin{equation*}
\begin{aligned}
\sqrt{k_1+k_2}\mathcal{S}&\le \frac{\|\nabla u^-_{a_1}\|^2_2+\|\nabla v^-_{a_2}\|^2_2}{\big[\alpha_1\| u^-_{a_1}\|^p_p+\alpha_2\|v^-_{a_2}\|^p_p+2\beta\| u^-_{a_1}v^-_{a_2}\|^2_2\big]^{1/2}}\\
&=\big(\|\nabla u^-_{a_1}\|^2_2+\|\nabla v^-_{a_2}\|^2_2-\alpha_1\gamma_p\|u^-_{a_1}\|^p_p-\alpha_2\gamma_p\|v^-_{a_2}\|^p_p\big)^{\frac{1}{2}}\\
&\quad+\frac{\alpha_1\gamma_p\|u^-_{a_1}\|^p_p+\alpha_2\gamma_p\|v^-_{a_2}\|^p_p}{\sqrt{k_1+k_2}\mathcal{S}}\\
&\quad+O\big((R_{\varepsilon}\varepsilon^{-1})^{-2}\big)+O\big((\bar{R}_{\varepsilon}\varepsilon^{-1})^{-2}\big),
\end{aligned}
\end{equation*}
and then
\begin{equation}\label{m2}
\begin{aligned}
\|\nabla u^-_{a_1}\|^2_2+\|\nabla v^-_{a_2}\|^2_2&\ge \big[\sqrt{k_1+k_2}\mathcal{S}-\frac{\alpha_1\gamma_p\|u^-_{a_1}\|^p_p+\alpha_2\gamma_p\|v^-_{a_2}\|^p_p}{\sqrt{k_1+k_2}\mathcal{S}}\\
&\quad+O\big((R_{\varepsilon}\varepsilon^{-1})^{-2}\big)+O\big((\bar{R}_{\varepsilon}\varepsilon^{-1})^{-2}\big)\big]^2\\
&\quad+\alpha_1\gamma_p\|u^-_{a_1}\|^p_p+\alpha_2\gamma_p\|v^-_{a_2}\|^p_p\\
&\ge (k_1+k_2)\mathcal{S}^2-\alpha_1\gamma_p\|u^-_{a_1}\|^p_p-\alpha_2\gamma_p\|v^-_{a_2}\|^p_p\\
&\quad+O\big((R_{\varepsilon}\varepsilon^{-1})^{-2}\big)+O\big((\bar{R}_{\varepsilon}\varepsilon^{-1})^{-2}\big).
\end{aligned}
\end{equation}
From \eqref{l4}--\eqref{m2}, we have
\begin{equation}\label{z1}
\begin{aligned}
\alpha_1\|u^-_{a_1}\|^p_p+\alpha_2\|v^-_{a_2}\|^p_p&\gtrsim \alpha_1\|V_\varepsilon\|^p_p+\alpha_2\|\overline{V}_\varepsilon\|^p_p-O\big((R_{\varepsilon}\varepsilon^{-1})^{-2}\big)-O\big((\bar{R}_{\varepsilon}\varepsilon^{-1})^{-2}\big)\\ &\gtrsim\varepsilon^{4-p}+\varepsilon^{4-p}-O\big((R_{\varepsilon}\varepsilon^{-1})^{-2}\big)-O\big((\bar{R}_{\varepsilon}\varepsilon^{-1})^{-2}\big).
\end{aligned}
\end{equation}
By \eqref{m3}-\eqref{m31}, we choose $\varepsilon_{a_1,a_2}$ such that $\|u^-_{a_1}\|^p_p+\|v^-_{a_2}\|^p_p\gtrsim \varepsilon^{4-p}_{a_1,a_2}$. %Since $\{(u^-_{a_1},v^-_{a_2})\}$ is a minimizing sequence of the minimizing problem \eqref{f4}, if $\beta>\max\{\mu_1,\mu_2\}$, by Lemma \ref{lem2.2}, we have
%\begin{equation*}
%(u^-_{a_1},v^-_{a_2})\to (\sqrt{k_1}U_{\varepsilon_0},\sqrt{k_2}U_{\varepsilon_0})\ \ \text{in}\ D^2_{1}(\R^4)\times %D^2_{1}(\R^4)
%\end{equation*}
%as $(a_1,a_2)\to (0,0)$ up to a subsequence.
%Using Moser iteration (see \cite{hanq}), then $(u^-_{a_1},v^-_{a_2})\to (\sqrt{k_1}U_{\varepsilon_0},\sqrt{k_2}U_{\varepsilon_0})\ \ \text{in}\ L^{\infty}(\R^4)\times  L^{\infty}(\R^4)$
%as $(a_1,a_2)\to (0,0)$ up to a subsequence. Hence, from \eqref{z1}, we obtain $\|u^-_{a_1}\|^p_p\ge \varepsilon^{4-p}$ and $\|v^-_{a_2}\|^q_q\ge \varepsilon^{4-q}$.
Define
\begin{equation}\label{z3}
w_1=\varepsilon_{a_1,a_2} u^-_{a_1}(\varepsilon_{a_1,a_2} x) \quad \text{and} \quad w_2=\varepsilon_{a_1,a_2} v^-_{a_2}(\varepsilon_{a_1,a_2} x),
\end{equation}
then
$\|\nabla w_1\|^2_2\sim 1$, $\|\nabla w_2\|^2_2\sim 1$, $\|w_1\|^4_4\sim 1$, $\|w_2\|^4_4\sim 1$ and $\|w_1\|^p_p+\|w_2\|^p_p\gtrsim 1$. By \eqref{m4}, it is easy to
see that $\sigma^{4-p}_1\|u_1\|^p_p=\|u^-_{a_1}\|^p_p=\varepsilon^{4-p}_{a_1,a_2}\|w_1\|^p_p$ and $\sigma^{4-p}_1\|v_1\|^p_p=\|v^-_{a_2}\|^p_p=\varepsilon^{4-p}_{a_1,a_2}\|w_2\|^p_p$. Then, we have
$\sigma_1\gtrsim\varepsilon_{a_1,a_2}$. %$\sigma_2\gtrsim\varepsilon_{a_1,a_2}$ and $\sigma_1\sim \sigma_2$.

\vskip1mm
Next, we need to prove that $\sigma_1\lesssim\varepsilon_{a_1,a_2}$, %and $\sigma_2\lesssim\varepsilon_{a_1,a_2}$, %It is sufficient to prove that $\sigma_1\lesssim\varepsilon$ because $\sigma_1\sim\sigma_2$,
then
\begin{equation*}
\varepsilon_{a_1,a_2} u^-_{a_1}(\varepsilon_{a_1,a_2} x)\to \sqrt{k_1}U_{\varepsilon_0}\quad   \text{and}  \quad \varepsilon_{a_1,a_2} v^-_{a_2}(\varepsilon_{a_1,a_2} x)\to \sqrt{k_2}U_{\varepsilon_0} \ \ \text{in}\ D^{1,2}(\R^4),
\end{equation*}
as $(a_1,a_2)\to (0,0)$ and $a_1\sim a_2$.

\vskip1mm
In fact, we know that
\begin{equation*}
w_1(x)=\big( \frac{\varepsilon_{a_1,a_2}}{\sigma_1}\big)u_1\big( \frac{\varepsilon_{a_1,a_2}}{\sigma_1}x\big) \ \ \text{and}\ \ w_2(x)=\big( \frac{\varepsilon_{a_1,a_2}}{\sigma_1}\big)v_1\big( \frac{\varepsilon_{a_1,a_2}}{\sigma_1}x\big)
\end{equation*}
and
$(\bar{w}_1,\bar{w}_2)$ satisfies
\begin{equation*}
\begin{cases}
-\Delta \bar{w}_1=h_1(\bar{w}_1,\bar{w}_2)\ \ \text{in}\ \R^4,\\
-\Delta \bar{w}_2=h_2(\bar{w}_1,\bar{w}_2)\ \ \text{in}\ \R^4,
\end{cases}
\end{equation*}
where
\begin{equation*}
\bar{w}_1=\frac{1}{w_1(0)}w_1\big(\frac{1}{w_1(0)}x \big),\ \ \ \ \bar{w}_2=\frac{1}{w_2(0)}w_2\big(\frac{1}{w_2(0)}x \big),
\end{equation*}
\begin{align*}
h_1(\bar{w}_1,\bar{w}_2)&=-\lambda_1\varepsilon^2_{a_1,a_2}w^{-2}_1(0)\bar{w}_1\\
&\quad+\mu_1\bar{w}^3_1+\alpha_1\varepsilon^{4-p}_{a_1,a_2}[w_1(0)]^{p-4}|\bar{w}_1|^{p-2}\bar{w}_1+\beta[w_2(0)]^{2}w^{-2}_1(0)\bar{w}^2_2\bar{w}_1,
\end{align*}
and
\begin{align*}
h_2(\bar{w}_1,\bar{w}_2)&=-\lambda_2\varepsilon^2_{a_1,a_2}w^{-2}_2(0)\bar{w}_2\\
&\quad+\mu_2\bar{w}^3_2+\alpha_2\varepsilon^{4-p}_{a_1,a_2}[w_2(0)]^{p-4}|\bar{w}_2|^{p-2}\bar{w}_1+\beta[w_1(0)]^{2}w^{-2}_2(0)\bar{w}^2_1\bar{w}_2.
\end{align*}
By similar arguments as used for \eqref{n1}-\eqref{n2}, we have
\begin{equation*}
w_1\lesssim \frac{w_1(0)}{1+b_1r^2}\ \ \text{and}\ \ w_2\lesssim \frac{w_2(0)}{1+b_2r^2},
\end{equation*}
where
\begin{equation*}
b_1=-\lambda_1\varepsilon^2_{a_1,a_2}w^{-2}_1(0)+\mu_1+\alpha_1\varepsilon^{4-p}_{a_1,a_2}[w_1(0)]^{p-4}+\beta[w_2(0)]^{2}w^{-2}_1(0),
\end{equation*}
and
\begin{equation*}
b_2=-\lambda_2\varepsilon^2_{a_1,a_2}w^{-2}_2(0)+\mu_2+\alpha_2\varepsilon^{4-p}_{a_1,a_2}[w_2(0)]^{p-4}+\beta[w_1(0)]^{2}w^{-2}_2(0).
\end{equation*}
It follows from \eqref{z3} that $b_1\sim 1$, $b_2\sim 1$, and then
\begin{equation*}
\|w_1\|^p_p\lesssim \big(\frac{\varepsilon_{a_1,a_2}}{\sigma_1}\big)^{p} \ \ \text{and}\ \ \|w_2\|^p_p\lesssim \big(\frac{\varepsilon_{a_1,a_2}}{\sigma_1}\big)^{p},
\end{equation*}
for $2<p<3$. Therefore, $\sigma_1\lesssim \varepsilon_{a_1,a_2}$. %and $\sigma_2\lesssim \varepsilon_{a_1,a_2}$
%and  $\sigma_1\sim \varepsilon_{a_1,a_2}\sim \sigma_2$.

\vskip1mm
Finally, there exists $\varepsilon_{a_1,a_2}>0$ such that
\begin{equation*}
\big(\varepsilon_{a_1,a_2} u^-_{a_1}(\varepsilon_{a_1,a_2} x),\varepsilon_{a_1,a_2} v^-_{a_2}(\varepsilon_{a_1,a_2} x)\big)\to (\sqrt{k_1}U_{\varepsilon_0},\sqrt{k_2}U_{\varepsilon_0})\ \ \text{in} \ D^{1,2}(\R^4,\R^2),
\end{equation*}
for some $\varepsilon_0>0$ as $(a_1,a_2)\to (0,0)$ and $a_1\sim a_2$.
%where $\varepsilon_{a_1,a_2}=o(a_1)$ and $\varepsilon_{a_1,a_2}=o(a_2)$.
We complete the proof.
\end{proof}

\noindent \textbf{Proof of Theorem 1.3.}
%By Lemma \ref{lem4.2}, there exists a second solution $(u^-_{a_1},v^-_{a_2})\in \mathcal{P}^-_{a_1,a_2}$. %which satisfies $I(u^-_{a_1},v^-_{a_2})<\big\{m^+(a_1,a_2)+\frac{k_1+k_2}{4}\mathcal{S}^{2},m^-(a_1,0),m^-(0,a_2)\big\}$.
From Lemmas \ref{lem4.2} and \ref{lem4.4}, we can obtain Theorem \ref{th1.2}.
\qed

\section{Mass-critical perturbation}
In this section, we deal with the case $p=3$, $\alpha_i,\mu_i>0(i=1,2)$ and $\beta>0$. We recall the decomposition of $\mathcal{ P}_{a_1,a_2}=\mathcal{ P}_{a_1,a_2}^+\cup \mathcal{ P}_{a_1,a_2}^0\cup \mathcal{ P}_{a_1,a_2}^-$ (see \eqref{c41}). From the definition of $\mathcal{ P}_{a_1,a_2}^0$, this is $\Psi'_{(u,v)}(0)=\Psi''_{(u,v)}(0)=0$, then necessarily $\|\nabla u\|^2_2=\|\nabla v\|^2_2=0$, which is not possible because $(u,v)\in S(a_1)\times S(a_2)$. Therefore, $\mathcal{ P}_{a_1,a_2}^0=\emptyset$. By Lemma \ref{lem3.1}, we can also check that $\mathcal{ P}_{a_1,a_2}$ is a smooth manifold of codimension 1 in $S(a_1)\times S(a_2)$.
\vskip1mm

Let $\phi_0$ be a minimizer of Gagliardo-Nirenberg inequality \eqref{b2} with $p=3$. Then, from \cite{WeM}, $\phi_0$ satisfies
\begin{equation*}
\begin{cases}
-\Delta \phi_0+\frac{1}{2}\phi_0=|\phi_0|\phi_0 \quad \text{in} \quad \R^4,\\
\phi_0(x)>0 \quad \text{in} \quad \R^4.\\
\end{cases}
\end{equation*}
By the well-known uniqueness result (cf. \cite{KM}) and the scaling invariance of \eqref{g2}, we have $\phi_0(x)=\frac{1}{2}w_3(\frac{x}{\sqrt{2}})$, and
$$\frac{1}{|\mathcal{C}_3|^{3}}=\frac{\|\nabla \phi_0\|^2_2\|\phi_0\|_2}{\|\phi_0\|^3_3}=\frac{2\|w_3\|_2}{3}.$$

%Let $T(a_1,a_2):=\alpha_1a_1+\alpha_2a_2$. %and $$\gamma_1\!:=\!\frac{3}{2|\mathcal{C}_3|^{3}}.$$

\begin{Lem}\label{lem5.1} %If $T(a_1,a_2)\le\!\frac{3}{2|\mathcal{C}_3|^{3}}$,
If $0<\alpha_ia_i<\!\frac{3}{2|\mathcal{C}_3|^{3}}~(i=1,2)$, for all $(u,v)\in S(a_1)\times S(a_2)$, there exists $t_{(u,v)}$ such that $t_{(u,v)}\star(u,v)\in\mathcal{P}_{a_1,a_2}$. Further, $t_{(u,v)}$ is the unique critical point of the function $\Psi_{(u,v)}$ and is a strict maximum point at positive level.
Moreover:
\vskip1mm
\noindent $(1)$ $\mathcal{ P}_{a_1,a_2}=\mathcal{ P}_{a_1,a_2}^-$ and $P_{a_1,a_2}(u,v)<0$ iff $t_{(u,v)}<0$.\\
\noindent $(2)$ $\Psi_{(u,v)}$ is strictly increasing in $(-\infty,t_{(u,v)})$. \\
\noindent $(3)$ The map $(u,v) \mapsto t_{(u,v)} \in \mathbb{R}$ is of class $C^1$.
\end{Lem}
\begin{proof}
Since $p=3$, we have
\begin{equation*}
\begin{aligned}
\Psi_{(u,v)}(s) &=e^{2s}\Big[\big(\frac{\|\nabla u\|^2_2}{2}-\frac{\alpha_1\|u\|^3_3}{3}\big)
+\big(\frac{\|\nabla v\|^2_2}{2}-\frac{\alpha_2\|v\|^3_3}{3}\big)\Big]\\
&\quad-\frac{e^{4s}}{4}\Big(\mu_1\|u\|^4_4+\mu_2\|v\|^4_4+2\beta\|uv\|^{2}_2\Big)\\
&\ge \frac{e^{2s}}{2}\Big[\|\nabla u\|^2_2\big(1-\frac{2|\mathcal{C}_3|^3}{3}\alpha_1a_1\big)
+\|\nabla v\|^2_2\big(1-\frac{2|\mathcal{C}_3|^3}{3}\alpha_2a_2\big)\Big]\\
%&\ge e^{2s}\|\nabla u\|^2_2\big[\frac{1}{2}-\frac{\alpha_1}{3}|\mathcal{C}_3|^3a_1\big]+e^{2s}\|\nabla v\|^2_2\big[\frac{1}{2}-\frac{\alpha_2}{3}|\mathcal{C}_3|^3a_2\big]\\
&\quad-\frac{e^{4s}}{4}\Big(\mu_1\|u\|^4_4+\mu_2\|v\|^4_4+2\beta\|uv\|^{2}_2\Big).\\
\end{aligned}
\end{equation*}
Note that $s\star (u,v)\in \mathcal{P}_{a_1,a_2}$ if and only if $\Psi'_{(u,v)}(s)=0$. It's easy to see that if $\big[\|\nabla u\|^2_2\big(1-\frac{2|\mathcal{C}_3|^3}{3}\alpha_1a_1\big)
+\|\nabla v\|^2_2\big(1-\frac{2|\mathcal{C}_3|^3}{3}\alpha_2a_2\big)\big]$ is positive, then $\Psi_{(u,v)}(s)$ has a unique critical point $t_{(u,v)}$, which is a strict maximum point at positive level. Therefore, under the condition of $0<\alpha_ia_i<\!\frac{3}{2|\mathcal{C}_3|^{3}}~(i=1,2)$, we know that $\big[\frac{\|\nabla u\|^2_2
+\|\nabla v\|^2_2}{2}-\frac{\alpha_1\|u\|^3_3+\alpha_2\|v\|^3_3}{3}\big]>0$. If $(u,v)\in \mathcal{ P}_{a_1,a_2}$, then $t_{(u,v)}$ is a maximum point, we have that $\Psi''_{(u,v)}(t_{(u,v)})\le 0$. Since $\mathcal{P}^0_{a_1,a_2}=\emptyset$, we have $(u,v)\in \mathcal{ P}^-_{a_1,a_2}$. To prove that the map $(u,v)\in S(a_1)\times S(a_2)\mapsto t_{(u,v)}\in \R$ is of class $C^1$, we can apply the implicit function theorem as in Lemma \ref{lem3.2}. Finally, $\Psi'_{(u,v)}(s)<0$ if and only if $s>t_{(u,v)}$, then $P_{a_1,a_2}(u,v)=\Psi'_{(u,v)}(0)<0$ if and only if $t_{(u,v)}<0$.
%The proof is similar to the Lemma 6.2 in \cite{Soave1}.
\end{proof}

\begin{Cor} \label{cor5.1}
Assume $p=3$, $\alpha_i,\mu_i,a_i>0    (i=1,2)$ and $\beta>0$. Let
$0<\alpha_ia_i<\!\frac{3}{2|\mathcal{C}_3|^{3}}~(i=1,2)$, then
\begin{equation*}
m^-(a_1,a_2):=\inf_{(u,v)\in \mathcal{ P}^-_{a_1,a_2}}I(u,v)>0.\\
\end{equation*}
%$(2)$ There exists $k>0$ sufficiently small such that
%\begin{equation*}
%0<\sup_{(u,v)\in\overline{A_k}} I(u,v)<m(a_1,a_2),\quad %\text{and}\quad  (u,v)\in\overline{A_k}\Rightarrow I(u,v), P(u,v)>0,
%\end{equation*}
%where $\overline{A_k}=\big\{(u,v)\in S(a_1)\times S(a_2): \|\nabla u\|^2_2+\|\nabla v\|^2_2<k\big\}$.
\end{Cor}

\begin{proof}
If $(u,v)\in \mathcal{P}^{-}_{a_1,a_2}$, then by Gagliardo-Nirenberg and the Sobolev inequalities, we have
\begin{equation*}
\begin{aligned}
\|\nabla u\|^2_2+\|\nabla v\|^2_2&<2\big[\mu_1\|u\|^4_4+\mu_2\|v\|^4_4+2\beta\|uv\|^2_2\big]+\frac{2\alpha_1}{3}\|u\|^3_3+\frac{2\alpha_2}{3}\|v\|^3_3\\
&\le 8D_1[\|\nabla u\|^2_2+\|\nabla v\|^2_2]^2+\frac{2|\mathcal{C}_3|^3}{3}\big(\alpha_1a_1\|\nabla u\|^2_2+\alpha_2a_2\|\nabla v\|^2_2\big).\\
\end{aligned}
\end{equation*}
Moreover, $0<\alpha_ia_i<\!\frac{3}{2|\mathcal{C}_3|^{3}}~(i=1,2)$ and $\|\nabla u\|^2_2+\|\nabla v\|^2_2\neq 0$ (since $(u,v)\in S(a_1)\times S(a_2)$), we get
\begin{equation*}
\inf_{(u,v)\in \mathcal{P}^{-}_{a_1,a_2}}\|\nabla u\|^2_2+\|\nabla v\|^2_2\ge C>0.
\end{equation*}
So
\begin{equation*}
\begin{aligned}
m^-(a_1,a_2)&=\inf_{(u,v)\in \mathcal{P}^{-}_{a_1,a_2}}I(u,v)\\
&=\inf_{(u,v)\in \mathcal{P}^{-}_{a_1,a_2}}\frac{1}{4}\big[\|\nabla u\|^2_2+\|\nabla v\|^2_2-\frac{2\alpha_1}{3}\|u\|^3_3-\frac{2\alpha_2}{3}\|v\|^3_3\big]\\
&\ge \inf_{(u,v)\in \mathcal{P}^{-}_{a_1,a_2}}\Big[\frac{1}{4}\big(1-\frac{2|\mathcal{C}_3|^3\alpha_1a_1}{3}\big)\|\nabla u\|^2_2\\
&\qquad \qquad\qquad+\frac{1}{4}\big(1-\frac{2|\mathcal{C}_3|^3\alpha_2a_2}{3}\big)\|\nabla v\|^2_2\Big]\\
&\ge C>0.
\end{aligned}
\end{equation*}
\end{proof}

From the above lemmas, then $I\big|_{S(a_1)\times S(a_2)}$ has a mountain pass geometry. We need an estimate from above on $$m^-(a_1,a_2)=\inf\limits_{(u,v)\in\mathcal{P}_{a_1,a_2}}I(u,v)=\inf\limits_{(u,v)\in\mathcal{P}^{-}_{a_1,a_2}}I(u,v).$$ % where $S_r(a_1)\times S_r(a_2)$ is the subset of the radial functions in $S(a_1)\times S(a_2)$.

\begin{Lem}\label{lem5.2}
Let $p=3$, $\alpha_i,\mu_i,a_i>0(i=1,2)$, and $\beta\in\big(0,\min\{\mu_1,\mu_2\}\big)\cup\big(\max\{\mu_1,\mu_2\},\infty\big)$. If $0<\alpha_ia_i<\!\frac{3}{2|\mathcal{C}_3|^{3}}~(i=1,2)$, then $$0<m^-(a_1,a_2)<\frac{k_1+k_2}{4}\mathcal{S}^2,$$
where $k_1=\frac{\beta-\mu_2}{\beta^2-\mu_1\mu_2}$ and $k_2=\frac{\beta-\mu_1}{\beta^2-\mu_1\mu_2}$.
\end{Lem}
\begin{proof}
%Since the proof is similar to that of Lemma \ref{lem4.1}, we only sketch it.
From \eqref{L1}, $U_{\varepsilon}=\frac{2\sqrt{2}\varepsilon}{\varepsilon^2+|x|^2}$, taking a radially decreasing cut-off function $\xi \in C_0^{\infty}(\R^4)$ such that $\xi \equiv 1$ in $B_1$, $\xi \equiv 0$ in $\R^4 \backslash B_2$, and let $W_{\varepsilon}(x) = \xi(x) U_{\varepsilon}(x)$. We have (see \cite{ABC}),
\begin{equation}\label{af1}
\|\nabla W_\varepsilon\|^2_2=\mathcal{S}^2+O(\varepsilon^2),    \quad \quad \|W_\varepsilon\|^4_4=\mathcal{S}^2+O(\varepsilon^4),%\quad\| W_\varepsilon\|^p_p=O(\varepsilon^{4-p}),
\end{equation}
and
\begin{equation}\label{af11}
 \| W_\varepsilon\|^3_3=O(\varepsilon), \quad \quad \|W_\varepsilon\|^2_2=O(\varepsilon^2|\ln\varepsilon|).
\end{equation}
Setting
\begin{equation*}
%(W_{\varepsilon,t_1},W_{\varepsilon,t_2})=\big(t_1W_{\varepsilon},t_2W_{\varepsilon}\big)\quad \text{and} \quad
\big(\overline{W}_{\varepsilon},\overline{V}_{\varepsilon}\big)=\Big(a_1\frac{W_{\varepsilon}}{\|W_{\varepsilon}\|_2},a_2\frac{W_{\varepsilon}}{\|W_{\varepsilon}\|_2}\Big), \end{equation*}
then $(\overline{W}_{\varepsilon},\overline{V}_{\varepsilon})\in S(a_1)\times S(a_2)$. From Lemma \ref{lem5.1}, there exists $\tau_{\varepsilon}\in\R$ such that $\tau_{\varepsilon}\star(\overline{W}_{\varepsilon},\overline{V}_{\varepsilon})\in \mathcal{P}_{a_1,a_2}$,
this implies that
\begin{equation*}
m^{-}(a_1,a_2)\le I\big(\tau_{\varepsilon}\star(\overline{W}_{\varepsilon},\overline{V}_{\varepsilon})\big)=\max_{t\in \R}I\big(t\star(\overline{W}_{\varepsilon},\overline{V}_{\varepsilon})\big),
\end{equation*}
and
\begin{equation}\label{ab3}
\begin{aligned}
e^{2\tau_{\varepsilon}}\big[\|\nabla \overline{W}_{\varepsilon}\|_{2}^2+\|\nabla \overline{V}_{\varepsilon}\|_{2}^2\big]&=e^{4\tau_{\varepsilon}}\big[\mu_1\|\overline{W}_{\varepsilon}\|_{4}^4+\mu_2\|\overline{V}_{\varepsilon}\|_{4}^4+2\beta\|\overline{W}_{\varepsilon}\overline{V}_{\varepsilon}\|^2_2\big]\\
&\quad +e^{2\tau_{\varepsilon}}\Big(\frac{2\alpha_1}{3}\|\overline{W}_{\varepsilon}\|_{3}^{3}+\frac{2\alpha_2}{3}\|\overline{V}_{\varepsilon}\|_{3}^{3}\Big).
\end{aligned}
\end{equation}
%\begin{equation}\label{ab3}
%\begin{aligned}
%e^{2\tau_{\varepsilon}}\big[\|\nabla \overline{W}_{\varepsilon}\|_{2}^2+\|\nabla \overline{V}_{\varepsilon}\|_{2}^2\big]&=e^{4\tau_{\varepsilon}}\big[\mu_1\|\overline{W}_{\varepsilon}\|_{4}^4+\mu_2\|\overline{V}_{\varepsilon}\|_{4}^4+2\rho\|\overline{W}_{\varepsilon}\overline{V}_{\varepsilon}\|^2_2\big]+e^{2\tau_{\varepsilon}}\int_{\R^4}\beta|\overline{W}_{\varepsilon}|^2\overline{V}_{\varepsilon}.
%\end{aligned}
%\end{equation}
Then,
\begin{equation}\label{y4}
\begin{aligned}
&e^{2\tau_{\varepsilon}}\big[\mu_1\|\overline{W}_{\varepsilon}\|_{4}^4+\mu_2\|\overline{V}_{\varepsilon}\|_{4}^4+2\beta\|\overline{W}_{\varepsilon}\overline{V}_{\varepsilon}\|^2_2\big]\\
&\quad\ge
\big(1-\frac{2\alpha_1|\mathcal{C}_3|^3a_1}{3}\big)\|\nabla \overline{W}_{\varepsilon}\|^2_2+\big(1-\frac{2\alpha_2|\mathcal{C}_3|^3a_2}{3}\big)\|\nabla \overline{V}_{\varepsilon}\|^2_2.\\
\end{aligned}
\end{equation}
Therefore, for $\varepsilon$ small enough, we have
\begin{equation*}
\begin{aligned}
I\big(\tau_{\varepsilon}\star(\overline{W}_{\varepsilon},\overline{V}_{\varepsilon})\big)&=\frac{e^{2\tau_{\varepsilon}}}{2}\big(\|\nabla \overline{W}_{\varepsilon}\|_{2}^2+\|\nabla \overline{V}_{\varepsilon}\|^2_2\big)\\
&\qquad-\frac{e^{4\tau_{\varepsilon}}}{4}\big(\|\overline{W}_{\varepsilon}\|_{4}^4+\|\overline{V}_{\varepsilon}\|_{4}^4+2\beta\|\overline{W}_{\varepsilon}\overline{V}_{\varepsilon}\|^2_2\big)\\
&\qquad-\frac{e^{2\tau_{\varepsilon}}}{2}\Big(\alpha_1\|\overline{W}_{\varepsilon}\|_{3}^{3}+\alpha_2\|\overline{V}_{\varepsilon}\|_{3}^{3}\Big)\\
&\le \max_{s_1,s_2>0}\frac{s^2_1+s^2_2}{2}\|\nabla W_{\varepsilon}\|_{2}^2-\frac{\big(\mu_1s^4_1+\mu_2s^4_2+2\rho s^2_1s^2_2\big)}{4}\|W_{\varepsilon}\|_{4}^4\\
&\qquad-\frac{e^{2\tau_{\varepsilon}}}{2}\big(\alpha_1\|\overline{W}_{\varepsilon}\|_{3}^{3}+\alpha_2\|\overline{V}_{\varepsilon}\|_{3}^{3}\big)\\
\end{aligned}
\end{equation*}
where $s_1=\frac{e^{\tau_{\varepsilon}}a_1}{\|W_{\varepsilon}\|_2}$ and $s_2=\frac{e^{\tau_{\varepsilon}}a_2}{\|W_{\varepsilon}\|_2}$.
Define $$f(s_1,s_2)=\frac{s^2_1+s^2_2}{2}\|\nabla W_{\varepsilon}\|_{2}^2-\frac{(\mu_1s^4_1+\mu_2s^4_2+2\rho s^2_1s^2_2)}{4}\|W_{\varepsilon}\|_{4}^4.$$
Using that, for all $0<s_1,s_2$,
\begin{equation*}
\max_{s_1,s_2>0}f(s_1,s_2)\le \frac{k_1+k_2}{4}\mathcal{S}^2+O(\varepsilon^2).
\end{equation*}
Finally, by \eqref{y4}, we have
\begin{equation*}
\begin{aligned}
\frac{e^{2\tau_{\varepsilon}}}{2}\Big(\alpha_1\|\overline{W}_{\varepsilon}\|_{3}^{3}+\alpha_2\|\overline{V}_{\varepsilon}\|_{3}^{3}\Big)
&=\frac{e^{2\tau_{\varepsilon}}}{2}\frac{\alpha_1a^3_1+\alpha_2a^3_2}{\|W_\varepsilon\|^3_2}\int_{\R^4}|W_{\varepsilon}|^{3}\\
&\ge \frac{C}{\|W_{\varepsilon}\|_2} \int_{\R^4}|W_{\varepsilon}|^{3}\\
&\ge C|\ln\varepsilon|^{-\frac{1}{2}}.
\end{aligned}
\end{equation*}
Hence, we deduce that
\begin{equation*}
m^{-}(a_1,a_2)\le \max_{t\in \R}I\big(t\star(\overline{W}_{\varepsilon},\overline{V}_{\varepsilon})\big)<\frac{k_1+k_2}{4}\mathcal{S}^2.
\end{equation*}
\end{proof}

\begin{Lem}\label{lem5.31}
Let $p=3$, $\mu_i,\alpha_i,a_i>0(i=1,2)$ and $0<\alpha_ia_i<\!\frac{3}{2|\mathcal{C}_3|^{3}}~ (i=1,2)$, there exists $\beta_0>0$ such that
$$m^-(a_1,a_2)<\min\big\{m^-(a_1,0),m^-(0,a_2)\big\},$$
for any $\beta>\beta_0$, where $m^-(a_1,0)$ and $ m^-(0,a_2)$ are defined in \eqref{c12}.
\end{Lem}
\begin{proof}
By Lemma \ref{lem2.1}, $m^-(a_1,0)$ can be achieved by $u^*\in S(a_1)$, and $m^-(0,a_2)$ can be achieved by $v^*\in S(a_2)$. Moreover, $u^*$ and $v^*$ are positive, radially symmetric and radially decreasing functions.
From Section 2, it is easy to see that
\begin{equation*}
\mathcal{A}_{p,\mu_1,\alpha_1}(s\star u^*)\to 0\quad \text{and}\quad \mathcal{A}_{p,\mu_2,\alpha_2}(s\star v^*)\to 0,\quad \text{as}\ s\to -\infty.
\end{equation*}
Therefore, there exists an $s_0\ll-1$ which is independent of $\beta$ such that
\begin{equation*}
	\max_{s<s_0}I(s\star(u,v))<\max_{s<s_0}\mathcal{A}_{p,\mu_1,\alpha_1}(s\star u^*)+\mathcal{A}_{p,\mu_2,\alpha_2}(s\star v^*)
\end{equation*}	
\begin{equation*}	
<\min\big\{m^-(a_1,0),m^-(0,a_2)\big\}.
\end{equation*}
If $s\ge s_0$, then the intersection term can be bounded from below:
\begin{equation*}
\int_{\R^4}|s\star u^*|^{2}|s\star v^*|^{2}=e^{4s}\int_{\R^4}|u^*|^{2}|v^*|^{2}\ge Ke^{4s_0},
\end{equation*}
where $K>0$. Thus, we have
\begin{equation*}
\begin{aligned}
	\max_{s\ge s_0}I(s\star(u,v))&\le\max_{s\ge s_0}\mathcal{A}_{p,\mu_1,\alpha_1}(s\star u^*)+\mathcal{A}_{p,\mu_2,\alpha_2}(s\star v^*)-Ke^{4s_0}\beta\\
		&\le m^-(a_1,0)+m^-(0,a_2)-Ke^{4s_0}\beta.
\end{aligned}
\end{equation*}
%From \cite{Soave2}, we get $m^-(a_1,0)<\frac{\mu_1\mathcal{S}^2}{4}$, $m^-(0,a_2)<\frac{\mu_2\mathcal{S}^2}{4}$ for any $0<a_1,a_2<\tau$.
It is clear that there exists $\beta_0>0$ such that $$\max\limits_{s\ge s_0}I(s\star(u,v))<\min\big\{m^-(a_1,0),m^-(0,a_2)\big\}      \;\; \hbox{ for all }  \;\; \beta>\beta_0.$$
%and the last term is strictly smaller than $\min\big\{m(a_1,0),m(0,a_2)\big\}$ provided $\beta$ is sufficiently large.
\end{proof}

By using the same techniques as that in Lemma \ref{lem3.4}, we can prove the following lemma.
\begin{Lem}\label{lem5.3}
Let $p=3$, $\mu_i,\alpha_i,a_i>0(i=1,2)$, there exists a $\beta_0>0$ such that $\beta>\beta_0$ and $\beta\in \big(0,\min\{\mu_1,\mu_2\}\big)\cup\big(\max\{\mu_1,\mu_2\},\infty\big)$.
If $0<\alpha_ia_i<\!\frac{3}{2|\mathcal{C}_3|^{3}}~ (i=1,2)$ and
\begin{equation*}
0<m^-(a_1,a_2)<\min\big\{\frac{k_1+k_2}{4}\mathcal{S}^2,m^-(a_1,0),m^-(0,a_2)\big\},
\end{equation*}
where $m^-(a_1,0)$ and $m^-(0,a_2)$ are defined in \eqref{c12}, $k_1=\frac{\beta-\mu_2}{\beta^2-\mu_1\mu_2}$ and $k_2=\frac{\beta-\mu_1}{\beta^2-\mu_1\mu_2}$, then $m^-(a_1,a_2)$ can be achieved by some function $(u^-_{a_1},v^-_{a_2})\in S(a_1)\times S(a_2)$ which is real valued, positive, radially symmetric and radially decreasing.
\end{Lem}
\begin{proof}
Similar to the proof of Lemma \ref{lem3.4}. Let $\{(u_n,v_n)\}\subset \mathcal{P}^-_{a_1,a_2}$ be a minimizing sequence.
We also assume that $(u_n,v_n)$ are all real valued, nonnegative, radially symmetric and decreasing in $r=|x|$.
Since $P_{a_1,a_2}(u_n,v_n)=o_n(1)$, we have
\begin{equation}\label{z4}
\begin{aligned}
&||\nabla u_n||^2_2+\|\nabla v_n\|^2_2-\frac{2\alpha_1}{3}\|u_n\|^3_3\\
&-\frac{2\alpha_2}{3}\|v_n\|^3_3-\mu_1\|u_n\|^{4}_{4}-\mu_2\|v_n\|^{4}_{4}-
2\beta\|u_nv_n\|^2_2=o_n(1).
\end{aligned}
\end{equation}
Thus,
\begin{equation*}
m^-(a_1,a_2)=I(u_n,v_n)=\frac{1}{4}\big[\mu_1\|u_n\|^{4}_{4}+\mu_2\|v_n\|^{4}_{4}+
2\beta\|u_nv_n\|^2_2\big]+o_n(1),
\end{equation*}
and by the H\"{o}lder's inequality, we get $\alpha_1\|u_n\|^3_3+\alpha_2\|v_n\|^3_3$ and $\mu_1\|u_n\|^{4}_{4}+\mu_2\|v_n\|^{4}_{4}+
2\beta\|u_nv_n\|^2_2$ are bounded. Hence, it follows from \eqref{z4} that $\{(u_n,v_n)\}$ is bounded in $H^1(\R^4,\R^2)$. There exists $(u,v)\in H^1(\R^4,\R^2)$ such that $(u_n,v_n)\rightharpoonup (u,v)$ in $H^1(\R^4,\R^2)$.
%Since $P_{a_1,a_2}(u_n,v_n)\to 0$, we have
We need to prove that $u\neq 0$ and $v\neq0$.
\vskip1mm
{\bf Case 1.} If $u=0$ and $v=0$, then $\int_{\R^4}\alpha_1|u_n|^3\to 0$, $\int_{\R^4}\alpha_2|v_n|^3\to 0$,
 we have
\begin{equation*}
P_{a_1,a_2}(u_n,v_n)=||\nabla u_n||^2_2+\|\nabla v_n\|^2_2-\mu_1\|u_n\|^{4}_{4}-\mu_2\|v_n\|^{4}_{4}-
2\beta\|u_nv_n\|^2_2=o_n(1).
\end{equation*}
By \eqref{f4}, we have
\begin{equation}\label{fg5}
\begin{aligned}
\sqrt{k_1+k_2}\mathcal{S}\big(\mu_1\|u_n\|^{4}_{4}+\mu_2\|v_n\|^{4}_{4}+2\beta\|u_nv_n\|^2_2\big)^{\frac{1}{2}}
&\le ||\nabla u_n||^2_2+\|\nabla v_n\|^2_2.\\
%&=\mu_1\|u_n\|^{4}_{4}+\mu_2\|v_n\|^{4}_{4}+2\beta\|u_nv_n\|^2_2+o_n(1).
\end{aligned}
\end{equation}
Then, we distinguish the two cases:
\begin{equation*}
\text{either}\ \ (i)\ \ ||\nabla u_n||^2_2+\|\nabla v_n\|^2_2\to 0\ \ \text{or}\ \ (ii) \ \ ||\nabla u_n||^2_2+\|\nabla v_n\|^2_2\to l>0.
\end{equation*}
If $(i)$ holds, we have that $I(u_n,v_n)\to 0$ which contradicts the fact that $m^-(a_1,a_2)>0$. If $(ii)$ holds, we deduce \eqref{fg5}
that
\begin{equation*}
||\nabla u_n||^2_2+\|\nabla v_n\|^2_2\ge (k_1+k_2)\mathcal{S}^2+o_n(1).
\end{equation*}
Therefore,
\begin{equation*}
\begin{aligned}
m^-(a_1,a_2)+o_n(1)=I(u_n,v_n)=\frac{1}{4}\int_{\R^4}|\nabla u_n|^2+|\nabla v_n|^2
%&\qquad-\frac{1}{4}\int_{\R^4}\big(\mu_1|u_n|^4+\mu_2|v_n|^4+2\beta|u_n|^2|v_n|^2\big)\\
\ge \frac{k_1+k_2}{4}\mathcal{S}^2,
\end{aligned}
\end{equation*}
this contradicts $m^-(a_1,a_2)<\frac{k_1+k_2}{4}\mathcal{S}^2$.
\vskip3mm
{\bf Case 2.} If $u\neq0$ and $v=0$, then $v_n\to 0$ in $L^3(\R^4)$. Let $\bar{u}_n=u_n-u$, $\bar{u}_n\to 0$ in $L^3(\R^4)$. By the maximum principle (see \cite[Theorem 2.10]{hanq}), $u$ is a positive solution of \eqref{a2}, then $m^-(a_1,0)\le m^-(\|u\|_2,0)$.
By the Br\'{e}zis-Lieb Lemma \cite{WM}, we deduce that
\begin{equation*}
\begin{aligned}
&P_{a_1,a_2}(u_n,v_n)\\
&=||\nabla u_n||^2_2+\|\nabla v_n\|^2_2-\mu_1\|u_n\|^{4}_{4}-\mu_2\|v_n\|^{4}_{4}\\
&\quad-\gamma_p\alpha_1\|u_n\|^p_p-2\beta\|u_nv_n\|^2_2+o_n(1)\\
%&=||\nabla \bar{u}_n||^2_2+\|\nabla v_n\|^2_2-\mu_1\|\bar{u}_n\|^{4}_{4}-\mu_2\|v_n\|^{4}_{4}-2\beta\|\bar{u}_nv_n\|^2_2\\
%&\quad+||\nabla u||^2_2-\mu_1\|u\|^{4}_{4}-\gamma_p\alpha_1\|u\|^p_p+o_n(1)\\
%&=P_{a_1,a_2}(\bar{u}_n,v_n)+P_{a_1,a_2}(u,0)+o_n(1)\\
&=||\nabla \bar{u}_n||^2_2+\|\nabla v_n\|^2_2-\mu_1\|\bar{u}_n\|^{4}_{4}-\mu_2\|v_n\|^{4}_{4}-2\beta\|\bar{u}_nv_n\|^2_2+o_n(1).
\end{aligned}
\end{equation*}
By \eqref{f4}, we have
\begin{equation*}
\begin{aligned}
\sqrt{k_1+k_2}\mathcal{S}\big(\mu_1\|\bar{u}_n\|^{4}_{4}+\mu_2\|v_n\|^{4}_{4}+2\beta\|\bar{u}_nv_n\|^2_2\big)^{\frac{1}{2}}
&\le ||\nabla \bar{u}_n||^2_2+\|\nabla v_n\|^2_2.\\
\end{aligned}
\end{equation*}
Then, we distinguish the two cases:
\begin{equation*}
\text{either}\ \ (i)\ \ ||\nabla \bar{u}_n||^2_2+\|\nabla v_n\|^2_2\to 0\ \ \text{or}\ \ (ii) \ \ ||\nabla \bar{u}_n||^2_2+\|\nabla v_n\|^2_2\to l>0.
\end{equation*}
If $(i)$ holds, we have $I(\bar{u}_n,v_n)\to 0$. %which contradicts the fact that $m^-(a_1,a_2)>0$.
Thus,
\begin{equation*}
\begin{aligned}
m^-(a_1,a_2)+o_n(1)&=I(u_n,v_n)=I(\bar{u}_n,v_n)+I(u,0)+o_n(1)\\
                   &=I(u,0)+o_n(1)\ge m^-(a_1,0).
\end{aligned}
\end{equation*}
%where we used the fact that $ m(a_1,0)=0$ for $T(a_1,a_2)<\gamma_{1}$ (see \cite[Theorem 1.1]{Soave1}).
This is a contradiction.
If $(ii)$ holds,
\begin{equation*}
\begin{aligned}
m^-(a_1,a_2)+o_n(1)&=I(u_n,v_n)=I(\bar{u}_n,v_n)+I(u,0)+o_n(1)\\
&\ge \frac{1}{4}\int_{\R^4}(|\nabla \bar{u}_n|^2+|\nabla v_n|^2)+\frac{\mu_1}{4}\int_{\R^4}|u|^4+o_n(1)\\
&\ge \frac{k_1+k_2}{4}\mathcal{S}^2,\\
\end{aligned}
\end{equation*}
which contradicts our assumption on $m^-(a_1,a_2)$.
\vskip1mm
{\bf Case 3.} If $u=0$ and $v\neq0$. Analysis similar to that in the proof of Case 2,
%shows that $m^-(a_1,a_2)\ge \frac{k_1+k_2}{4}\mathcal{S}^2+m(a_1,a_2)+o_n(1)$,
we also have a contradiction.
\vskip1mm
Therefore, $u\neq0$ and $v\neq0$. It remains to show that $m^-(a_1,a_2)$ is achieved.
It follows from the same computations in \eqref{d1}-\eqref{d4} that there exist $\lambda_1,\lambda_2>0$ such that
\begin{equation}\label{d52}
\lambda_1a^2_1+ \lambda_2a^2_2=\lambda_1\|u\|^2_2+\lambda_2\|v\|^2_2.
\end{equation}
%By the maximum principle, $u,v>0$. Then, from Lemma \ref{lem2.11}, we get $\lambda_1,\lambda_2>0$.
Since $\|u\|^2_2\le a^2_1$ and $\|v\|^2_2\le a^2_2$, it follows from \eqref{d52} that $\|u\|^2_2= a^2_1$ and $\|v\|^2_2= a^2_2$, and hence $(u,v)\in \mathcal{P}_{a_1,a_2}$. Let $(\bar{u}_n,\bar{v}_n)=(u_n-u,v_n-v)$, we obtain
\begin{equation*}
\begin{aligned}
o_n(1)&=P_{a_1,a_2}(u_n,v_n)=P_{a_1,a_2}(\bar{u}_n,\bar{v}_n)+P_{a_1,a_2}(u,v)+o_n(1)\\
&=||\nabla \bar{u}_n||^2_2+\|\nabla \bar{v}_n\|^2_2-\mu_1\|\bar{u}_n\|^{4}_{4}-\mu_2\|\bar{v}_n\|^{4}_{4}-2\beta\|\bar{u}_n\bar{v}_n\|^2_2+o_n(1).\\
\end{aligned}
\end{equation*}
Here again we distinguish the two cases:
\begin{equation*}
\text{either}\ \ (i)\ \ ||\nabla \bar{u}_n||^2_2+\|\nabla \bar{v}_n\|^2_2\to 0\ \ \text{or}\ \ (ii) \ \ ||\nabla \bar{u}_n||^2_2+\|\nabla \bar{v}_n\|^2_2\to l>0.
\end{equation*}
If $(ii)$ holds,
\begin{equation*}
\begin{aligned}
m^-(a_1,a_2)+o_n(1)&=I(u_n,v_n)=I(\bar{u}_n,\bar{v}_n)+I(u,v)+o_n(1)\\
&\ge \frac{1}{4}\int_{\R^4}(|\nabla \bar{u}_n|^2+|\nabla \bar{v}_n|^2)\\
&\quad+\frac{1}{4}\int_{\R^4}(\mu_1|u|^4+\mu_2|v|^4)+o_n(1)\\
&\ge \frac{k_1+k_2}{4}\mathcal{S}^2+o_n(1).
\end{aligned}
\end{equation*}
We then conclude that $I(u,v)=m^-(a_1,a_2)$ and $(u_n,v_n)\to (u,v)$ in $H^1(\R^4,\R^2)$. Therefore, we have proved that $m^-(a_1,a_2)$ can be attained by some $(u^-_{a_1},v^-_{a_2})$ which is real valued, positive, radially symmetric and decreasing in $r=|x|$.
\end{proof}

In the following, we study the nonexistence of ground states for the case $\alpha_1a_1\ge\frac{3}{2|\mathcal{C}_3|^3}$ and $\alpha_2a_2\ge\frac{3}{2|\mathcal{C}_3|^3}$. We follow the ideas of \cite{WuZ} and \cite[Lemma 3.2]{LXF}.

\begin{Lem}\label{lem5.4}
Let $p=3$, $\mu_i,\alpha_i,a_i>0(i=1,2)$, there exists a $\beta_0>0$ such that $\beta>\beta_0$ and $\beta\in \big(0,\min\{\mu_1,\mu_2\}\big)\cup\big(\max\{\mu_1,\mu_2\},\infty\big)$.
%$0<\beta<\min\{\mu_1,\mu_2\}$ or $\beta>\max\{\mu_1,\mu_2\}$.
If $\alpha_1a_1\ge\frac{3}{2|\mathcal{C}_3|^3}$ and $\alpha_2a_2\ge\frac{3}{2|\mathcal{C}_3|^3}$, we have
\begin{equation*}
m^-(a_1,a_2)=0.
\end{equation*}
Moreover, $m^-(a_1,a_2)$ can not be achieved and therefore  \eqref{eq1.1}-\eqref{eq1.11} has no ground states.
\end{Lem}

\begin{proof}
For all $\alpha_1,\alpha_2>0$, we can choose  $(u,v)\in S(a_1)\times S(a_2)$ such that (see \cite[Lemma 3.3]{Soave1})
\begin{equation*}
\|\nabla u\|^2_2>\frac{2\alpha_1}{3}\|u\|^3_3\quad \text{and}\quad \|\nabla v\|^2_2 >\frac{2\alpha_2}{3}\|v\|^3_3.
\end{equation*}

\textbf{Case 1 ($\alpha_1a_1=\frac{3}{2|\mathcal{C}_3|^3}$ and $\alpha_2a_2=\frac{3}{2|\mathcal{C}_3|^3}$)}.

Since the assumption $\frac{2\alpha_1}{3}a_1=\frac{1}{|\mathcal{C}_3|^3}$ and $\frac{2\alpha_2}{3}a_2=\frac{1}{|\mathcal{C}_3|^3}$,
for any $(u,v)\in S(a_1)\times S(a_2)$, we have
\begin{equation*}
\frac{2\alpha_1}{3}\|u\|_3^3+\frac{2\alpha_2}{3}\|v\|_3^3\leq \frac{2\alpha_1}{3} |\mathcal{C}_3|^3
a_1\|\nabla u\|_2^2+\frac{2\alpha_2}{3} |\mathcal{C}_3|^3a_2\|\nabla v\|_2^2=\|\nabla u\|_2^2+\|\nabla v\|_2^2.
\end{equation*}
Let $\{Q_n\}$ is a minimizing sequence of (see \cite{KM,WeM})
\begin{equation}\label{e3.2}
\begin{split}
\frac{1}{|\mathcal{C}_3|^3}:&=\inf_{ u\in
H^{1}(\mathbb{R}^4)\setminus\{0\}}\frac{\|\nabla
u\|_2^{2}\|u\|_2}{\|u\|_3^3}.
\end{split}
\end{equation}
%$Q$ is the ground state of the equation
%\begin{equation*}
%-\Delta u+\frac{1}{2}u=u^2.
%\end{equation*}
Then by scaling $u_n(x):=a_1t^2_n\|Q_n\|_2^{-1}Q_n(t_nx)$ and $v_n(x):=a_2s^2_n\|Q_n\|_2^{-1}Q_n(s_nx)$ if necessary, we may assume that $\|u_n\|^2_2=a^2_1$, $\|u_n\|^3_3=1$ and $\|\nabla u_n\|^2_2=|\mathcal{C}_3|^{-3}(a_1)^{-1}$    $ +o_n(1)$. Similarly, $\|v_n\|^2_2=a^2_2$, $\|v_n\|^3_3=1$ and $\|\nabla v_n\|^2_2=|\mathcal{C}_3|^{-3}(a_2)^{-1}+o_n(1)$.
Then we have
\begin{equation*}
\begin{aligned}
h(s)&=e^{2s}\Big(\|\nabla u_n\|^2_2+\|\nabla v_n\|^2_2-\frac{2\alpha_1\|u_n\|^3_3+2\alpha_2\|v_n\|^3_3}{3}\Big)\\
&
\quad-e^{4s}\Big(\mu_1\|u_n\|^4_4+\mu_2\|v_n\|^4_4+2\beta\|u_nv_n\|^{2}_2\Big)\\
&=e^{2s}\Big(|\mathcal{C}_3|^{-3}(a_1)^{-1}-\frac{2\alpha_1}{3}+|\mathcal{C}_3|^{-3}(a_2)^{-1}-\frac{2\alpha_2}{3}+o_n(1)\Big)\\
&\quad -e^{4s}\Big(\mu_1\|u_n\|^4_4+\mu_2\|v_n\|^4_4+2\beta\|u_nv_n\|^{2}_2\Big).\\
\end{aligned}
\end{equation*}
By Lemma \ref{lem5.1}, there exists a unique $s_{a_1,a_2}\in \R$ such that $h(s_{a_1,a_2})=0$ for $0<\alpha_ia_i<\!\frac{3}{2|\mathcal{C}_3|^{3}}~ (i=1,2)$. Therefore, $s_{a_1,a_2}\star(u_n,v_n)\in \mathcal{P}_{a_1,a_2}$ for $0<\alpha_ia_i<\!\frac{3}{2|\mathcal{C}_3|^{3}}~ (i=1,2)$. Since $\|u_n\|^3_3=1$ and $\|v_n\|^3_3=1$, by H\"{o}lder inequality, $\|u_n\|^4_4\gtrsim1$ and $\|v_n\|^4_4\gtrsim1$. It follows that $s_{a_1,a_2}\to -\infty$ as $(a_1,a_2)\to \big(\frac{3}{2|\mathcal{C}_3|^{3}\alpha_1},\frac{3}{2|\mathcal{C}_3|^{3}\alpha_2}\big)$, which implies
\begin{equation*}
I\big(s_{a_1,a_2}\star(u_n,v_n) \big)=\frac{1}{4}e^{4s_{a_1,a_2}}\big[\mu_1\|u_n\|_{4}^4+\mu_2\|v_n\|^4_4+2\beta\|u_nv_n\|^2_2\big]=o_n(1)
\end{equation*}
as $(a_1,a_2)\to \big(\frac{3}{2|\mathcal{C}_3|^{3}\alpha_1},\frac{3}{2|\mathcal{C}_3|^{3}\alpha_2}\big)$. Thus, we get $m^-(a_1,a_2)\le 0$ for $(a_1,a_2)
=\big(\frac{3}{2|\mathcal{C}_3|^{3}\alpha_1},\frac{3}{2|\mathcal{C}_3|^{3}\alpha_2}\big)$. Moreover,
\begin{equation}\label{n3}
I(u,v)=\frac{1}{4}\big[\mu_1\|u\|_{4}^4+\mu_2\|v\|^4_4+2\beta\|uv\|^2_2\big]\ge 0, \quad \forall (u,v)\in \mathcal{P}_{a_1,a_2}.
\end{equation}
Thus, we get $m^-(a_1,a_2)= 0$ for case 1.
\vskip1mm
\textbf{Case 2 ($\alpha_1a_1> \frac{3}{2|\mathcal{C}_3|^3}$ and $\alpha_2a_2\ge \frac{3}{2|\mathcal{C}_3|^3}$)
or ($\alpha_1a_1\ge \frac{3}{2|\mathcal{C}_3|^3}$ and $\alpha_2a_2> \frac{3}{2|\mathcal{C}_3|^3}$)}.
%\textbf{Case 2 ($\alpha_1a_1> \frac{3}{2|\mathcal{C}_3|^3}$ and $\alpha_2a_2> \frac{3}{2|\mathcal{C}_3|^3}$)}.

We first consider the case $\alpha_1a_1>\frac{3}{2|\mathcal{C}_3|^3}$ and $\alpha_2a_2\ge\frac{3}{2|\mathcal{C}_3|^3}$. For any $(u,v)\in S(a_1)\times S(a_2)$ with $\|\nabla u_n\|_2^2+\|\nabla v_n\|_2^2>\frac{2\alpha_1}{3}\|u_n\|_3^3+\frac{2\alpha_2}{3}\|v_n\|_3^3$, there exists a unique $\tau\in (0,\infty)$ such that $P_{a_1,a_2}(\tau\star(u,v) )=0$. Hence, we have
\begin{equation*}
e^{2\tau}\Big(\|\nabla u\|^2_2+\|\nabla v\|^2_2-\frac{2\alpha_1\|u\|^3_3+2\alpha_2\|v\|^3_3}{3}\Big)
=e^{4\tau}\Big(\mu_1\|u\|^4_4+\mu_2\|v\|^4_4+2\beta\|uv\|^{2}_2\Big),
\end{equation*}
and then
\begin{equation*}
I(\tau\star(u,v))=\frac{1}{4}\Big(\frac{\|\nabla u\|_2^2+\|\nabla v\|_2^2-\frac{2\alpha_1}{3}\|u\|_3^3-\frac{2\alpha_2}{3}\|v\|_3^3}
{[\mu_1\|u\|_{4}^4+\mu_2\|v\|^4_4+2\beta\|uv\|^2_2]^{\frac{1}{2}}}\Big)^{2}.
\end{equation*}
If $(u,v)\in S(a_1)\times S(a_2)$ with $\|\nabla u_n\|_2^2+\|\nabla v_n\|_2^2\le \frac{2\alpha_1}{3}\|u_n\|_3^3+\frac{2\alpha_2}{3}\|v_n\|_3^3$, then there does not exist $\tau\in (0,\infty)$ such that $P_{a_1,a_2}(\tau\star(u,v) )=0$, and $\Psi_{(u,v)}(s)$ is negative and is strictly decreasing on $(0,\infty)$.
To complete the proof it suffices to find
$\{(u_n,v_n)\}\subset S(a_1)\times S(a_2)$ with
\begin{equation}\label{e3.1}
\|\nabla u_n\|_2^2+\|\nabla v_n\|_2^2>\frac{2\alpha_1}{3}\|u_n\|_3^3+\frac{2\alpha_2}{3}\|v_n\|_3^3,
\end{equation}
and
\begin{equation}\label{e3.11}
\frac{\|\nabla u_n\|_2^2+\|\nabla v_n\|_2^2-\frac{2\alpha_1}{3}\|u_n\|_3^3-\frac{2\alpha_2}{3}\|v_n\|_3^3}
{[\mu_1\|u_n\|_{4}^4+\mu_2\|v_n\|^4_4+2\beta\|u_nv_n\|^2_2]^{\frac{1}{2}}}\to 0.
\end{equation}
Define
\begin{equation*}
f(u):=\frac{\|\nabla u\|_2^{2}\|u\|_2}{\|u\|_3^3}.
\end{equation*}
By (\ref{e3.2}), for any $M>\frac{1}{|\mathcal{C}_3|^3}$, there exists $u\in
H^1(\mathbb{R}^4)\setminus \{0\}$ such that $f(u)=M$. For any
$\alpha, \beta>0$, we define $\tilde{u}(x):=\alpha u(\beta x)$. By
direct calculations, we have
\begin{equation*}
\|\tilde{u}\|_2=\alpha \beta^{-2}\|u\|_2,\ \
\|\tilde{u}\|_3=\alpha \beta^{-4/3}\|u\|_3,\ \  \|\nabla
\tilde{u}\|_2=\alpha\beta^{-1}\|\nabla u\|_2
\end{equation*}
and $f(\tilde{u})=M$.
So we can choose
$\alpha=\frac{1}{\|u\|_3}\left(\frac{\|u\|_2}{a_1\|u\|_3}\right)^{2}$
and $\beta=\left(\frac{\|u\|_2}{a_1\|u\|_3}\right)^{3/2}$ such
that $\|\tilde{u}\|_2^2=a^2_1$ and $\|\tilde{u}\|_3=1$.

\vskip1mm
We distinguish the two cases:
\begin{equation*}
\text{either}\ \ (i)\ \  \alpha_2a_2>\frac{3}{2|\mathcal{C}_3|^3} \quad  \text{or}\ \ (ii)\ \ \alpha_2a_2=\frac{3}{2|\mathcal{C}_3|^3}.
\end{equation*}
If $(i)$ holds, for any $N>\frac{1}{|\mathcal{C}_3|^3}$, we choose $\tilde{v}$ such that $f(\tilde{v})=N$, $\|\tilde{v}\|_2^2=a^2_2$ and $\|\tilde{v}\|_3=1$. Under the assumption $\frac{2\alpha_1}{3}a_1>\frac{1}{|\mathcal{C}_3|^3}$ and $\frac{2\alpha_2}{3}a_2>\frac{1}{|\mathcal{C}_3|^3}$.
Thus, there exist $\{A_n\},\{B_n\}\subset \mathbb{R}$ with $A_n, B_n>0$, and
$A_n\to 0$, $B_n\to 0$ as $n\to\infty$ such that
\begin{equation*}
M_n:=(\frac{2\alpha_1}{3}+A_n)
a_1>\frac{1}{|\mathcal{C}_3|^3}\quad \text{and}\quad N_n:=(\frac{2\alpha_2}{3}+B_n)
a_2>\frac{1}{|\mathcal{C}_3|^3}.
\end{equation*}
For such chosen $M_n$, we choose $\{(u_n,v_n)\}\subset H^1(\mathbb{R}^4)\times H^1(\R^4)$
such that $\|u_n\|_2^2=a^2_1$, $\|u_n\|_3=1$ and $f(u_n)=M_n$. Similarly, $\|v_n\|_2^2=a^2_2$, $\|v_n\|_3=1$ and $f(v_n)=N_n$. Then we have
\begin{align*}
&\|\nabla u_n\|_2^2=\big(\frac{2\alpha_1}{3}+A_n\big)\|u_n\|_3^3> \frac{2\alpha_1}{3}\|u_n\|_3^3, \\
&\|\nabla v_n\|_2^2=\big(\frac{2\alpha_2}{3}+B_n\big)\|v_n\|_3^3> \frac{2\alpha_2}{3}\|v_n\|_3^3,
\end{align*}
%\begin{equation*}
%1=\|u_n\|_3\leq
%\|u_n\|_2^{\frac{1}{3}}\|u_n\|_{4}^{\frac{2}{3}}=a^{\frac{1}{6}}_1\|u_n\|_{4}^{\frac{2}{3}}
%\end{equation*}
and
\begin{align*}
\begin{split}
&\frac{\|\nabla u_n\|_2^2+\|\nabla v_n\|_2^2-\frac{2\alpha_1}{3}\|u_n\|_3^3-\frac{2\alpha_2}{3}\|v_n\|_3^3}
{[\mu_1\|u_n\|_{4}^4+\mu_2\|v_n\|^4_4+2\beta\|u_nv_n\|^2_2]^{\frac{1}{2}}}\\
&=\frac{A_n\|u_n\|_3^3+B_n\|v_n\|_3^3}
{\big[\mu_1\|u_n\|_{4}^{4}+\mu_1\|v_n\|_{4}^{4}+2\beta\|u_nv_n\|^2_2\big]^{\frac{1}{2}}}
\to 0
\end{split}
\end{align*}
as $n\to \infty$.

\vskip3mm
If $(ii)$ holds, similar as in case (1), there exists $\bar{v}_n$ such that $\|\bar{v}_n\|^2_2=a^2_2$, $\|\bar{v}_n\|^3_3=1$ and $\|\nabla \bar{v}_n\|^2_2=|\mathcal{C}_3|^{-3}a_2+o_n(1)$. Then we have
\begin{align*}
\begin{split}
&\frac{\|\nabla u_n\|_2^2+\|\nabla \bar{v}_n\|_2^2-\frac{2\alpha_1}{3}\|u_n\|_3^3-\frac{2\alpha_2}{3}\|\bar{v}_n\|_3^3}
{[\mu_1\|u_n\|_{4}^4+\mu_2\|\bar{v}_n\|^4_4+2\beta\|u_n\bar{v}_n\|^2_2]^{\frac{1}{2}}}\\
&=\frac{A_n\|u_n\|_3^3+o_n(1)}
{\big[\mu_1\|u_n\|_{4}^{4}+\mu_1\|v_n\|_{4}^{4}+2\beta\|u_nv_n\|^2_2\big]^{\frac{1}{2}}}
\to 0
\end{split}
\end{align*}
as $n\to \infty$. Thus, we get $m^-(a_1,a_2)= 0$ for case $(2)$.
%\textbf{Case 3 ($0<\alpha_1a_1<\frac{3}{2|\mathcal{C}_3|^3}$ and $\alpha_2a_2\ge \frac{3}{2|\mathcal{C}_3|^3}$)
%or ($\alpha_1a_1\ge \frac{3}{2|\mathcal{C}_3|^3}$ and $0<\alpha_2a_2< \frac{3}{2|\mathcal{C}_3|^3}$)}.
%By \cite[Theorem 1.1]{WuZ} or \cite[Lemma 3.2]{LXF}, we have $m^-(a_1,0)=0$ for $\alpha_1a_1\ge \frac{3}{2|\mathcal{C}_3|^3}$ and $m^-(0,a_2)=0$ for $\alpha_2a_2\ge \frac{3}{2|\mathcal{C}_3|^3}$.
\vskip1mm

Suppose by contradiction that $m^-(a_1,a_2)$ is attained by some $(u^-_{a_1},v^-_{a_2})$ for case $\alpha_1a_1\ge\frac{3}{2|\mathcal{C}_3|^3}$ and $\alpha_2a_2\ge\frac{3}{2|\mathcal{C}_3|^3}$. By \eqref{n3}, then $\|u^-_{a_1}\|^4_4=\|v^-_{a_2}\|^4_4=0$, which is a contradiction with $(u^-_{a_1},v^-_{a_2})\in S(a_1)\times S(a_2)$. Thus, $m^-(a_1,a_2)$ can not be attained for $\alpha_1a_1\ge\frac{3}{2|\mathcal{C}_3|^3}$ and $\alpha_2a_2\ge\frac{3}{2|\mathcal{C}_3|^3}$. The proof is complete.
\end{proof}

%Let $\phi_0$ be a minimizer of Gagliardo-Nirenberg inequality \eqref{b2} with $p=3$. Then, from \cite{WeM}, $\phi_0$ satisfies
%\begin{equation*}
%\begin{cases}
%-\Delta \phi_0+\frac{1}{2}\phi_0=\phi^2_0 \quad \text{in} \quad \R^4,\\
%\phi_0(x)>0 \quad \text{in} \quad \R^4.\\
%\end{cases}
%\end{equation*}
%By the well-known uniqueness result (cf. \cite{KM}) and the scaling invariance of \eqref{g2}, we have $\phi_0(x)=\frac{1}{2}w_3(\frac{x}{\sqrt{2}})$.

\begin{Lem}\label{lem5.5}
Let $p=3$, $\mu_i,\alpha_i,a_i>0(i=1,2)$, there exists a $\beta_0>0$ such that $\beta>\beta_0$ and $\beta\in \big(0,\min\{\mu_1,\mu_2\}\big)\cup\big(\max\{\mu_1,\mu_2\},\infty\big)$.
If $\alpha_ia_i<\frac{3}{2|\mathcal{C}_3|^3}(i=1,2)$, for any ground state $(u^-_{a_1},v^-_{a_2})$ of \eqref{eq1.1}-\eqref{eq1.11}, there exist $\nu_1,\nu_2>0$ such that
\begin{align*}
&\Big(\big(\frac{a_1}{\|w_3\|_2} \big)r^{2}_1u^-_{a_1}(\frac{a_1}{\|w_3\|_2}r_1x), \big(\frac{a_2}{\|w_3\|_2} \big)r^{2}_2v^-_{a_2}(\frac{a_2}{\|w_3\|_2}r_2x)\Big)\\
&\to \big(\nu_1w_3(\sqrt{\nu_1}x), \nu_2w_3(\sqrt{\nu_2}x)\big),
\end{align*}
in $H^1(\R^4,\R^2)$ as $(a_1,a_2)\to (\frac{3}{2\alpha_1|\mathcal{C}_3|^3},\frac{3}{2\alpha_2|\mathcal{C}_3|^3})$ and $1-\frac{2\alpha_1a_1|\mathcal{C}_3|^3}{3}\sim 1-\frac{2\alpha_2a_2|\mathcal{C}_3|^3}{3}$ up to a subsequence, where
$r_1=\big(1-\frac{2\alpha_1a_1|\mathcal{C}_3|^3}{3}\big)^{-\frac{1}{2}}$ and $r_2=\big(1-\frac{2\alpha_2a_2|\mathcal{C}_3|^3}{3}\big)^{-\frac{1}{2}}$.
\end{Lem}

\begin{proof}
We choose $(\varphi_1,\varphi_2)=\big(\frac{a_1}{\|w_3\|_2}w_3,\frac{a_2}{\|w_3\|_2}w_3 \big)\in S(a_1)\times S(a_2)$ as a text function of $m^-(a_1,a_2)$. There exists a unique $\tau\in (0,\infty)$ such that $P_{a_1,a_2}(\tau\star(\varphi_1,\varphi_2) )=0$. Hence, we have
\begin{equation*}
I(\tau\star(\varphi_1,\varphi_2))=\frac{1}{4}\Big(\frac{\|\nabla \varphi_1\|_2^2+\|\nabla \varphi_2\|_2^2-\frac{2\alpha_1}{3}\|\varphi_1\|_3^3-\frac{2\alpha_2}{3}\|\varphi_2\|_3^3}
{[\mu_1\|\varphi_1\|_{4}^4+\mu_2\|\varphi_2\|^4_4+2\beta\|\varphi_1\varphi_2\|^2_2]^{\frac{1}{2}}}\Big)^{2},
\end{equation*}
and then
\begin{equation*}
\begin{aligned}
&m^-(a_1,a_2)\\&\le I(\tau\star(\varphi_1,\varphi_2)) \\
&\le\frac{\max\{a^4_1,a^4_2\}\|\nabla w_3\|^4_2}{4[\mu_1a^4_1+\mu_2a^4_2+2\beta a^2_1a^2_2]\|w_3\|^4_4}
\Big[ \big(1-\frac{2\alpha_1a_1|\mathcal{C}_3|^3}{3}\big)+\big(1-\frac{2\alpha_2a_2|\mathcal{C}_3|^3}{3}\big)\Big]^2.\\
\end{aligned}
\end{equation*}
It follows from $(u^-_{a_1},v^-_{a_2})\in \mathcal{P}_{a_1,a_2}$ and  the Gagliardo-Nirenberg and the Sobolev inequalities that
\begin{equation*}
\big(1-\frac{2\alpha_1a_1|\mathcal{C}_3|^3}{3}\big)+\big(1-\frac{2\alpha_2a_2|\mathcal{C}_3|^3}{3}\big)\lesssim \frac{1}{\mathcal{S}^2_{\mu_1,\mu_2,\beta}}\big(\|\nabla u^-_{a_1}\|^2_2+\|v^-_{a_2}\|^2_2 \big)
\end{equation*}
and
\begin{align*}
&\big(\|\nabla u^-_{a_1}\|^2_2+\|v^-_{a_2}\|^2_2 \big)\\
&\lesssim \frac{\max\{a^4_1,a^4_2\}\|\nabla w_3\|^4_2}{[\mu_1a^4_1+\mu_2a^4_2+2\beta a^2_1a^2_2]\|w_3\|^4_4}
\Big[ \big(1-\frac{2\alpha_1a_1|\mathcal{C}_3|^3}{3}\big)+\big(1-\frac{2\alpha_2a_2|\mathcal{C}_3|^3}{3}\big)\Big].
\end{align*}
By the  Pohozaev identity, we have
\begin{equation*}
\lambda_1a^2_1+\lambda_2a^2_2=\frac{2\alpha_1}{3}\| u^-_{a_1}\|^3_3+\frac{2\alpha_2}{3}\| v^-_{a_2}\|^3_3\sim \Big[\big(1-\frac{2\alpha_1a_1|\mathcal{C}_3|^3}{3}\big)+\big(1-\frac{2\alpha_2a_2|\mathcal{C}_3|^3}{3}\big)\Big].
\end{equation*}
Let $r_1=\big(1-\frac{2\alpha_1a_1|\mathcal{C}_3|^3}{3}\big)^{-\frac{1}{2}}$, $r_2=\big(1-\frac{2\alpha_2a_2|\mathcal{C}_3|^3}{3}\big)^{-\frac{1}{2}}$ and
\begin{equation*}
(u_{r_1},v_{r_2})=\big(\frac{a_1}{\|w_3\|_2}r^2_1u^-_{a_1}(\frac{a_1}{\|w_3\|_2}r_1x),\frac{a_2}{\|w_3\|_2}r^2_2v^-_{a_2}(\frac{a_2}{\|w_3\|_2}r_2x)
\big).
\end{equation*}
By direct calculation, $\{(u_{r_1},v_{r_2})\}$ is bounded and $(u_{r_1},v_{r_2})$ satisfies
\begin{equation*}
\begin{cases}
-\Delta u_{r_1} +\lambda_1r^2_1\frac{a^2_1}{\|w_3\|^2_2}u_{r_1}=\mu_1r^{-2}_1 u^3_{r_1}+
\alpha_1\frac{a_1}{\|w_3\|_2}u^2_{r_1}+\beta\frac{a^2_1}{a^2_2}r^{-4}_2r^{2}_1v^2_{r_2}u_{r_1},\\
-\Delta v_{r_2} +\lambda_2r^2_2\frac{a^2_2}{\|w_3\|^2_2}v_{r_2}=\mu_1r^{-2}_2 v^3_{r_2}+
\alpha_1\frac{a_2}{\|w_3\|_2}v^2_{r_2}+\beta\frac{a^2_2}{a^2_1}r^{-4}_1r^{2}_2u^2_{r_1}v_{r_2}.\\
\end{cases}
\end{equation*}
Similar to the proof in Lemma \ref{lem3.8}, we deduce that $$(u_{r_1},v_{r_2})\to (\nu_1w_3(\sqrt{\nu_1}x), \nu_2w_3(\sqrt{\nu_2}x))$$ strongly in $H^1(\R^4,\R^2)$ as $(a_1,a_2)\to (\frac{3}{2\alpha_1|\mathcal{C}_3|^3},\frac{3}{2\alpha_2|\mathcal{C}_3|^3})$ and $1-\frac{2\alpha_1a_1|\mathcal{C}_3|^3}{3}\sim 1-\frac{2\alpha_2a_2|\mathcal{C}_3|^3}{3}$ up to a subsequence for some $\nu_1,\nu_2>0$.
\end{proof}

\begin{Lem}\label{lem5.6}
Let $p=3$, $\mu_i,\alpha_i>0(i=1,2)$, there exists a $\beta_0>0$ such that $\beta>\beta_0$ and $\beta\in \big(0,\min\{\mu_1,\mu_2\}\big)\cup\big(\max\{\mu_1,\mu_2\},\infty\big)$.
If $\alpha_ia_i<\frac{3}{2|\mathcal{C}_3|^3}(i=1,2)$, for any ground state $(u^-_{a_1},v^-_{a_2})$ of \eqref{eq1.1}-\eqref{eq1.11}, then
\begin{equation*}
I(u^-_{a_1},v^-_{a_2})=m^-(a_1,a_2)\to\frac{k_1+k_2}{4}\mathcal{S}^2\ \ \text{as} \ \ (a_1,a_2)\to (0,0).
\end{equation*}
Moreover, there exists $\varepsilon_{a_1,a_2}>0$ such that
\begin{equation*}
\big(\varepsilon_{a_1,a_2} u^-_{a_1}(\varepsilon_{a_1,a_2} x),\varepsilon_{a_1,a_2} v^-_{a_2}(\varepsilon_{a_1,a_2} x)\big)\to (\sqrt{k_1}U_{\varepsilon_0},\sqrt{k_2}U_{\varepsilon_0})
\end{equation*}
in $D^{1,2}(\R^4,\R^2)$, for some $\varepsilon_0>0$ as $(a_1,a_2)\to (0,0)$ and $a_1\sim a_2$ up to a subsequence, where
$k_1=\frac{\beta-\mu_2}{\beta^2-\mu_1\mu_2}$ and $k_2=\frac{\beta-\mu_1}{\beta^2-\mu_1\mu_2}$.
\end{Lem}
\begin{proof}
The proof is similar to that of Lemma \ref{lem4.4}, thus we focus only on the differences. By $P_{a_1,a_2}(u^-_{a_1},v^-_{a_2})=0$, we have
\begin{equation*}
\begin{aligned}
&\frac{k_1+k_2}{4}\mathcal{S}^2\\
&>I(u^-_{a_1},v^-_{a_2})\\
&=\frac{1}{4}\big(\|\nabla u^-_{a_1}\|^2_2+\|\nabla v^-_{a_2}\|^2_2-\frac{2}{3}\alpha_1\|u^-_{a_1}\|^3_3-\frac{2}{3}\alpha_2\|v^-_{a_2}\|^3_3  \big)\\
&\ge \frac{1}{4}\big(1-\frac{2\alpha_1a_1}{3}|\mathcal{C}_3|^3\big)\|\nabla u^-_{a_1}\|^2_2+\frac{1}{4}\big(1-\frac{2\alpha_2a_2}{3}|\mathcal{C}_3|^3\big)\|\nabla v^-_{a_2}\|^2_2,
\end{aligned}
\end{equation*}
then $\{(u^-_{a_1},v^-_{a_2})\}$ is bounded in $H^1(\R^4,\R^2)$. We claim that $l\neq 0$. Indeed, if $l=0$, by $(u^-_{a_1},v^-_{a_2})\in \mathcal{P}^-_{a_1,a_2}$ and
\begin{align*}
&\|\nabla u^-_{a_1}\|^2_2+\|\nabla v^-_{a_2}\|^2_2\\
&<2\big[
\mu_1\|u^-_{a_1}\|^4_4+\mu_2\|v^-_{a_2}\|^4_4+2\beta\|u^-_{a_1}v^-_{a_2}\|^2_2 \big]+\frac{2\alpha_1}{3}\|u^-_{a_1}\|^3_3+\frac{2\alpha_2}{3}\|v^-_{a_2}\|^3_3,
\end{align*}
we obtain a contradiction:
\begin{align*}
&\big(1-\frac{2\alpha_1a_1}{3}|\mathcal{C}_3|^3\big)\|\nabla u^-_{a_1}\|^2_2+\big(1-\frac{2\alpha_2a_2}{3}|\mathcal{C}_3|^3\big)\|\nabla v^-_{a_2}\|^2_2\\
&< \frac{2}{\mathcal{S}^2_{\mu_1,\mu_2,\beta}}\big[\|\nabla u^-_{a_1}\|^2_2+\|\nabla v^-_{a_2}\|^2_2\big]^2.
\end{align*}
Therefore, $l>0$ and then $I(u^-_{a_1},v^-_{a_2})\to \frac{k_1+k_2}{4}\mathcal{S}^2$ as $(a_1,a_2)\to (0,0)$. The rest of the proof runs as before.
\end{proof}

\noindent \textbf{Proof of Theorem 1.4.}
The proof is finished when we combine Lemmas \ref{lem5.3} and   \ref{lem5.4}.
\qed

%In the following, we study the asymptotic behavior of the minimizers of $I(u,v)\big|_{S(a_1)\times S(a_2)}$ in %$\mathcal{P}_{a_1,a_2}$ as $(a_1,a_2)\to %\big(\frac{3}{2|\mathcal{C}_3|^{3}\alpha_1},\frac{3}{2|\mathcal{C}_3|^{3}\alpha_2}\big)$.
%\begin{Lem}\label{lem5.5}
%Let $p=3$, $\mu_i,\alpha_i>0(i=1,2)$, $0<\beta<\min\{\mu_1,\mu_2\}$ or $\beta>\max\{\mu_1,\mu_2\}$. Let %$(\bar{u},\bar{v})$ be the minimizer of $I(u,v)\big|_{S(a_1)\times S(a_2)}$ in $\mathcal{P}_{a_1,a_2}$ for %$T(a_1,a_2)<\gamma_1$. Then

%\end{Lem}

\section{Mass-supercritical perturbation}

In this section, we fix $3<p<4$, $\alpha_i,\mu_i>0(i=1,2)$ and $\beta>0$. We recall the decomposition of $\mathcal{ P}_{a_1,a_2}=\mathcal{ P}_{a_1,a_2}^+\cup \mathcal{ P}_{a_1,a_2}^0\cup \mathcal{ P}_{a_1,a_2}^-$ (see \eqref{c41}). From the definition of $\mathcal{ P}_{a_1,a_2}^0$, this is $\Psi'_{(u,v)}(0)=\Psi''_{(u,v)}(0)=0$, we have
\begin{equation*}
-(p\gamma_p-2)\gamma_p\big[\alpha_1\|u\|^p_p+\alpha_2\|v\|^p_p\big]=2\big[\mu_1\|u\|^4_4+\mu_2\|v\|^4_4+2\beta\|uv\|^2_2\big],
\end{equation*}
which is not possible. Then $\mathcal{ P}_{a_1,a_2}^0=\emptyset$. By Lemma \ref{lem3.1}, we can also check that $\mathcal{ P}_{a_1,a_2}$ is a smooth manifold of codimension 1 in $S(a_1)\times S(a_2)$.

\begin{Lem}\label{lem6.1}
For all $(u,v)\in S(a_1)\times S(a_2)$, there exists $t_{(u,v)}$ such that $t_{(u,v)}\star(u,v)\in\mathcal{P}_{a_1,a_2}$. $t_{(u,v)}$ is the unique critical point of the function $\Psi_{(u,v)}$ and is a strict maximum point at positive level.
Moreover:
\vskip1mm
\noindent $(1)$ $\mathcal{ P}_{a_1,a_2}=\mathcal{ P}_{a_1,a_2}^-$ and $P_{a_1,a_2}(u,v)<0$ iff $t_{(u,v)}<0$.\\
\noindent $(2)$ $\Psi_{(u,v)}$ is strictly increasing in $(-\infty,t_{(u,v)})$. \\
\noindent $(3)$ The map $(u,v) \mapsto t_{(u,v)} \in \mathbb{R}$ is of class $C^1$.
\end{Lem}
\begin{proof}
The proof is the  same as that of  Lemma 6.1 \cite{Soave2}.
\end{proof}

\begin{Lem} \label{cor6.1}
Assume $3<p<4$, $\alpha_i,\mu_i,a_i>0(i=1,2)$ and $\beta>0$. For any $a_1,a_2>0$, then
\begin{equation*}
m^-(a_1,a_2):=\inf_{(u,v)\in \mathcal{ P}_{a_1,a_2}}I(u,v)>0.\\
\end{equation*}
%$(2)$ There exists $k>0$ sufficiently small such that
%\begin{equation*}
%0<\sup_{(u,v)\in\overline{A_k}} I(u,v)<m^-(a_1,a_2),%\quad \text{and}\quad  (u,v)\in\overline{A_k}\Rightarrow I(u,v), P(u,v)>0,
%\end{equation*}
%where $\overline{A_k}=\big\{(u,v)\in S(a_1)\times S(a_2): \|\nabla u\|^2_2+\|\nabla v\|^2_2<k\big\}$.
\end{Lem}
\begin{proof}
If $(u,v)\in \mathcal{P}_{a_1,a_2}$, then by the Gagliardo-Nirenberg and the Sobolev inequalities, we have
\begin{equation*}
\begin{aligned}
\|\nabla u\|^2_2+\|\nabla v\|^2_2&=\mu_1\|u\|^4_4+\mu_2\|v\|^4_4+2\beta\|uv\|^2_2+\alpha_1\gamma_p\|u\|^p_p+\alpha_2\gamma_p\|v\|^p_p\\
&\le D_1[\|\nabla u\|^2_2+\|\nabla v\|^2_2]^2+(D_2+D_3)[\|\nabla u\|^2_2+\|\nabla v\|^2_2]^{\frac{p\gamma_p}{2}}.\\
\end{aligned}
\end{equation*}
Thus, from the above inequality and $\|\nabla u\|^2_2+\|\nabla v\|^2_2\neq 0$ (since $(u,v)\in S(a_1)\times S(a_2)$), we get that
\begin{equation*}
\inf_{(u,v)\in \mathcal{P}_{a_1,a_2}}\|\nabla u\|^2_2+\|\nabla v\|^2_2\ge C>0.
\end{equation*}
Therefore, from $P_{a_1,a_2}(u,v)=0$, we have
\begin{equation*}
\begin{aligned}
m^-(a_1,a_2)&=\inf_{(u,v)\in \mathcal{P}_{a_1,a_2}}I(u,v)\\
&=\inf_{(u,v)\in \mathcal{P}_{a_1,a_2}}\frac{1}{4}\big[\mu_1\|u\|^4_4+\mu_2\|v\|^4_4+2\beta\|uv\|^2_2\big]\\
&\qquad \qquad\qquad+\frac{p\gamma_p-2}{2p}\big[\alpha_1\|u\|^p_p+\alpha_2\|v\|^p_p\big]\\
&\ge C\inf_{(u,v)\in \mathcal{P}_{a_1,a_2}}\|\nabla u\|^2_2+\|\nabla v\|^2_2>0.
\end{aligned}
\end{equation*}
%$(2)$ By Gagliardo-Nirenberg and the Sobolev inequalities,
%\begin{equation*}
%I(u,v)\ge \frac{1}{2}[\|\nabla u\|^2_2+\|\nabla v\|^2_2]-D_1[\|\nabla u\|^2_2+\|\nabla v\|^2_2]^2-(D_2+D_3)[\|\nabla u\|^2_2+\|\nabla v\|^2_2]^{\frac{p\gamma_p}{2}}>0,\\
%\end{equation*}
%%\begin{equation*}
%%P(u,v)\ge [\|\nabla u\|^2_2+\|\nabla v\|^2_2]-4D_1[\|\nabla u\|^2_2+\|\nabla v\|^2_2]^2-p(D_2+D_3)[\|\nabla %%u\|^2_2+\|\nabla v\|^2_2]^{\frac{p\gamma_p}{2}},\\
%%\end{equation*}
%if $(u,v)\in\overline{A_k}$ for $k$ small enough. We also have $I(u,v)\le \frac{1}{2}[\|\nabla u\|^2_2+\|\nabla v\|^2_2]<m^-(a_1,a_2)$.
\end{proof}

\begin{Lem}\label{lem6.2}
Let $3<p<4$, $\mu_i,\alpha_i,a_i>0(i=1,2)$ and $\beta\in \big(0,\min\{\mu_1,\mu_2\}\big)\cup\big(\max\{\mu_1,\mu_2\},\infty\big)$. For any $a_1,a_2>0$, then
$$0<m^-(a_1,a_2)<\frac{k_1+k_2}{4}\mathcal{S}^2,$$
where $k_1=\frac{\beta-\mu_2}{\beta^2-\mu_1\mu_2}$ and $k_2=\frac{\beta-\mu_1}{\beta^2-\mu_1\mu_2}$.
\end{Lem}
\begin{proof}
The proof is similar to that of Lemma \ref{lem5.2}.
\end{proof}

\begin{Lem}\label{lem6.31}
Let $3<p<4$, $\mu_i,\alpha_i,a_i>0(i=1,2)$, there exists $\beta_1>0$ such that
$$m^-(a_1,a_2)<\min\big\{m^-(a_1,0),m^-(0,a_2)\big\},$$
for any $\beta>\beta_1$, where $m^-(a_1,0)$ and $m^-(0,a_2)$ are defined in \eqref{c12}.
\end{Lem}
\begin{proof}
The proof is similar to that of Lemma \ref{lem5.31}.
\end{proof}

By using the same techniques as that in Lemma \ref{lem5.3}, we can prove the following lemma.
\begin{Lem}\label{lem6.3}
Let $3<p<4$, $\mu_i,\alpha_i>0(i=1,2)$, there exists $\beta_1>0$ such that $\beta>\beta_1$ and $\beta\in \big(0,\min\{\mu_1,\mu_2\}\big)\cup\big(\max\{\mu_1,\mu_2\},\infty\big)$.
For any $a_1,a_2>0$, if
\begin{equation*}
0<m^-(a_1,a_2)<\min\big\{\frac{k_1+k_2}{4}\mathcal{S}^2, m^-(a_1,0),m^-(0,a_2)\big\},
\end{equation*}
where $m^-(a_1,0)$ and $m^-(0,a_2)$ are defined in \eqref{c12}, $k_1=\frac{\beta-\mu_2}{\beta^2-\mu_1\mu_2}$ and $k_2=\frac{\beta-\mu_1}{\beta^2-\mu_1\mu_2}$, then $m^-(a_1,a_2)$ can be achieved by some function $(u^-_{a_1},v^-_{a_2})\in S(a_1)\times S(a_2)$ which is real valued, positive, radially symmetric and radially decreasing.
\end{Lem}

In the following, we study the asymptotic behavior of the minimizers of $$I(u,v)\big|_{S(a_1)\times S(a_2)} \quad\hbox{in} \quad \mathcal{P}^{-}_{a_1,a_2}$$as $(a_1,a_2)\to (+\infty,+\infty)$ and $a_{1}\sim a_{2}$. By direct calculations, we also have the following corollaries. If $3<p<4$, \eqref{g1} also has a unique positive solution $(\lambda,u_{p,\alpha_1})$ (see \eqref{z6}).

\vskip0.2in

\begin{Cor} \label{cor6.2}
Let $a_1>0$, $\alpha_1>0$ and  $p\in(3,4)$. Then, we have
$$0<I_0(u_{p,\alpha_1})=K_{p,\alpha_1} \cdot a_1^{-\frac{4-p}{p-3}},$$
where $K_{p,\alpha_1}:=\frac{p-3}{4-p}\|w_{p}\|^{\frac{p-2}{p-3}}_2\alpha_1^{\frac{1}{3-p}}>0$.
\end{Cor}

\vskip0.2in
\begin{Cor} \label{cor6.3}
Let $a_1,a_2>0$, $\alpha_1,\alpha_2>0$ and  $3<p<4$. Then, we have
$$0<m^-(a_1,a_2)<K_{p,\alpha_1}\cdot a^{-\frac{4-p}{p-3}}_1+K_{p,\alpha_2} \cdot a^{-\frac{4-p}{p-3}}_2$$
where $K_{p,\alpha_1}:=\frac{p-3}{4-p}\|w_{p}\|^{\frac{p-2}{p-3}}_2\alpha^{\frac{1}{3-p}}_1>0$ and $K_{p,\alpha_2}:=\frac{p-3}{4-p}\|w_p\|^{\frac{p-2}{p-3}}_2\alpha^{\frac{1}{3-p}}_2>0$.
\end{Cor}

\vskip0.2in

\begin{Lem}\label{lem6.4}
Assume $3<p<4$, $\mu_i,\alpha_i>0(i=1,2)$, there exists $\beta_1>0$ such that $\beta>\beta_1$ and $\beta\in \big(0,\min\{\mu_1,\mu_2\}\big)\cup\big(\max\{\mu_1,\mu_2\},\infty\big)$. Let $(u^-_{a_1},v^-_{a_2})$ be the minimizer of $I(u,v)\big|_{S(a_1)\times S(a_2)}$ in $\mathcal{P}^-_{a_1,a_2}$ for any $a_1,a_2>0$. Then
\begin{equation*}
\Big(\big(\frac{L_1}{\alpha_1}\big)^{-\frac{1}{p-2}}u^-_{a_1}(L_1^{-\frac{1}{2}}x), \big(\frac{L_2}{\alpha_2}\big)^{-\frac{1}{p-2}}v^-_{a_2}(L_2^{-\frac{1}{2}}x)\Big)\to (w_p,w_p)
 \ \  \text{in}\ H^{1}(\R^4,\R^2),
\end{equation*}
as $(a_1,a_2)\to (+\infty,+\infty)$ and $a_1\sim a_2$, where $w_p$ is defined in \eqref{g2} and
$$L_1=\big(\frac{a^2_1}{\|w_{p}\|^2_2}\alpha_1^{\frac{2}{p-2}} \big)^{-\frac{p-2}{2p-6}};\quad L_2=\big(\frac{a^2_2}{\|w_{p}\|^2_2}\alpha_2^{\frac{2}{p-2}} \big)^{-\frac{p-2}{2p-6}}.$$
%and $\gamma_p=2(p-2)$, $\gamma_q=2(q-2)$.
\end{Lem}

\begin{proof}
Since the proof is similar to that of Lemma \ref{lem3.8}, we only sketch it.
By Lemma $\ref{lem6.3}$, we can suppose that $\big\{(u^-_{a_{1,n}},v^-_{a_{2,n}})\big\}$ is positive and radially symmetric, i.e.,  $0< u^-_{a_{1,n}},v^-_{a_{2,n}}\in H^{1}_r(\R^4)$. Moreover,
\begin{equation*}
0<m^-(a_{1,n},a_{2,n})+o_n(1)=I(u^-_{a_{1,n}},v^-_{a_{2,n}})\le K_{p,\alpha_1} a_{1,n}^{-\frac{4-p}{p-3}} + K_{p,\alpha_2} a_{2,n}^{-\frac{4-p}{p-3}},
\end{equation*}
it follows that $m^-(a_{1,n},a_{2,n})\to 0$ as $(a_{1,n},a_{2,n})\to (+\infty,+\infty)$, and $\|\nabla u^-_{a_{1,n}}\|^2_2+\|\nabla v^-_{a_{2,n}}\|^2_2\to 0$ as $(a_{1,n},a_{2,n})\to (+\infty,+\infty)$.
Since $P_{a_{1,n},a_{2,n}}(u_{a_{1,n}},v_{a_{2,n}})=0$, we get
\begin{equation*}
\begin{aligned}
I(u^-_{a_{1,n}},v^-_{a_{2,n}})&=\big(\frac{1}{2}-\frac{1}{p\gamma_p}\big)\int_{\R^4}|\nabla u^-_{a_{1,n}}|^2+|\nabla v^-_{a_{2,n}}|^2\\
&\quad -\big(\frac{1}{4}-\frac{1}{p\gamma_p}\big)\int_{\R^4}\big(\mu_1|u^-_{a_{1,n}}|^4+\mu_2|v^-_{a_{2,n}}|^4\\
&\quad +2\beta|u^-_{a_{1,n}}|^2|v^-_{a_{2,n}}|^2\big)\\%-\big(\frac{1}{q}-\frac{\gamma_q}{p\gamma_p}\big)\int_{\R^4}\alpha_2|v_{a_{2,k}}|^q\\
&\le K_{p,\alpha_1} a_{1,n}^{-\frac{4-p}{p-3}} + K_{p,\alpha_2} a_{2,n}^{-\frac{4-p}{p-3}}.\\
\end{aligned}
\end{equation*}
Moreover, we use the fact that
\begin{equation*}
\begin{aligned}
&\int_{\R^4}\big(\mu_1|u^-_{a_{1,n}}|^4+\mu_2|v^-_{a_{2,n}}|^4+2\beta|u^-_{a_{1,n}}|^2|v^-_{a_{2,n}}|^2\big)\\
&\le \mathcal{S}^2_{\mu_1,\mu_2, \beta}\Big(\int_{\R^4}|\nabla u^-_{a_{1,n}}|^2+|\nabla v^-_{a_{2,n}}|^2\Big)^2\\
&\le C a_{1,n}^{-\frac{2(4-p)}{p-3}}+C'a_{2,n}^{-\frac{2(4-p)}{p-3}}.
\end{aligned}
\end{equation*}
Then
\begin{equation*}
(1+o_n(1))\big(\alpha_1\gamma_pa_{1,n}^{p-p\gamma_p}|\mathcal{C}_p|^p+\alpha_2\gamma_pa_{2,n}^{p-p\gamma_p}|\mathcal{C}_p|^p\big)^{-\frac{2}{p\gamma_p-2}}\le \|\nabla u^-_{a_{1,n}}\|^2_2+\|\nabla v^-_{a_{2,n}}\|^2_2.
\end{equation*}
There exist $C_i,C'_i>0(i=1,2)$ such that
\begin{equation*}
C_1 a^{-\frac{4-p}{p-3}}_{1,n}+C_2 a^{-\frac{4-p}{p-3}}_{2,n}\le\int_{\R^4}|\nabla u^-_{a_{1,n}}|^2+|\nabla v^-_{a_{2,n}}|^2 \le C'_1 a_{1,n}^{-\frac{4-p}{p-3}}+C'_2 a_{2,n}^{-\frac{4-p}{p-3}}.
\end{equation*}
%The Lagrange multipliers rule implies the existence of some $\lambda_{1,n},\lambda_{2,n} \in \R$ such that
%\begin{equation*}\label{g3}
%\int_{\R^4}\nabla u^-_{a_{1,n}} \nabla\phi+\lambda_{1,n}\int_{\R^4}u^-_{a_{1,n}}\phi=
%\alpha_1\int_{\R^4}|u^-_{a_{1,n}}|^{p-2}u^-_{a_{1,n}}\phi+\mu_1\int_{\R^4}|u^-_{a_{1,n}}|^{3}
%\phi+\beta\int_{\R^4}|u^-_{a_{1,n}}|^2v^-_{a_{2,n}}\phi,\\
%\end{equation*}
%\begin{equation*}
%\int_{\R^4}\nabla v^-_{a_{2,n}} \nabla\psi+\lambda_{2,n}\int_{\R^4}v^-_{a_{2,n}}\psi=
%\alpha_2\int_{\R^4}|v^-_{a_{2,n}}|^{p-2}v^-_{a_{2,n}}\psi+\mu_2\int_{\R^4}|v^-_{a_{2,n}}|^{3}
%\psi+\beta\int_{\R^4}|v^-_{a_{2,n}}|^2u^-_{a_{1,n}}\psi,
%\end{equation*}
%for each $\phi,\psi\in H^{1}(\R^4)$.
In addition, by $P_{a_{1,n},a_{2,n}}(u^-_{a_{1,n}},v^-_{a_{2,n}})=0$,
\begin{equation*}
\begin{aligned}
\lambda_{1,n}a^2_{1,n}+\lambda_{2,n}a^2_{2,n}&=(1-\gamma_p)\alpha_1\|u^-_{a_{1,n}}\|^p_p+(1-\gamma_p)\alpha_2\|v^-_{a_{2,n}}\|^p_p\\
&=(1+o_n(1))\big( C_1a^{-\frac{4-p}{p-3}}_{1,n}+ C'_1a^{-\frac{4-p}{p-3}}_{2,n}\big).\\
\end{aligned}
\end{equation*}
Therefore, there exist $\lambda_{1,n}\sim a^{-\frac{p-2}{p-3}}_{1,n}$ and $\lambda_{2,n}\sim a^{-\frac{p-2}{p-3}}_{2,n}$ as $(a_{1,n},a_{2,n})\to (+\infty,+\infty)$ and $a_{1,n}\sim a_{2,n}$.
Define
\begin{equation*}
\tilde{u}^-_{a_{1,n}}=\frac{1}{\theta_{1,n}}u^-_{a_{1,n}}\big(\frac{1}{\gamma_{1,n}}x\big)\ \ \text{and}\ \ \tilde{v}^-_{a_{2,n}}=\frac{1}{\theta_{2,n}}v^-_{a_{2,n}}\big(\frac{1}{\gamma_{2,n}}x\big),
\end{equation*}
where
\begin{equation*}
\theta_{1,n}=\big(\frac{a^2_{1,n}}{\|w_{p}\|^2_2}\big)^{-\frac{1}{2p-6}}\alpha_1^{-\frac{1}{p-3}}, \qquad \theta_{2,n}=\big(\frac{a^2_{2,n}}{\|w_{p}\|^2_2}\big)^{-\frac{1}{2p-6}}\alpha_2^{-\frac{1}{p-3}},
\end{equation*}
and
\begin{equation*}
\gamma_{1,n}=\big(\frac{a^2_{1,n}}{\|w_{p}\|^2_2}\big)^{-\frac{p-2}{4(p-3)}}\alpha_1^{-\frac{1}{2p-6}},\qquad \gamma_{2,n}=\big(\frac{a^2_{2,n}}{\|w_{p}\|^2_2}\big)^{-\frac{p-2}{4(p-3)}}\alpha_2^{-\frac{1}{2p-6}}.
\end{equation*}
By similar arguments as used in Lemma \ref{lem3.8},
we have $(\tilde{u}^-_{a_{1,n}},\tilde{v}^-_{a_{2,n}})\to (w_p,w_p)$ in $H^1(\R^4,\R^2)$ as $(a_1,a_2)\to (+\infty,+\infty)$ and $a_1\sim a_2$.
\end{proof}

\begin{Lem}\label{lem6.5}
Let $3<p<4$, $\mu_i,\alpha_i,a_i>0(i=1,2)$, there exists a $\beta_1>0$ such that $\beta>\beta_1$ and $\beta\in \big(0,\min\{\mu_1,\mu_2\}\big)\cup\big(\max\{\mu_1,\mu_2\},\infty\big)$.
If $a_1,a_2>0$, for any ground state $(u^-_{a_1},v^-_{a_2})$ of \eqref{eq1.1}-\eqref{eq1.11}, then
\begin{equation*}
I(u^-_{a_1},v^-_{a_2})=m^-(a_1,a_2)\to\frac{k_1+k_2}{4}\mathcal{S}^2\ \ \text{as} \ \ (a_1,a_2)\to (0,0).
\end{equation*}
Moreover, there exists $\varepsilon_{a_1,a_2}>0$ such that
\begin{equation*}
\big(\varepsilon_{a_1,a_2} u^-_{a_1}(\varepsilon_{a_1,a_2} x),\varepsilon_{a_1,a_2} v^-_{a_2}(\varepsilon_{a_1,a_2} x)\big)\to (\sqrt{k_1}U_{\varepsilon_0},\sqrt{k_2}U_{\varepsilon_0})
\end{equation*}
in  $D^{1,2}(\R^4,\R^2)$, for some $\varepsilon_0>0$ as $(a_1,a_2)\to (0,0)$ and $a_1\sim a_2$ up to a subsequence, where $k_1=\frac{\beta-\mu_2}{\beta^2-\mu_1\mu_2}$ and $k_2=\frac{\beta-\mu_1}{\beta^2-\mu_1\mu_2}$.
\end{Lem}
\begin{proof}
The proof is similar to that of Lemma \ref{lem4.4}, thus we focus only on the differences. By $P_{a_1,a_2}(u^-_{a_1},v^-_{a_2})=0$, we have
\begin{equation*}
\begin{aligned}
\frac{k_1+k_2}{4}\mathcal{S}^2>I(u^-_{a_1},v^-_{a_2})&=\frac{1}{4}\big(\mu_1\|u^-_{a_1}\|^4_4+\mu_2\|v^-_{a_2}\|^4_4+2\beta\|u^-_{a_1}v^-_{a_2}\|^2_2\big)
\\
&\quad+\frac{1}{p}(\frac{p\gamma_p}{2}-1)\big(\alpha_1\|u^-_{a_1}\|^p_p+\alpha_2\|v^-_{a_2}\|^p_p\big),\\
\end{aligned}
\end{equation*}
since $p\gamma_p>2$, we deduce that $\{(u^-_{a_1},v^-_{a_2})\}$ is bounded in $L^p(\R^4,\R^2)\cap L^{4}(\R^4,\R^2)$, and hence, since $P_{a_1,a_2}(u^-_{a_1},v^-_{a_2})=0$,
it is also bounded in $H^1(\R^4,\R^2)$. We claim that $l\neq 0$. Indeed, if $l=0$, by $(u^-_{a_1},v^-_{a_2})\in \mathcal{P}^-_{a_1,a_2}$, we see that
\begin{align*}
&2\|\nabla u^-_{a_1}\|^2_2+2\|\nabla v^-_{a_2}\|^2_2\\
&<4\big[
\mu_1\|u^-_{a_1}\|^4_4+\mu_2\|v^-_{a_2}\|^4_4+2\beta\|u^-_{a_1}v^-_{a_2}\|^2_2 \big]+\alpha_1p\gamma^2_p\|u^-_{a_1}\|^p_p+\alpha_2p\gamma^2_p\|v^-_{a_2}\|^p_p,
\end{align*}
we obtain a contradiction:
\begin{equation*}
\begin{aligned}
\|\nabla u^-_{a_1}\|^2_2+\|\nabla v^-_{a_2}\|^2_2&< \frac{2}{\mathcal{S}^2_{\mu_1,\mu_2,\beta}}\big[\|\nabla u^-_{a_1}\|^2_2+\|\nabla v^-_{a_2}\|^2_2\big]^2\\
&\quad+p\gamma^2_p|\mathcal{C}_p|^{p}\big[a^{p-p\gamma_p}_1\|\nabla u^-_{a_1}\|^{p\gamma_p}_2+a^{p-p\gamma_p}_2\|\nabla v^-_{a_2}\|^{p\gamma_p}_2\big].
\end{aligned}
\end{equation*}
Therefore, $l>0$, we obtain that $I(u^-_{a_1},v^-_{a_2})\to \frac{k_1+k_2}{4}\mathcal{S}^2$ as $(a_1,a_2)\to (0,0)$. The rest of the proof runs as before.
\end{proof}

\noindent \textbf{Proof of Theorem 1.5.}
The proof is finished when we combine Lemma \ref{lem6.3} and Lemmas \ref{lem6.4}-\ref{lem6.5}.
%The existence of positive ground state solutions is a direct result of Lemma \ref{lem6.3}. Lemma \ref{lem6.4} and \ref{lem6.5} proves the asymptotic behavior of associated ground states solutions when $(a_1,a_2)\to (+\infty,+\infty)$ and $a_1\sim a_2$, $(a_1,a_2)\to (0,0)$ and $a_1\sim a_2$, respectively.
\qed

%%%%%%%%%%%%%%%%%%%%%%%%%%%%%%%%%%%%%%%%%%%%%%%%%%%%%%%%%%%%%%%%%
\end{CJK*}
\end{document}